\documentclass[reqno, 11pt]{amsart}
\usepackage[T1]{fontenc}
\usepackage{amsthm}
\usepackage{amssymb}
\usepackage{amsmath}
\usepackage{parskip}
\usepackage{tikz}
\usepackage{tikz-cd}
\usepackage{bbm}
\usepackage{enumitem}
\usepackage{xcolor}
\usepackage{parskip}
\usepackage{graphicx}
\usepackage{import}
\usepackage{yhmath}
\usepackage{aliascnt}
\usepackage{standalone}
\usepackage{caption}

\usepackage[colorlinks=true, linkcolor=cyan, anchorcolor=cyan,citecolor=cyan,filecolor=cyan,menucolor=cyan,runcolor=cyan,urlcolor=cyan]{hyperref}

\usepackage{mathtools}

\usepackage[margin=1in,marginparwidth=0.8in, marginparsep=0.1in]{geometry}

\usepackage{zref-clever}
\zcsetup{cap}

\newtheorem{theorem}{Theorem}[section]

\newtheorem{lemma}[theorem]{Lemma}
\AddToHook{env/lemma/begin}{%
   \zcsetup{countertype={theorem=lemma}}}
\zcRefTypeSetup{lemma}{Name-sg=Lemma}

\newtheorem{corollary}[theorem]{Corollary}
\AddToHook{env/corollary/begin}{%
   \zcsetup{countertype={theorem=corollary}}}
\zcRefTypeSetup{corollary}{Name-sg=Corollary}

\newtheorem{proposition}[theorem]{Proposition}
\AddToHook{env/proposition/begin}{%
   \zcsetup{countertype={theorem=proposition}}}
\zcRefTypeSetup{proposition}{Name-sg=Proposition}

\theoremstyle{definition}
\newtheorem{definition}[theorem]{Definition}
\AddToHook{env/definition/begin}{%
   \zcsetup{countertype={theorem=definition}}}
\zcRefTypeSetup{definition}{Name-sg=Definition}

\newtheorem{example}[theorem]{Example}
\AddToHook{env/example/begin}{%
   \zcsetup{countertype={theorem=example}}}
\zcRefTypeSetup{example}{Name-sg=Example}

\newtheorem{remark}[theorem]{Remark}
\AddToHook{env/remark/begin}{%
   \zcsetup{countertype={theorem=remark}}}
\zcRefTypeSetup{remark}{Name-sg=Remark}

\newtheorem{construction}[theorem]{Construction}
\AddToHook{env/construction/begin}{%
   \zcsetup{countertype={theorem=construction}}}
\zcRefTypeSetup{construction}{Name-sg=Construction}

\newtheorem{claim}[theorem]{Claim}
\AddToHook{env/claim/begin}{%
   \zcsetup{countertype={theorem=claim}}}
\zcRefTypeSetup{claim}{Name-sg=Claim}

\newtheorem{procedure}[theorem]{Procedure}
\newtheorem{notation}[theorem]{Notation}

\newtheorem{thmy}{Theorem}

\DeclareRobustCommand{\SkipTocEntry}[5]{}

\newcommand{\eps}{\varepsilon}

\newcommand{\abs}[1]{\left\vert #1 \right\vert}
\newcommand{\norm}[1]{\left\Vert #1 \right\Vert}

\newcommand{\R}{\mathbb{R}}
\newcommand{\C}{\mathbb{C}}
\newcommand{\Z}{\mathbb{Z}}

\newcommand{\D}{\mathbb{D}}

\newcommand{\q}{\mathbf{q}}

\renewcommand{\P}{\mathbb{P}}

\newcommand{\del}{\partial}

\newcommand{\ddt}[2][\vphantom{a}]{ \frac{d #1}{d #2} }

\newcommand{\pdev}[2][\vphantom{a}]{\frac{\partial #1}{\partial #2}}

\newcommand{\dd}{\mathrm{d}}
\newcommand{\DD}{\mathrm{D}}

\title{Weinstein Handlebodies for the Painlevé Betti Spaces}

\author{Joël D. Beimler and William E. Olsen}
\date{\today}

\usepackage[backend=biber, style=alphabetic, sorting=nyt, giveninits=false, maxalphanames=99, maxbibnames=99]{biblatex}

\begin{document}
\begin{abstract}
We prove that the generic fibre of the Betti moduli space associated to any of the ten Painlevé equations coincides with the result of attaching Weinstein handles along the Stokes Legendrian, and provide Weinstein handlebody diagrams for each of them in the process. We moreover extend the procedure for obtaining presentations in Gompf normal form of Weinstein manifolds obtained by attaching handles to Legendrian lifts of cooriented curves in closed surfaces to surfaces with nonempty boundary.
\end{abstract}
\maketitle

\begin{figure}[htbp]
	\centering
	\def\svgwidth{\textwidth}
	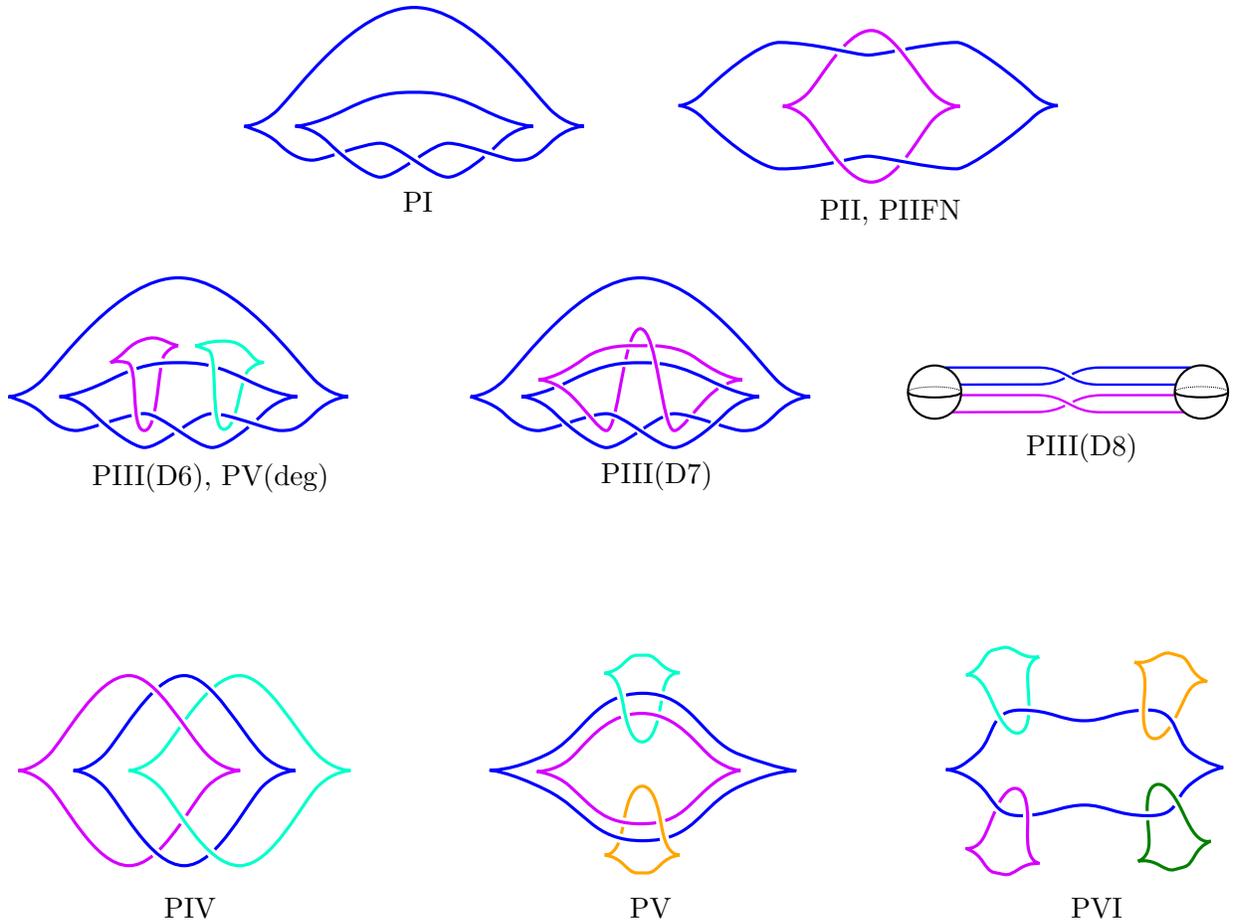
	\caption{Legendrian attaching links for the Painlevé moduli spaces.}
	\label{fig:main_results}
\end{figure}

\newpage

\tableofcontents
\newpage

\section{Introduction}
\label{sec:introduction}
Moduli spaces of local systems on Riemann surfaces carry a variety of rich geometric structures which have been studied from many different points of view. Here, we study certain moduli spaces of rank-two local systems on a Riemann surface $\Sigma$, additionally equipped with Stokes data. Such data encodes their monodromy around a set of marked points $\mathbf{p} \subset \Sigma$, together with data compatible with a type of ``irregular singularity'' at each puncture. 

Specifically, we study moduli spaces associated to the Painlevé equations. There are ten of these equations, and they are customarily indexed by the set
\[
    \mathfrak{P}:= \{ I, \, II, \, II(FN), \, III(D6), \, III(D7), \, III(D8), \, IV, \, V, \, V(deg), \, VI \}.
\]   

Associated to each Painlevé equation is Okamoto's space of initial conditions \cite{okamoto_feuilletages}, which is a moduli space of flat connections on $\P^1$ with singularities of fixed (possibly) irregular type at a set of marked points $\mathbf{p}$, up to gauge equivalence. The moduli spaces we study were introduced to parametrize the data needed to recover Okamoto's spaces under the irregular Riemann-Hilbert correspondence. In this context, van der Put and Saito explicitly describe in \cite{saito_vput} these moduli spaces as families of affine cubic surfaces, depending on some number of parameters. For a generic fixed choice of parameters, one obtains a smooth affine variety of complex dimension $2$.

One can alternatively encode the topological data associated by the Riemann-Hilbert correspondence to an irregular type by a Legendrian link $\Lambda \subset T^\infty \Sigma$ whose front projection to $\Sigma$ encircles each puncture, called the \textbf{Stokes Legendrian}. Serving as our definition for the rest of this paper, the moduli space of monodromies and Stokes data is in this formulation expressed as $\mathbf{M}_{\Lambda}(\Sigma, \mathbf{p})$, the moduli space of sheaves on $\Sigma$ whose microsupport lies in the Stokes Legendrian $\Lambda$, and whose stalks have rank $1$ at $\Lambda$ and vanish at $\mathbf{p}$. See \cite[Section 3.3]{shende2019cluster}, and \cite[Appendix B]{su_betti_dual_boundary} for a careful account.

There is a natural map $\mathbf{M}_\Lambda(\Sigma, \mathbf{p}) \to Loc(\Lambda)$, the microlocal monodromy; denote by $M_\Lambda(\Sigma, \mathbf{p})$ the fibre over a regular value in its image. For the Painlevé spaces, its image can be identified with the parameter spaces of the families exhibited in \cite{saito_vput}. We write $\mathbf{M}_X := \mathbf{M}_{\Lambda_X}(\P^1, \mathbf{p}_X)$ for $X \in \mathfrak{P}$, where the pairs $(\Lambda_X, \mathbf{p}_X)$ are given in \zcref{fig:stokes_legendrians}. Write $M_X$ for a generic fibre of the monodromy, which corresponds to fixing generic parameters picking out a smooth affine cubic surface from the family $\mathbf{M}_X$.

\begin{figure}
	\centering
	 \def\svgwidth{\textwidth}
    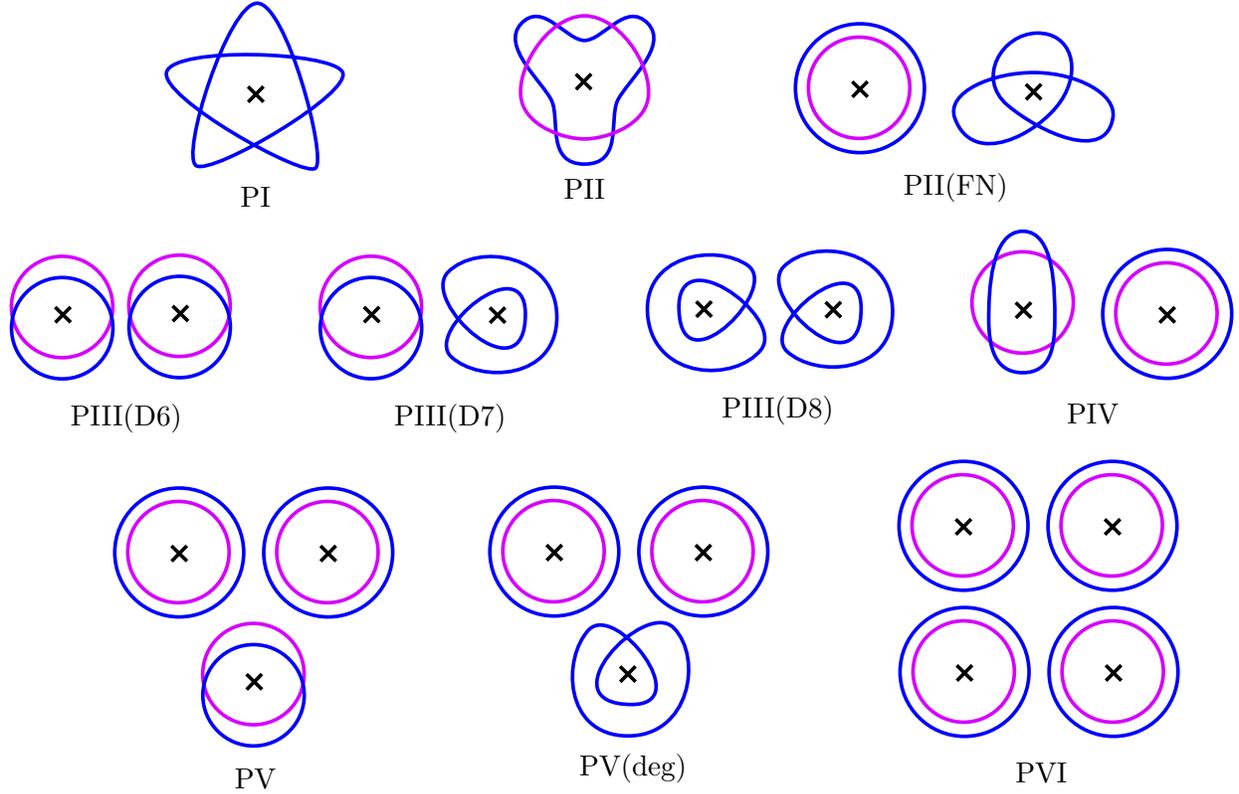
	\caption{Stokes data associated to the Painlevé spaces. Each diagram should be interpreted as lying on $\P^1$, and the coorientation is taken to point away from the puncture.}
	\label{fig:stokes_legendrians}
\end{figure}

Recall that any smooth complex affine variety is an open exact symplectic manifold which can be endowed with a symplectic primitive whose dual vector field (the \textbf{Liouville vector field}) is gradientlike for some proper Morse function, in other words, a \textbf{Weinstein manifold} (see e.g. \cite{ce_stein_weinstein}). In this setup, such a taming function induces a handlebody calculus adapted to the exact symplectic setting \cite{weinstein}, which governs the symplectic topology of the Weinstein manifold. A notable phenomenon is that the handles comprising a Weinstein manifold of dimension $2n$ can have index at most $n$. Connected Weinstein $4$-manifolds can thus be described by the combinatorics of Legendrian links (attaching spheres of $2$-handles) and isotropic $0$-spheres (attaching spheres of the $1$-handles) in the boundary of a small sublevel set of the Morse function.

The next most abundant class of examples of Weinstein manifolds after smooth complex affine varieties are cotangent bundles, and one can obtain even more examples by the following construction. One can describe the boundary of a sublevel set by the unit cotangent bundle $T^\infty Q = \{(q,p) \in T^*Q \mid \norm{p} = 1  \}$ for any choice of metric, and in the four-dimensional case (that is, $Q = \Sigma$ is a surface), one can lift an immersed curve $\gamma \subset \Sigma$, together with a transverse vector field $n \in \Gamma(\gamma^*(T\Sigma/T \gamma))$ (that is, a coorientation of $\gamma$), to a Legendrian 
\[
\Lambda_\gamma = \{ (q,p) \in T^\infty \Sigma \mid q \in \mathrm{im}(\gamma), \, p(T \gamma) = 0, \, p(n) > 0 \}.
\] 

One obtains a new Weinstein manifold by attaching handles along $\Lambda_\gamma$. We will be interested in performing this operation for the surfaces $\P^1 - \mathbf{p}_X$ and the Stokes Legendrians $\Lambda_X$; The reader familiar with Weinstein manifolds will note that one has to slightly modify $T^*\Sigma$ in order to make it into an honest Weinstein manifold, but the modifications are minor and we defer them to \zcref{sec:stokes_lifts}.

We can now state our main theorem.

\begin{theorem}\label{thm:stokes_is_generic_monodromy}
    For any Painlevé type $X \in \mathfrak{P}$, the space obtained by attaching handles along the Stokes Legendrian $\Lambda_X \subset T^\infty(\P^1 - \mathbf{p}_X)$ coincides with $M_X$ up to Weinstein deformation equivalence.
\end{theorem}

\begin{remark}
    Recall from \cite{gps3} that the wrapped Fukaya category of $M_X$ is equivalent to $\mu \mathrm{sh}(\mathrm{Skel}(M_X))$, the category of microsheaves on $\mathrm{Skel}(M_X)$. \zcref{thm:stokes_is_generic_monodromy} provides such skeleta: they are given by $\P^1 - \mathbf{p}_X$, and disks attached along the projections of the Stokes Legendrians from \zcref{fig:stokes_legendrians} (the disks are the cores of the $2$-handles). 

    There is a map 
    \[
    \mu \mathrm{sh}(\mathrm{Skel}(M_X)) \to \mathrm{sh}_{\Lambda_X}(\P^1 - \mathbf{p}_X),
    \]
    induced by restriction; the image of this map consists of sheaves with unipotent monodromy. We expect this fact to have an interpretation in the context of the Betti geometric Langlands conjecture \cite{betti_geom_langlands}.
\end{remark}

We prove \zcref{thm:stokes_is_generic_monodromy} by applying the techniques from \cite{cm_lefs} to the affine cubic surfaces exhibited in \cite{saito_vput} to obtain a handlebody diagram of each of the spaces $M_X$ in Gompf normal form: the outcome is shown in \zcref{fig:main_results}. We then establish in \zcref{prop:lifting_with_boundary} a mild generalization of \cite[Theorem 8.1]{acu_complements}, which may be of independent interest: a procedure to obtain a Weinstein handlebody diagram in Gompf normal form of the Weinstein manifold obtained by attaching handles along the lifts of cooriented curves in a surface with possibly nonempty boundary to its unit cotangent bundle. Applying \zcref{prop:lifting_with_boundary} to the pairs $(\P^1-\mathbf{p}_X, \Lambda_X)$ from \zcref{fig:stokes_legendrians}, we recover the diagrams for the $M_X$ after some simplifications via Legendrian Reidemeister moves.



\addtocontents{toc}{\SkipTocEntry}
\subsection*{Acknowledgements}
We wish to thank Vivek Shende for proposing this project to us; Roger Casals for his interest and helpful discussions; and Emmy Murphy for explaining her result on the Weinstein handlebody of the cotangent bundle of a closed surface. Both authors are supported by Villum Fonden Villum Investigator grant 37814.

\section{Weinstein Manifolds}
\label{sec:Weinstein-manifolds}

\subsection{Generalities}
Weinstein manifolds are a class of exact symplectic manifolds which admit a combinatorial description in terms of handle attachments \cite{weinstein}. Smooth complex affine varieties are naturally Weinstein, and thus in light of \cite{saito_vput}, so are the spaces $M_X$ for $X \in \mathfrak{P}$.

Recall first that any exact symplectic manifold $(X, \omega=\dd\lambda)$ admits a \textbf{Liouville vector field} $Z$, defined by
\[
    \imath_Z \omega = \lambda.
\]

The primitive and the Liouville vector field mutually determine each other.

\begin{definition}
    A \textbf{Weinstein manifold} $(X, \omega, Z)$ is a noncompact exact symplectic manifold $(X, \omega)$, with a choice of primitive whose Liouville vector field $Z$ admitting a proper and bounded-from-below Morse-Bott function $\phi: X \to \R$ which is \textbf{Lyapunov} for $Z$, that is,
    \[
        \dd \phi(Z)\geq \delta(\norm{Z}^2 + \norm{\dd \phi}^2),
    \]
    for some positive $\delta$ and any choice of metric. 
\end{definition}

\begin{remark}
    In light of $Z$ and $\lambda$ determining each other, we will sometimes denote Weinstein manifolds by $(X, \lambda)$ or $(X, \lambda)$. The function $\phi$ is not part of the data specifying a Weinstein manifold, only its existence is required. We may still occasionally refer to a Weinstein manifold as $(X, \phi)$ if the Liouville structure is understood.
\end{remark}

\begin{example}\label{ex:weinstein_on_variety}
    Let $X \subset \C^n$ be a smooth affine variety. Consider the function $\phi_0$ given by
    \[
        \C^n \to \C, \quad \mathbf{z} \mapsto \norm{\mathbf{z}}^2.
    \]
    $\phi_0$ is \textbf{plurisubharmonic}, that is, $\dd \dd^\C \phi_0$ is a symplectic form on $\C^n$. In fact, $\omega_{\phi_0} = -\dd \dd^\C\phi_0$ is the standard symplectic structure on $\C^n$, and $\omega_{\phi_0}(\cdot, J \cdot)$ is the standard metric. Restricting these structures to $X$ yields an exact symplectic form and a compatible metric $g_{\phi_0}$, and it turns out that the gradient vector field $\nabla_{g_{\phi_0}} \phi_0$ is Liouville, making $(X, \omega_{\phi_0}, \nabla_{g_{\phi_0}} \phi_0, \phi_0)$ into a Weinstein manifold \cite{ce_stein_weinstein}.
\end{example}

\begin{example}\label{ex:ample_weinstein}
    Let $\mathcal{L} \to X$ be an ample line bundle on a smooth projective variety $\overline{X}$. Choose a section $s \in H^0(\overline{X}, \mathcal{L})$ and set $D := s^{-1}(0)$. Now choose a metric $\norm{\cdot}$ on $\mathcal{L}$ whose curvature is a Kähler form on $\overline{X}$ (which exists by ampleness of $\mathcal{L}$); the section $s$ gives rise to a plurisubharmonic function on $X := \overline{X}\setminus D$, namely
    \[
        \phi_\mathcal{L}(x) = -\log(\norm{s(x)}^2) : X \to \R.
    \]
    One can show that $\phi_\mathcal{L}$ has finitely many critical points \cite[Lemma 8]{ramanujan}, and that $\nabla_{\phi_{\mathcal{L}}}$ is a Liouville vector field, so that $(X, \omega_{\phi_\mathcal{L}}, \nabla_{g_{\phi_{\mathcal{L}}}}, \phi_{\mathcal{L}})$ is a Weinstein manifold.

    Note that ampleness of $\mathcal{L}$ implies that $X$ is affine; One can show that the two Weinstein structures on $X$ (i.e., one from this construction, and the other from embedding $X \hookrightarrow \C^n$ and restricting the standard structure) are Weinstein homotopic.

\end{example}

\begin{definition} The Lyapunov condition implies that $Z$ is outward-pointing on the sublevel sets $\{\phi \leq c\}$; these sublevel sets are referred to as \textbf{Weinstein domains}.
\end{definition}

In the absence of a Lyapunov function, we call $(X, \omega, Z)$ a \textbf{Liouville manifold} if there exists an exhausting collection of \textbf{Liouville domains}, i.e., compact subdomains with boundary along which $Z$ points outward.

\begin{definition}
    For any Liouville manifold $(X, \dd \lambda, Z)$, the restriction of $\lambda$ to any hypersurface which is transverse to $Z$ is a contact form. In particular, this holds true for the boundary of any Liouville subdomain of $X$. 

    In the Weinstein case, fix $\phi$ to be a Morse function (this can always be arranged) and consider level sets $Y_i = \phi^{-1}(c_i)$ for regular values $c_0, c_1$ such that $[c_0, c_1]$ contains no critical value. The flow of the Liouville vector field then induces a contactomorphism between $Y_0$ and $Y_1$. If $\phi$ has finitely many critical values, there is a well-defined contact manifold $\del_\infty X$, contactomorphic to a large level set of a Lyapunov Morse function. We call $\del_\infty X$ its \textbf{ideal boundary}. Note that there is no canonical contact form on $\del_\infty X$.
\end{definition}

The notion of isomorphism in the category of Weinstein manifolds is that of an exact symplectomorphism, that is, a diffeomorphism $\Phi: (X_0, \dd \lambda_0) \to (X_1, \dd \lambda_1)$ such that $\Phi^*\lambda_1 = \lambda_0 + \dd K$, for some compactly supported function $K$. A family of Weinstein structures $(\omega_t, Z_t, \phi_t)$ on $X$ induces an exact symplectomorphism $X \to X$ (\cite[Proposition 11.8]{ce_stein_weinstein}), and we call such a family a \textbf{Weinstein homotopy}. Note that $\phi_t$ is allowed to exhibit birth-death type critical points. Two Weinstein manifolds which are related by a diffeomorphism $X_0 \to X_1$ so that the pullback structure is Weinstein homotopic to the one on $X_0$ are called \textbf{Weinstein deformation equivalent}.

We will not distinguish between Weinstein domains and manifolds; taking sublevel sets of $\phi$ yields a Weinstein domain, and attaching to a Weinstein domain $(X, \omega = \dd \lambda, Z)$ a piece of the symplectization $(\del X \times (-\eps, \infty], \dd(e^{t}\lambda\vert_{\del X}))$ via the flow of $Z$ yields a Weinstein manifold, denoted $\hat{X}$ and called the \textbf{completion} of $X$. For a Weinstein manifold $X$ such that $\phi$ has finitely many critical values, completing a domain $\{\phi \leq c\}$ for $c$ larger than the largest critical value produces a Weinstein manifold which is deformation equivalent to $X$.

\subsection{Weinstein handlebodies}

If $(X, Z, \phi)$ is such that $\phi$ is Morse (which can always be arranged), then $\phi$ induces a symplectic handle decomposition \cite{weinstein} of $X$; different choices of $\phi$ correspond to different handle decompositions of the same manifold.

A remarkable fact about the handlebodies of Weinstein manifolds $X^{2n}$ is that no handles of index greater than $n$ occur, and these \textbf{critical handles} are the only data that affect the symplectic topology of $X$ in a meaningful way; this is reflected by the fact that Weinstein manifolds without critical handles (called \textbf{subcritical}) are deformation equivalent to a product $X^{2n} \cong F^{2n-2} \times \C$, where $F$ is another (possibly critical) Weinstein manifold \cite{split}.

When giving a handlebody presentation of a Weinstein manifold $X^{2n}$, it is efficient to describe the attaching data to the \textbf{subcritical boundary} of $X_0$, which is the boundary of the union of all subcritical handles. Equivalently, it is the sublevel set $\{\phi \leq c\}$ for some regular value $c$ of a Lyapunov function $\phi$ which is self-indexing, i.e., if $p \in \mathrm{Crit}(\phi)$ is a critical point of index less than $n$, then $\phi(p) < c$, and if the index equals $n$, then $\phi(p) > c$. Given any Lyapunov Morse function $\phi$, \cite[Proposition 12.20]{ce_stein_weinstein} allows one to permute the critical values, so this property can always be satisfied.

We choose to interpret the subcritical boundary as the contact open book $\mathrm{ob}(F_0, \lambda; \mathrm{id})$ whose page is a Weinstein domain $F_0$, the completion of which is such that $\hat{X}_0 \cong \hat{F}_0 \times \C$.

\begin{definition}
    Let $(\Sigma, \lambda)$ be a Liouville domain, and $\phi: \Sigma \to \Sigma$ an exact symplectomorphism preserving the boundary (that is, $\phi^*(\lambda) - \lambda = \dd K$ for some function $K$ vanishing near $\del \Sigma$). The \textbf{abstract open book} $\mathrm{ob}(\Sigma, \lambda; \phi)$ is a closed contact manifold, topologically given by
    \[
        \Sigma \times [0,2\pi]/((x,2\pi) \sim (\phi(x), 0)) \cup_{\del \Sigma \times \del \D^2} \left( \del \Sigma \times \D^2\right).
    \]

    The submanifold $\del \Sigma \times \{0\}$ in the right hand summand of the union is called the \textbf{binding}, and the submanifolds $\Sigma \times \{\eps\}$ are the \textbf{$\eps$-pages}, denoted $\Sigma_\eps$.

    The contact structure is defined by a $1$-form $\alpha$ which is such that the restrictions $(\Sigma_\eps,\alpha\vert_{\Sigma_\eps})$ are Liouville domains which are exact symplectomorphic to $(\Sigma, \lambda)$. Moreover, the Reeb vector field $R_\alpha$ is transverse to the pages, and such that its time-$2\pi$ flow is isotopic to $\phi$ (see, e.g., \cite{geiges_contact_intro}).
\end{definition}

\begin{remark}
    Given an abstract Weinstein Lefschetz fibration  $\mathrm{Lf}(\Sigma, \theta; C_1, \ldots, C_k)$ over $\D$, the boundary of its total space (with corners smoothed) is contactomorphic to the abstract open book $\mathrm{ob}(\Sigma, \theta; \phi)$, where $\phi = \tau_{C_1} \circ \ldots \circ \tau_{C_k}$ is the composition of right-handed Dehn twists along the vanishing cycles.
\end{remark}

    From this description and the fact that subcritical Weinstein manifolds are split, the trivial Lefschetz fibration $\Sigma \times \D^2 \to \D^2$ corresponding to the abstract Weinstein Lefschetz fibration $\mathrm{Lf}(\Sigma, \lambda; \emptyset)$ establishes $\mathrm{ob}(\Sigma, \theta; \mathrm{id})$ as the ideal boundary of the subcritical Weinstein manifold $X \cong \Sigma \times \C$.





Such a presentation of a Weinstein manifold becomes especially effective in dimension four, where it is enough to specify a $2$-dimensional Weinstein domain $(F, \lambda)$ (i.e., $F$ is an orientable surface with boundary), and the attaching sphere, in this setting a Legendrian link in the contact boundary $\mathrm{ob}(F, \lambda; \mathrm{id})$.

This open book is contactomorphic to $\#^k(S^1 \times S^2)$ with its standard contact structure, where $k = \mathrm{rk}(H_1(F))$. The framing for critical Weinstein handle attachment is unique; For each link component $\Lambda$, it is given by $tb(\Lambda)-1$, where $tb(\lambda)$ is the Thurston-Bennequin invariant of $\Lambda$ \cite{gompf_handlebody}.

The contact structure on $\#^k(S^1 \times S^2)$ can be obtained from the standard contact structure on $\R^3$, considered as $(\mathbb{S}^3, \xi_{\mathrm{std}})$ with one point removed, by performing $k$ $0$-framed contact surgeries along the attaching $0$-spheres of the $1$-handles. We can draw the effect of such a surgery on the standard front projection on $\R^3$ by adding two disjoint $3$-balls representing the attaching region of the $1$-handle, which act as wormholes for the attaching link of any $2$-handle.

\subsection{Handlebody calculus}
The  local modifications of a front diagram displayed in \zcref{fig:reidemeister_gompf_moves} are induced by an isotopy $\phi_t$ of the Lyapunov function, and hence give rise to Weinstein homotopic Weinstein manifolds \cite{gompf_handlebody}. These are referred to as \textbf{Legendrian Reidemeister moves} and \textbf{Gompf moves}.

\begin{figure}[htbp]
    \centering
    \def\svgwidth{\textwidth}
    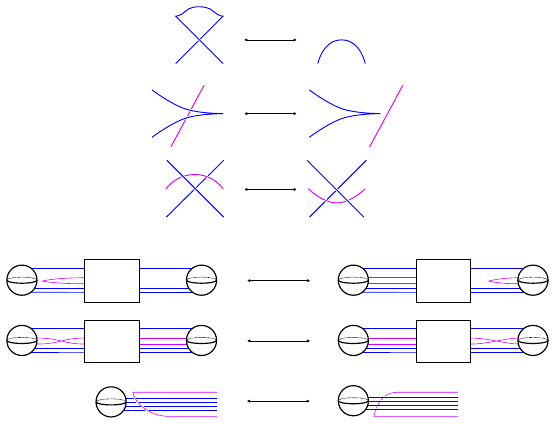
    \caption{Legendrian Reidemeister and Gompf moves}
    \label{fig:reidemeister_gompf_moves}
\end{figure}

There are two additional moves (handle slides and cancellations) originating from the perspective of interpreting these handlebody diagrams as surgery diagrams for how to obtain the contact manifold $\del_\infty X$ by starting from $(S^3, \xi_{\mathrm{std}})$.

In general, contact-$k$ surgeries are defined for any integer $k$, and any $k$-surgery can be expressed as a sequence of $(+1)$- and $(-1)$-surgeries \cite{ding_geiges_surgery}. From this viewpoint, attaching a subcritical handle to $\del_\infty X$ has the effect of performing a contact-$0$ surgery along an $S^0$; Attaching a critical Weinstein handle corresponds to (due to the unique $tb-1$-framing) performing a $(-1)$-surgery along the Legendrian attaching sphere.

To define handle slides, suppose $\Lambda$ and $\Sigma$ are Legendrians in a contact manifold such that $\Sigma$ is diffeomorphic to a sphere, and such that there is a Reeb chord from $\Lambda$ to $\Sigma$ for some choice of contact form. Flowing $\Lambda$ along the Reeb vector field eventually creates a singular Legendrian when the length of the Reeb chord is zero, but in the contact manifold obtained by surgery along $\Sigma$, one can push $\Lambda$ \emph{past} $\Sigma$. This defines a Legendrian isotopy in the contact manifold obtained by $(\pm 1)$-surgery along $\Sigma$.

In practice, $\Lambda$ and $\Sigma$ are often the attaching spheres of two different $2$-handles. Denote by $Y_\Lambda(k)$ the result of performing contact-$k$ surgery along $\Lambda$, and, given two Legendrian attaching spheres $\Sigma$ and $\Lambda$, denote by $h_\Sigma(\Lambda)$ the Legendrian obtained by sliding $\Lambda$ along $\Sigma$. Handle slides become useful in Kirby calculus once one knows how to draw them in a front projection.

\begin{proposition}[{\cite[Proposition 2.15]{cm_lefs}}]\label{handle_slides}
    Let $(Y, \xi)$ be a contact $3$-fold, and $\Lambda, \Sigma \subset (Y, \xi)$ two disjoint Legendrian submanifolds such that $\Sigma$ is a sphere. We have that
    \begin{enumerate}
        \item The Legendrians $\Lambda$ and $h_\Sigma(\Lambda)$ in the left of \zcref{fig:slides} are Legendrian isotopic in the surgered contact manifold $Y_\Sigma(-1)$.
        \item The Legendrians $\Sigma$ and $h_\Lambda(\Sigma)$ in the right of \zcref{fig:slides} are Legendrian isotopic in the surgered contact manifold $Y_\Lambda(+1)$.
    \end{enumerate}
\end{proposition}

\begin{figure}[htbp]
    \centering
    \def\svgwidth{\textwidth}
\begingroup%
  \makeatletter%
  \providecommand\color[2][]{%
    \errmessage{(Inkscape) Color is used for the text in Inkscape, but the package 'color.sty' is not loaded}%
    \renewcommand\color[2][]{}%
  }%
  \providecommand\transparent[1]{%
    \errmessage{(Inkscape) Transparency is used (non-zero) for the text in Inkscape, but the package 'transparent.sty' is not loaded}%
    \renewcommand\transparent[1]{}%
  }%
  \providecommand\rotatebox[2]{#2}%
  \newcommand*\fsize{\dimexpr\f@size pt\relax}%
  \newcommand*\lineheight[1]{\fontsize{\fsize}{#1\fsize}\selectfont}%
  \ifx\svgwidth\undefined%
    \setlength{\unitlength}{215.4557572bp}%
    \ifx\svgscale\undefined%
      \relax%
    \else%
      \setlength{\unitlength}{\unitlength * \real{\svgscale}}%
    \fi%
  \else%
    \setlength{\unitlength}{\svgwidth}%
  \fi%
  \global\let\svgwidth\undefined%
  \global\let\svgscale\undefined%
  \makeatother%
  \begin{picture}(1,0.29082463)%
    \lineheight{1}%
    \setlength\tabcolsep{0pt}%
    \put(0,0){\includegraphics[width=\unitlength,page=1]{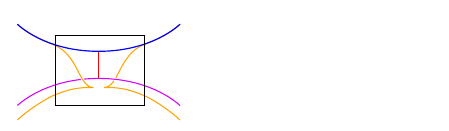}}%
    \put(0.82420925,0.09553439){\color[rgb]{0.83921569,0,1}\makebox(0,0)[lt]{\lineheight{1.25}\smash{\begin{tabular}[t]{l}$\Sigma (+1)$\end{tabular}}}}%
    \put(0.80434583,0.22754361){\color[rgb]{0,0,1}\makebox(0,0)[lt]{\lineheight{1.25}\smash{\begin{tabular}[t]{l}$\Lambda$\end{tabular}}}}%
    \put(0.82401584,0.13914827){\color[rgb]{1,0.64705882,0}\makebox(0,0)[lt]{\lineheight{1.25}\smash{\begin{tabular}[t]{l}$h_\Sigma(\Lambda)$\end{tabular}}}}%
    \put(0,0){\includegraphics[width=\unitlength,page=2]{surgery_slides.pdf}}%
    \put(0.34042553,0.10000091){\color[rgb]{0.83921569,0,1}\makebox(0,0)[lt]{\lineheight{1.25}\smash{\begin{tabular}[t]{l}$\Sigma (-1)$\end{tabular}}}}%
    \put(0.33310395,0.22169908){\color[rgb]{0,0,1}\makebox(0,0)[lt]{\lineheight{1.25}\smash{\begin{tabular}[t]{l}$\Lambda$\end{tabular}}}}%
    \put(0.34109855,0.14037984){\color[rgb]{1,0.64705882,0}\makebox(0,0)[lt]{\lineheight{1.25}\smash{\begin{tabular}[t]{l}$h_\Sigma(\Lambda)$\end{tabular}}}}%
  \end{picture}%
\endgroup%

    \caption{Handle slides over a $(\pm 1)$-framed Legendrian sphere $\Sigma$ in the front, using the Reeb chord in red. We reiterate that for Weinstein handlebody diagrams, the surgery coefficient is always $-1$; therefore, the diagram on the left describes a move which may always be applied to any surgery diagram of a Weinstein manifold. The move on the right will become useful in \zcref{dehn_fronts}.}
    \label{fig:slides}
\end{figure}

\begin{remark}
    \zcref{fig:slides} takes place in the front projection of a contact Darboux ball centred around the Reeb chord in red (we can assume the Legendrian isotopy $\Lambda_t$ has been applied for $t$ large enough so that the depicted region does indeed lie inside the Darboux ball). Outside of this neighbourhood, the Legendrian $h_\Lambda(\Sigma)$ coincides with $\Lambda$ or a small Reeb-pushoff $R_{\pm}(\Sigma)$ of $\Sigma$; the direction is that of the isotopy. In this picture, we slide blue downwards over purple, and so $h_\Sigma(\Lambda)$ lies \emph{below} $\Sigma$.
\end{remark}

The last move is called a \textbf{handle cancellation}; a $1$-handle $h_1$ and a $2$-handle $h_2$ can be attached or removed simultaneously, leaving the Weinstein manifold invariant, given that the attaching sphere of $h_2$ intersects the belt sphere of $h_1$ transversely in a single point \cite{ding_geiges_surgery}. Diagrammatically, we represent this move as depicted in \zcref{fig:cancellation}.

\begin{figure}[htbp]
    \centering
    \def\svgwidth{\textwidth}
\begingroup%
  \makeatletter%
  \providecommand\color[2][]{%
    \errmessage{(Inkscape) Color is used for the text in Inkscape, but the package 'color.sty' is not loaded}%
    \renewcommand\color[2][]{}%
  }%
  \providecommand\transparent[1]{%
    \errmessage{(Inkscape) Transparency is used (non-zero) for the text in Inkscape, but the package 'transparent.sty' is not loaded}%
    \renewcommand\transparent[1]{}%
  }%
  \providecommand\rotatebox[2]{#2}%
  \newcommand*\fsize{\dimexpr\f@size pt\relax}%
  \newcommand*\lineheight[1]{\fontsize{\fsize}{#1\fsize}\selectfont}%
  \ifx\svgwidth\undefined%
    \setlength{\unitlength}{190.79999308bp}%
    \ifx\svgscale\undefined%
      \relax%
    \else%
      \setlength{\unitlength}{\unitlength * \real{\svgscale}}%
    \fi%
  \else%
    \setlength{\unitlength}{\svgwidth}%
  \fi%
  \global\let\svgwidth\undefined%
  \global\let\svgscale\undefined%
  \makeatother%
  \begin{picture}(1,0.14029119)%
    \lineheight{1}%
    \setlength\tabcolsep{0pt}%
    \put(0,0){\includegraphics[width=\unitlength,page=1]{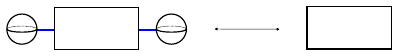}}%
    \put(0.86472729,0.06095345){\color[rgb]{0,0,0}\makebox(0,0)[lt]{\lineheight{1.25}\smash{\begin{tabular}[t]{l}$\Lambda'$\end{tabular}}}}%
    \put(0.24327554,0.06496869){\color[rgb]{0,0,0}\makebox(0,0)[lt]{\lineheight{1.25}\smash{\begin{tabular}[t]{l}$\Lambda$\end{tabular}}}}%
  \end{picture}%
\endgroup%

    \caption{Cancelling the blue $2$-handle with a $1$-handle; $\Lambda$ refers to a Legendrian tangle, possibly involving further $1$-handles and the blue attaching sphere, and $\Lambda'$ is the Legendrian tangle obtained from $\Lambda$ by removing the blue attaching sphere.
    }
    \label{fig:cancellation}
\end{figure}

\section{Symplectic Fibrations}
\label{sec:symplectic-fibrations}

We recall some background on symplectic and Lefschetz fibrations.

\begin{definition}
    Let $\Sigma \hookrightarrow X \xrightarrow{w} B$ be a smooth fibre bundle. Write $\Sigma_z = w^{-1}(z)$. We say that $w$ is a \textbf{symplectic fibration} if there exists a closed $2$-form $\Omega$ on $X$ such that $\omega_z := \Omega\vert_{\Sigma_z}$ is a symplectic form on $\Sigma_z$ for every $z \in B$. We say that such $\Omega$ are \textbf{compatible} with $w$.
\end{definition}

\begin{definition}[Symplectic connection]
    Every symplectic fibration carries a canonical connection whose horizontal spaces are defined by
    \[
        \mathcal{H}_p := (\ker\DD w(p))^{\Omega_p} \subset T_p X.
    \]
\end{definition}

The flow of the horizontal vector fields is a symplectomorphism of the fibres, whenever it is defined. It will be defined for all time if $X$, and hence $B$, is closed. In other situations (in particular, in any situation of interest in this article; $X$ compact with corners, or $X$ and $B$ noncompact Liouville manifolds), one faces the usual problem of integrability of vector fields.

\begin{definition}[Boundary components]
    Assume $\Sigma \hookrightarrow X \xrightarrow{w} B$ is a fibre bundle such that $B$ has nonempty boundary. The \textbf{horizontal boundary} is
    \[
        \del_h X := \bigcup_{z \in B} \del \Sigma_z,
    \]
    and the \textbf{vertical boundary} is given by
    \[
        \del_v X := w^{-1}(\del B).
    \]
    Observe that $\del X = \del_h X \cup \del_v X$, and that both components meet in a codimension-$2$ corner.
\end{definition}

Relevant for us are fibrations on Liouville manifolds which are symplectic away from simple singularities, i.e., Lefschetz fibrations on Liouville domains and manifolds. We follow the exposition in \cite[Section 2.2]{mclean_symp_homology}.

\begin{definition}
    An \textbf{exact symplectic Lefschetz fibration} $(X, \Omega, w)$ comprises

    \begin{itemize}
        \item a compact manifold-with-corners $X$ whose boundary is the union of two faces $\del_h X$ and $\del_v X$ meeting in a codimension-$2$ corner;
        \item an exact $2$-form $\Omega = \dd \Theta$ on $X$;
        \item a surface $S$ with nonempty boundary;
        \item and a proper map $w: X \to S$ compatible with $\Omega$. Write $\Sigma$ for the generic fibre.
    \end{itemize}
    These data are required to satisfy
    \begin{enumerate}
        \item The map $w$ is such that $\del_v X = w^{-1}(\del S)$, and $w \vert_{\del_v X}$ is a smooth fibre bundle. There is moreover a neighbourhood $\mathcal{N}(\del_h X)$ such that $w \vert_{\mathcal{N}(\del_h X)}: \del_h X \to  S$ is a product fibration $\mathcal{N}(\Sigma) \times S$, such that $\Omega_{\mathcal{N}(\del_h X)}$ and $\Theta_{\mathcal{N}(\del_h X)}$ are pullbacks from the first factor.
        \item There exists an integrable complex structure $J$ near $\mathrm{Crit}(w) \subset X$ and an integrable complex structure $j$ near $\mathrm{Critval}(w) \subset S$ such that $w$ is $(J,j)$-holomorphic on the locus where the complex structures are defined. Moreover, at any critical point, the complex Hessian $D^2 w$ is nondegenerate as a complex quadratic form. We impose that there may be at most one critical point in any fibre.
        \item $\Omega$ is Kähler for $J$ near $\mathrm{Crit}(w)$.
    \end{enumerate}

\end{definition}

To define Lefschetz fibrations on noncompact manifolds, one has to replace the triviality conditions near $\del_v X$ and $\del_h X$ with triviality conditions ``at infinity'', which ensure integrability of the parallel transport maps. For such a Lefschetz fibration $(X, \Omega, w)$, the total space $X$, the base $S$, and the fibre $\Sigma$ are open manifolds without boundary. One imposes
\begin{itemize}
    \item \textbf{Horizontal triviality:} There exists an open set $U \subset X$ such that, for all fibres $X_z$, $(X \setminus U) \cap X_z$ is compact; $w \vert_U: U \to S$ is a trivial product fibration $(U\cap \Sigma) \times S \to S$, with $\Omega\vert_U$ and $\Theta\vert_U$ being pullbacks from the first factor;
    \item \textbf{Vertical triviality:} There exists a compact set $K \subset S$ containing all critical values of $w$, such that on $w^{-1}(S \setminus K)$, $w: w^{-1}(S \setminus K) \to S \setminus K$ is a smooth fibre bundle.
\end{itemize}
$w$ is evidently also no longer proper. We impose that $\mathrm{Crit}(w)$ be finite.

Note that in both cases, $w: X \setminus \mathrm{Crit}(w) \to S \setminus \mathrm{Critval}(w)$ is a symplectic fibration, and that parallel transport is well-defined for all time. The regular fibres are therefore exact symplectomorphic as exact symplectic manifolds with boundary.

\begin{definition}
    An exact symplectic Lefschetz fibration $(X, \Omega, w)$ (for $X$ compact) is called a \textbf{Liouville Lefschetz fibration} if, for all $z \in S$, the vector field dual to $\Theta\vert_{X_z}$ is outward-pointing on the boundary $\del X_z$.
\end{definition}

Note that this implies that the regular fibres are Liouville domains. \cite[Section 2.2]{mclean_symp_homology} also explains how to complete Liouville Lefschetz fibrations $(X, \Omega, w)$ to a Lefschetz fibration on a noncompact manifold $(\hat{X}, \hat{\Omega}, w)$, such that the general fibre $(\hat{\Sigma}, \hat{\Theta}\vert_{\hat{\Sigma}})$ is the completion of $(\Sigma, \Theta\vert_{\Sigma})$.

\begin{remark}
    From \cite[Theorem 2.15]{mclean_symp_homology}, it follows that, given any Liouville Lefschetz fibration $(X, \Omega, w)$, there is a smooth family $\Omega_t$ starting at $\Omega$ such that $(X, \Omega_t, w)$ is a Liouville Lefschetz fibration for all $t$, and  that moreover $(X, \Omega_1)$ is an exact symplectic manifold such that smoothing corners yields a Liouville domain. In the following, we will always assume this transformation has been performed, so that the total spaces of our exact symplectic Lefschetz fibrations are considered to be Liouville manifolds or domains.
\end{remark}



We now specify to exact symplectic Lefschetz fibrations over $S = \D \subset \C$, and their respective completions to $\C$. We start by reviewing properties and constructions related to Lefschetz fibrations.

\begin{definition}
    Let $w: X \to \D$ be an exact symplectic Lefschetz fibration and let $\gamma: [0,1] \to \D$ be a path. We call $\gamma$ a \textbf{vanishing path} if $\gamma(1)$ is a critical value of $w$, and $\gamma([0,1))$ avoids the critical values. Denote the parallel transport maps between fibres over $\gamma$  by $\rho_\gamma^t$ for $t \in [0,1)$.

    Denote the critical point in $X_{\gamma(1)}$ by $x_0$. To $\gamma$, we associate the \textbf{vanishing cycle}
    \[
        C_\gamma := \{ x \in X_{\gamma(0)} \mid \lim_{t \to 1} \rho_\gamma^t(x) = x_0 \},
    \]

    as well as the \textbf{vanishing thimble}
    \[
        T_\gamma :=  \bigcup_{s \in [0,1)}\{x \in X_{\gamma(s)} \mid \lim_{t \to 1} \rho^t_{\gamma\vert_{[s,1]}}(x) = x_0 \} \cup \{x_0\}.
    \]
\end{definition}

\begin{remark}
    Let the dimension of $X$ be $2n$. The vanishing cycles and thimbles enjoy the following properties \cite{seidelbook}:
    \begin{itemize}
        \item $C_\gamma$ is an exact Lagrangian $S^{n-1}$ in $X_{\gamma(0)}$.
        \item $T_\gamma$ is an exact Lagrangian $\D^{n}$ in $X$ whose boundary is $C_\gamma$.
    \end{itemize}
\end{remark}

\begin{definition}
    To efficiently organize the vanishing cycles into a single fibre, one chooses a \textbf{distinguished basis of vanishing paths} which is a collection of vanishing paths $\gamma_1, \ldots, \gamma_k$, one for each critical value, such that, for all $i,j$,
    \begin{itemize}
        \item $\gamma_i([0,1)) \cap \mathrm{Critval}(w) = \emptyset$,
        \item $\gamma_i((0,1]) \cap \gamma_j((0,1]) = \emptyset$ for $i \neq j$, and
        \item $\gamma_i(0) = *$ for some common reference point $* \in \C$.
    \end{itemize}
    Moreover, we require the tangent vectors $\dot{\gamma}_i(0)$ to be distinct. This condition induces a cyclic order (in our convention, counterclockwise) on the vanishing paths, which we  apply to the collection of vanishing cycles $\mathcal{C} = (C_{\gamma_1}, \ldots, C_{\gamma_k})$.
\end{definition}

In fact, if we start with a \emph{Weinstein} domain $(\Sigma, \theta)$ and a cyclically ordered collection of Lagrangian spheres $C_1, \ldots, C_k$ in $\Sigma$, one can construct a Liouville Lefschetz fibration with the given data.

\begin{proposition}[{\cite[Proposition 16.8]{seidelbook}, \cite[Definition 6.3]{giroux_pardon}}]\label{abstract_lefschetz}
    Given a Weinstein domain $(\Sigma, \theta)$ and a cyclically ordered collection of exact Lagrangian spheres $C_1, \ldots, C_k$ in $\Sigma$, there exists a Liouville Lefschetz fibration $\Sigma \hookrightarrow X \xrightarrow{w} \C$, whose generic fibre is the Weinstein domain $(\Sigma, \theta)$, and whose vanishing cycles are given by the $C_i$ with the given ordering. The total space $X$ with smoothed corners is a Weinstein domain unique up to Weinstein deformation equivalence.

    Moreover, any Weinstein domain is deformation equivalent to the total space of an abstract Weinstein Lefschetz fibration \cite[Theorem 1.10]{giroux_pardon}.
\end{proposition}

\begin{remark}
    The total space $X$ is given by attaching to $\Sigma \times \D^2$ Weinstein handles along the Legendrian lifts $\Lambda_i$ to $\mathrm{ob}(\Sigma; \mathrm{id})$ of the vanishing cycles $C_i$ given by
    \[
        \Lambda_i = \{ [(x, f_i(x))] \in \Sigma \times S^1 \subset \mathrm{ob}(\Sigma; \mathrm{id}) \mid x \in C_i \};
    \]
    The function $f_i: C_i \to \R$ is a primitive of $\theta\vert_{C_i}$, that is, such that $\theta\vert_{C_i} = \dd f_i$, and such that the images $f_j(C_j) \subset S^1$ are disjoint and respect the ordering of $(C_1, \ldots, C_k)$ when traversing $S^1$ counterclockwise. We refer to \cite[Definition 6.3]{giroux_pardon} for more details.
\end{remark}

Abstract Weinstein Lefschetz fibrations can naturally be completed to Lefschetz fibrations on the completed Weinstein manifold. We will not distinguish between Leschetz fibrations on domains and their completions in the rest of the paper.

It is known moreover that the following operations on abstract Weinstein Lefschetz fibrations induce deformation equivalences of their total spaces:

\begin{itemize}
    \item \textbf{Deformation:} Simultaneous Weinstein deformation of $(\Sigma, \theta)$ and exact Lagrangian isotopy of $(C_1, \ldots, C_k)$.
    \item \textbf{Cyclic permutation:} Exchange $(C_1, \ldots, C_k)$ with $(C_k, C_1, \ldots, C_{k-1})$.
    \item \textbf{Hurwitz moves:} Exchange $\mathbf{C}:= (C_1, \ldots, C_{i-1}, C_i, \ldots, C_k)$ with either
          \begin{align*}
              \mathcal{H}_i(\mathbf{C})&:= (C_1, \ldots, C_{i-2}, C_{i}, \tau_{C_{i}}(C_{i-1}), C_{i+1}, \ldots, C_k)\\
              \text{or} \quad \mathcal{H}_i^{-1}(\mathbf{C})&:=(C_1, \ldots, C_{i-2}, \tau^{-1}_{C_{i-1}}(C_{i}), C_{i-1}, C_{i+1}, \ldots, C_k).
          \end{align*}
\end{itemize}

\begin{remark}
    The indexation for the Hurwitz moves depends on one's convention of whether one should attach handles along the $C_k$ in clockwise or counterclockwise direction; The convention employed in this paper is that they should be attached counterclockwise.
\end{remark}

When applying \zcref{recipe_3_3}, we will make intensive use of these moves, and moreover one more move on vanishing paths leaving the corresponding thimble invariant.

\begin{proposition}[V-moves, {\cite{cm_lefs}}]
    Consider the blue and purple vanishing paths in \zcref{fig:vmove}, and denote them by $\gamma_0$ and $\gamma_1$. Assume that the vanishing cycles $C_{\gamma_0}, C_{\gamma_1}$ intersect transversely in a single point in their common fibre. Then the vanishing thimbles $T_{\gamma_0}$ and  $T_{\gamma_1}$ are Lagrangian isotopic, and the isotopy is supported in the preimage of a small disk containing the two critical values.
\end{proposition}

\begin{figure}[htbp]
    \centering
    \def\svgwidth{0.3\textwidth}
    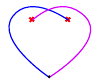
    \caption{A V-move in the complex plane. Critical values are indicated with red crosses; the thimbles over the blue and purple vanishing paths are Lagrangian isotopic in the total space.}
    \label{fig:vmove}
\end{figure}



Liouville Lefschetz fibrations hence give rise to a natural source of exact Lagrangian spheres in the fibres, as well as exact Lagrangian balls in the total space. The following construction produces Lagrangian spheres in the total space.

\begin{definition}[Matching cycles]
    Let $w: X \to \C$ be a Liouville Lefschetz fibration, and let $\gamma_0, \gamma_1$ be two vanishing paths to critical values $y_0, y_1 \in \C$ starting at the same point, which are moreover such that the concatenation $\mu := \overline{\gamma_0}\#\gamma_1$ is smooth. Assume that $C_{\gamma_0}$ and $C_{\gamma_1}$ are hamiltonian isotopic in $X_{\gamma_0(0)} = X_{\gamma_1(0)}$. Such a path $\mu$ is called a \textbf{matching path}.

    It is then possible to extend the isotopy to a hamiltonian isotopy of $X$ supported in a neighbourhood of $X_{\gamma_0(0)}$, identifying $C_{\gamma_0}$ and $C_{\gamma_1}$. The \textbf{matching cycle} $S_\mu$ associated to $\mu$ is the union of $T_{\gamma_0}$ and $T_{\gamma_1}$, glued along their boundaries via the above isotopy. $S_\mu$ is an exact Lagrangian sphere in $X$.
\end{definition}

Given two matching paths $\mu$ and $\nu$, we can perform \textbf{plane twists} $\tau_\nu^{\pm1}(\mu)$, which are symplectomorphisms of $\C$ supported in an arbitrarily small neighbourhood of $\nu$, defined pictorially in \zcref{fig:plane_twist}. The following lemma confirms that $\tau_\nu(\mu)$ is also  a matching path.

\begin{figure}[htbp]
    \centering
    \def\svgwidth{0.6\textwidth}
    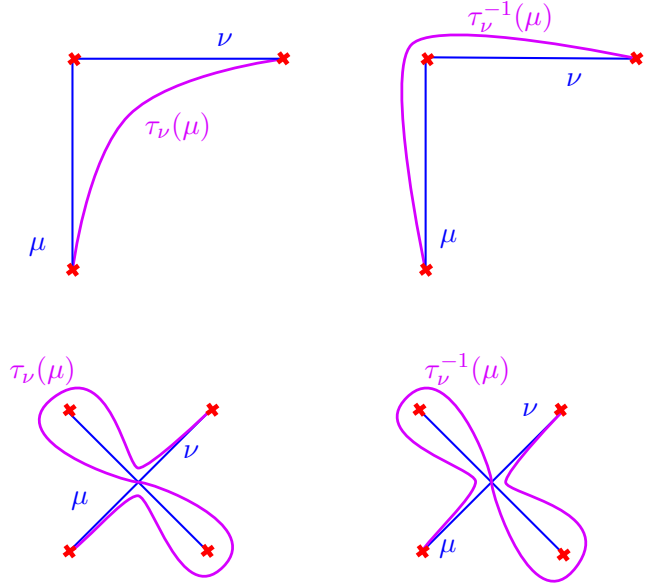
    \caption{The plane twists $\tau_\nu(\mu)$ (left) and $\tau_\nu^{-1}(\mu)$ (right).}
    \label{fig:plane_twist}
\end{figure}

\begin{lemma}\label{matching_paths_dehn_twist}\cite[Lemma 16.13]{seidelbook}
    Let $\mu$ and $\nu$ be two matching paths intersecting transversely. Then the plane twist $\tau_\nu \mu$ is a  matching path, and we have for the matching cycles that
    \[
        S_{\tau_\nu \mu} = \tau_{S_{\nu}}(S_{\mu}),
    \]
    where $\tau_{S_{\nu}} \in \mathrm{Symp}^c(X)$ is the symplectic Dehn twist along $S_\nu$.
\end{lemma}

\subsection{Lefschetz Fibrations on Affine Varieties}

Recall from \zcref{ex:weinstein_on_variety} and \zcref{ex:ample_weinstein} that smooth affine varieties $X \subset \C^n$ are naturally Weinstein manifolds, and thus admit a presentation as total spaces of abstract Weinstein Lefschetz fibrations. A plentiful supply of Lefschetz fibrations on affine varieties is obtained from the following construction originally found in the smooth setting in \cite{seidel_smith_fukaya_cotangents},\cite[Chapter 19]{seidelbook} and \cite[Section 8]{mclean_symp_homology}.

\begin{definition}[Lefschetz pencils]
    Let $\overline{X}$ be a complex projective variety (possibly singular), and consider an ample line bundle $\mathcal{L} \to \overline{X}$. Choose two sections $s_\infty, s_0 \in H^0(\overline{X}, \mathcal{L})$. Define the rational function
    \[
        p = \frac{s_0}{s_\infty} : \overline{X} \dashrightarrow \P^1,
    \]
    and call it a \textbf{Lefschetz pencil} if
    \begin{itemize}
        \item $s_\infty^{-1}(0) =: D$ is a normal crossing divisor;
        \item $X:= \overline{X} - D$ is smooth;
        \item $s_0^{-1}(0)$ is reduced, smooth near $D$ and intersects $D$ transversely;
        \item The \textbf{axis} $A := \{s_0 = s_\infty = 0\} \subset \overline{X}$ is smooth and contained in $\overline{X}^{sm}$;
        \item $p\vert_X: X \to \C$ only has nondegenerate critical points, at most one in each fibre.
    \end{itemize}
\end{definition}

Notice that $X$ as well as fibres of $p\vert_X$ are smooth complex affine varieties, making $p\vert_X$ into a symplectic fibration away from the critical values. Recall that there is a canonical symplectic connection associated to a symplectic fibration; due to noncompactness of the total space, failure of  the horizontal vector fields to be integrable may, a priori, prevent $p\vert_X$ from being a Lefschetz fibration.

We prove that parallel transport maps are in fact well-defined for all paths avoiding the critical values. Essentially the same proof goes through as for the smooth case in \cite[Lemma 2]{seidel_smith_fukaya_cotangents}, but we repeat the argument here for convenience.

\begin{lemma}
    Let $p: \overline{X} \dashrightarrow \P^1$ be a Lefschetz pencil. For any path $\gamma \subset \C$ avoiding the critical values, the parallel transport maps $X_{\gamma(0)} \to X_{\gamma(1)}$ are well-defined.
\end{lemma}

\begin{proof}
    We must show that trajectories of the horizontal vector fields do not escape to $D$ in finite time. The symplectic form on $X$ can be written as $\omega = -\dd \dd^\C h$, where $h = -\mathrm{log}\norm{\cdot}$ for a Kähler metric on $X$ obtained through ampleness of $\mathcal{L}$ as in \zcref{ex:ample_weinstein}. Recall that $h$ is an exhausting function for the Weinstein structure on $X$; hence establishing a bound on $\abs{\dd h (X_\gamma)}$ for a horizontal vector field $X_\gamma$ along $\gamma$ suffices to prove integrability of $X_\gamma$.

    To this end, let $\xi \in \mathfrak{X}(X)$ denote the unique horizontal lift of the vector field $\del_z$ on $\C$ away from the critical values and consider the gradient $\nabla p$ with respect to the Kähler metric. Then $Dp(\nabla p) = \dd p (\nabla p) \del_z = \norm{\nabla p}^2 \del_z$. Moreover, $\nabla p$ is  horizontal: For $w \in \ker \dd p$, we have $\omega(\nabla, w) = \langle \nabla p, Jw \rangle = \dd p (J w) = J \dd p(w) = 0$. It follows that

    \[
        \xi = \frac{\nabla p}{\norm{\nabla p}^2}.
    \]

    Now consider a point $\zeta \in A$. By smoothness of $A$ and $\overline{X}$ near $A$, as well as the assumption that $s_0^{-1}(0)$ intersects $D$ transversely and that $D$ is normal crossings, there exists a neighbourhood of $\zeta$ in $\overline{X}$ with local coordinates $(x_1, \ldots, x_n)$ on which

    \[
        s_0(x) = x_{k+1}, \quad s_{\infty}(x) = x_1^{m_1}\cdots x_{k}^{m_k}.
    \]

    In these coordinates, the Kähler metric is $\norm{\cdot}^2 = e^\sigma \abs{\cdot}^2$, where $\abs{\cdot}$ is the standard Euclidean metric and $\sigma$ is some smooth function. Hence $h = - \log \abs{s_\infty}^2 - \sigma$.

    One checks that
    \[
        \norm{\nabla p} \stackrel{\geq}{\sim} \abs{\pdev[p]{x_j}} \stackrel{\geq}{\sim} \frac{\abs{p}}{\abs{x_j}},
    \]
    and
    \[
        \norm{\nabla s_\infty} \stackrel{\leq}{\sim} \sum_{j=1}^k \frac{\abs{s_\infty}}{\abs{x_j}},
    \]
    where the notation $\stackrel{\geq}{\sim}$ indicates that the inequalities hold up to a positive multiplicative constant, mainly related to the comparison between $\norm{\cdot}$ and $\abs{\cdot}$. We can now estimate $\abs{\dd h (\xi)}$:
    \begin{align*}
        \abs{\dd h (\xi)} & = \frac{\abs{\langle \nabla p, \nabla h\rangle}}{\norm{\nabla p}^2}                                                                                           \\
                          & = \frac{\abs{\langle \nabla p, -\frac{2\abs{s_\infty}}{\abs{s_\infty}^2} \nabla s_\infty - \nabla \sigma \rangle}}{\norm{\nabla p}^2}                                                   \\
                          & \leq \frac{\abs{\langle \nabla p, \nabla \sigma \rangle}}{\norm{\nabla p}^2} + \frac{2\abs{\langle \nabla p, \nabla s_\infty \rangle}}{\norm{\nabla p}^2 \abs{s_\infty}} \\
                          & \stackrel{\leq}{\sim} \frac{1}{\norm{\nabla p}} + \frac{\norm{\nabla s_\infty}}{\norm{\nabla p} \abs{s_\infty}}.                                                            \\
    \end{align*}
    The last inequality uses Cauchy-Schwarz on both terms. Substituting the bounds obtained on $\norm{\nabla p}$ and $\norm{\nabla t}$, we obtain a $\frac{const.}{\abs{p}}$-bound on each summand.

    Now consider the flow of $\xi$ along $p^{-1}([a,b])$ for a horizontal segment $[a,b] \subset \C$ avoiding the critical values. Recall that the closure of any fibre $\overline{p^{-1}(z)} = \{s_0(x) = z s_{\infty}(x) \}$ meets precisely in $A$, so if a point in $p^{-1}([a,b])$ were to escape to $\zeta \in D$ in finite time under the flow of $\xi$, it would in fact escape to $A$. Assume for now that $0 \notin [a,b]$, so that we get a bound on $\frac{1}{\abs{p}}$.

    Cover the axis $A$ by finitely many coordinate patches as considered above, so that one obtains a bound on $\abs{\dd h (\xi)}$ on all of $p^{-1}([a,b])$ (in a neighbourhood of $A$ in $p^{-1}([a,b])$, one has the bounds established above, and outside this neighbourhood boundedness follows by compactness). Hence no trajectory of $\xi$ along any horizontal path in $\C$ away from $\mathrm{Critval}(p) \cup \{0\}$ escapes to $D$ in finite time.

    To remove the hypothesis on $0$, notice that for any constant $c$, $p+c$ is a symplectic fibration with the same parallel transport transport maps as $p$. Hence, by changing local coordinates, one can replace the $\frac{1}{\abs{p}}$-bound by a $\frac{1}{\abs{p+c}}$-bound, and choosing $c \notin [a,b]$ establishes boundedness of $\dd h (\xi)$ also for $0 \in [a, b]$.

    For an arbitrary path $\gamma \subset \C - \mathrm{Critval}(p)$, we have that $\dot{\gamma}$ is a complex multiple of $\del_z$, and thus the parallel transport vector field $X_\gamma$ is a complex multiple of $\xi$. Trajectories are well-behaved for $\xi$, and thereby for $X_\gamma$.
\end{proof}

\subsection{Lefschetz Bifibrations}

A Lefschetz bifibration should be thought of as a Lefschetz fibration $\pi: X \to \C$ equipped with a family $\rho_z: X_z \to \C$ of Lefschetz fibrations on the fibres. Their primary utility is that, given a vanishing path $\gamma$ for $\pi$, one can compute an explicit matching path $\mu$ for $\rho_{\gamma(0)}$; the matching cycle $S_\mu \subset X_{\gamma(0)}$ is the vanishing cycle $C_\gamma$ of $\pi$.

We restrict to the setting of affine varieties, as treated by \cite{parker}. For the general case, see \cite[(15e)]{seidelbook}.

\begin{definition}
    Let $X$ be a complex affine variety. A \textbf{Lefschetz bifibration} is a commutative diagram, satisfying the following properties:
    \[
        \begin{tikzcd}
            X \arrow[rr, bend left, "\pi"] \arrow{r}{\rho} & \mathcal{S}:=\C \times \C \arrow{r}{\psi}& \C,
        \end{tikzcd}
    \]
    \begin{enumerate}
        \item $\pi = \psi \circ \rho$ is a Lefschetz fibration, and $\psi$ has no critical points.
        \item For $z \in \C$, write $X_z = \pi^{-1}(z)$, and $\mathcal{S}_z = \psi^{-1}(z)$ for the fibres. Then the restriction $\rho_z: X_z \to \mathcal{S}_z$ is a Lefschetz fibration for all regular values $z$ of $\pi$.

              For $z_0 \in \mathrm{Critval}(\pi)$ and $x_0 \in X_{z_0}$ the corresponding critical point, $\rho_{z_0}: X_{z_0}\setminus\{x_0\} \to \mathcal{S}_{z_0}$ should be a holomorphic map whose critical points are nondegenerate, such that their images under $\rho$ are pairwise distinct, as well as distinct from $\rho(x_0)$.
        \item The critical points of $\rho$ are generic, in the sense that $\DD \rho$ is transverse in $\mathrm{Hom}_{\C}(TX, \rho^*T \mathcal{S})$ to both the zero section and the stratum of complex rank-one linear maps.
        \item Let $z$ be a critical value of $\pi$ and $x \in X_z$ the unique critical point. Then the restriction of $D^2\pi(x)$ to $\ker (\DD \rho (x)) \subset T_xX$ is non-degenerate.
    \end{enumerate}
\end{definition}

We will denote Lefschetz bifibrations by $(\pi, \rho)$. Let us note the examples giving the bifibrations we will use, along with some first consequences of this definition.

\begin{example}\label{ex:affine_bifibration}
    For a smooth complex affine variety $X = \{f=0\} \subset \C^n$, we obtain a Lefschetz bifibration by setting $\pi: X \to \C$ to be the restriction of a linear polynomial function in $\C[x_1, \ldots, x_n]$, $\psi: \C \times \C$ to be projection to the first factor, and $\rho: X \to \C \times \C$ given by $\rho = (\pi, \rho_t)$, where $\rho_t$ is another linear polynomial function.

    Indeed, consider the projectivization $\overline{X} \subset \P^n$  and the ample line bundle $\mathcal{O}(1)$; If $x_0$ denotes the homogenizing variable, let $s_\infty = x_0 \in H^0(\overline{X}, \mathcal{O}(1))$. Note that $s_\infty^{-1}(0) = H_\infty$ is smooth, and that the complement of the base locus $\overline{X} \setminus s_\infty^{-1}(0)$ is $X$. For a generic section linear section $s_0 \in H^0(\overline{X}, \mathcal{O}(1))$ (i.e., the restriction to $\overline{X}$ of a homogeneous polynomial of degree $1$), the ratio $s_0/s_\infty$ defines a Lefschetz pencil, and thus a Lefschetz fibration $\pi: X \to \C$. In practice, one thus chooses a generic linear polynomial $\pi = \sum_{i=1}^n a_i x_i: \C^n \to \C$, and this argument shows that its restriction to $X$ is a Lefschetz fibration.

    The fibres of $\pi$ are affine, so for a generic choice $\rho_t$, its restriction to the fibres is a Lefschetz fibration by the same argument.
\end{example}

\begin{claim}
    $\DD \rho$ is nonvanishing, and $\mathrm{Crit}(\rho) \subset X$ is a complex curve.
\end{claim}
\begin{proof}
    Transversality to the zero section $Z_0$ in $\mathrm{Hom}_{\C}(TX, \rho^*T \mathcal{S})$ implies that $\DD \rho^{-1}(Z_0)$ is a smooth submanifold of $X$ of the same real codimension as $Z_0$. Let $\dim_\C X = n$. The complex codimension of $Z_0$ is the dimension of the fibre, which is $2n$. Hence, $\DD \rho^{-1}(Z_0)$ is a real-codimension-$4n$ submanifold of $X^{2n}$, and hence empty. Hence $x \in  \mathrm{Crit}(\rho)$ iff $\DD \rho(x)$ has complex rank one.

    To find the complex dimension of the stratum $Z_1$ of rank-$1$ linear maps, recall that the dimension of the variety of rank-$r$ $m\times n$ matrices is $r(m+n-r)$, and so $\dim_\C Z_1 = n+1$. Its codimension is hence $2n-n-1 =n-1$, and thus $\mathrm{Crit}(\rho) = \DD \rho^{-1}(Z_1)$ is a smooth submanifold of $X$ of real dimension $2n - 2(n-1) = 2$.
\end{proof}

One can prove the following local form for Lefschetz bifibrations \cite[Lemma 15.9]{seidelbook}:

\begin{lemma}\label{local_bifibration}
    For each critical point $x_0 \in X$ of $\pi$, there exist complex charts around $x_0 \in X$, near $\rho(x_0) \in \mathcal{S}$, and near $\pi(x_0) \in \C$, in which
    \[
        \rho(x_1, \ldots, x_n) = (x_1^2 + \ldots + x_n^2, x_1), \quad \psi(y_1, y_2) = y_1.
    \]
\end{lemma}

Using these coordinates, one can prove that $\mathrm{Crit}(\rho)$ is in fact a complex curve, that is, $T \mathrm{Crit}(\rho)$ is invariant under the complex structure. Moreover, \cite[Lemmata 15.7, 15.8]{seidelbook} show that the critical points of $\rho_z$ are precisely $X_z \cap \mathrm{Crit}(\rho)$, and that
\[
    \pi\vert_{\mathrm{Crit}(\rho)}: \mathrm{Crit}(\rho) \to \C
\]
is a branched cover with only doubly ramified points, which occur at the critical points of $\pi$.

\begin{construction}\label{matching_cycle_construction}
    Given a path $\gamma$ in $\C$, we may consider the finite collection of critical values $\mathrm{Critval}(\rho_{\gamma(t)}) \subset \mathcal{S}_{\gamma(t)}$, varying smoothly along $\gamma$ in $t$. Parallel transport on $\mathcal{S}$ in turn gives a symplectomorphism
    \[
        k_\gamma^t: (\mathcal{S}_{\gamma(0)}, \mathrm{Critval}(\rho_{\gamma(0)})) \to (\mathcal{S}_{\gamma(t)}, \mathrm{Critval}(\rho_{\gamma(t)})).
    \]
    The assertion that parallel transport on $\mathcal{S}$ maps $\mathrm{Critval}(\rho_{\gamma(0)})$ into $\mathrm{Critval}(\rho_{\gamma(t)})$ can be shown by proving that $T \mathrm{Crit}(\rho)$ is a horizontal subspace of $T X$ using a computation in the local model \zcref{local_bifibration}. This implies that parallel transport on $X$ maps the critical points of $\rho_{\gamma(0)}$ to the critical points of $\rho_{\gamma(t)}$, and the claim for $k_\gamma^t$ follows.




    We can extract a matching path for $\rho_{\gamma(0)}$ from this branching behaviour. For $t$ close enough to $1$, pick the shortest path joining the critical values that are identified under $k_\gamma^1$. Its preimage under $k_\gamma^t$ in $\mathcal{S}_{\gamma(0)}$ is a matching path $\mu$ for $\rho_{\gamma(0)}$.

    The upshot of this construction is that the matching cycle $S_\mu \subset X_{\gamma(0)}$ associated to this path is the vanishing cycle $C_\gamma$ of the Lefschetz fibration $\pi$.
\end{construction}

Recall that the Lefschetz fibration $\pi$ is determined by the Weinstein type of the fibre $X_{\gamma(0)}$ and the ordered collection of vanishing cycles. As we explain now, \zcref{matching_cycle_construction} combined with \zcref{matching_paths_dehn_twist} allows us to reconstruct this data from the Lefschetz fibration $\rho_{\gamma(0)}: X_{\gamma(0)} \to \C$.

\begin{definition}
    A \textbf{basis of matching paths} for a Lefschetz fibration $\rho: X^{2n} \to \C$ is a collection of matching paths $\{\nu_1, \ldots, \nu_r\}$ such that the collection of matching cycles $\{S_{\nu_i}\}$ induces a basis of $H_n(X; \Z)$.
\end{definition}

Starting from a Lefschetz bifibration $(\pi, \rho)$, choose a distinguished basis of vanishing paths $\Gamma = (\gamma_1, \ldots, \gamma_k)$ with base point $* \in \C$. Next, find a basis of matching paths $N = \{\nu_1, \ldots, \nu_r\}$ for $\rho_*: X_* \to \C$. Applying \zcref{matching_cycle_construction} for each $\gamma_i \in \Gamma$ yields a new set of matching paths for $\rho_*$, which we denote by $(\mu_1, \ldots, \mu_k)$. Next, one attempts to express the $\mu_i$ in terms of plane twists along the elements of $N$ in the form

\begin{equation}\label{matching_paths_expression}
    \mu_i = \tau_{\nu_{i_1}}^{e_1} \circ \cdots \circ \tau_{\nu_{i_l}}^{e_l}(\nu_j), \quad e_j \in \Z.
\end{equation}

\zcref{matching_paths_dehn_twist} then implies that $S_{\mu_i}$, and thus $C_{\gamma_i}$, is hamiltonian isotopic to $\tau_{S_{\nu_{i_1}}}^{e_1} \circ \cdots \circ \tau_{S_{\nu_{i_l}}}^{e_l}(S_{\nu_j})$.

Let us specialize \zcref{matching_cycle_construction} for bifibrations $X \xrightarrow{(\pi, \rho_t)} \C \times \C \xrightarrow{\mathrm{pr}_1} \C$ obtained as in \zcref{ex:affine_bifibration}; that is, $\pi$ and $\rho_t$ are the restrictions to $X$ of sufficiently generic linear polynomials.

\begin{procedure}
    \label{proc:matching_cycles}
    \begin{enumerate}
        \item Compute $\mathrm{Critval}(\pi) \subset \C$. To compute the critical values of the restriction $\pi: X = \{f=0\} \to \C$, find the locus where the gradient $\nabla f$ is parallel to $\dd \pi$, where one considers $\pi$ as a polynomial map $\C^n \to \C$.
        \item Choose a distinguished basis of vanishing paths; whenever we perform this procedure in the rest of this paper, this choice will consist in determining a single origin point so that all our vanishing paths are the linear paths joining it to the critical values. Our ordering is counterclockwise.
        \item For each vanishing path $\gamma$ compute $\mathrm{Critval}(\rho_{\gamma(0)})$. This can be done explicitly over any fibre because we chose $\rho_t$ and $\pi$ to be linear polynomials: to obtain an explicit description of the fibre $X_{z}$, simply solve $\pi(x_1, \ldots, x_n) = z$ for any variable, say, $x_1$, and substitute into a defining polynomial $f$ of $X$. This will yield a defining polynomial for the fibre (in one variable less), so that critical values can be computed just as in Step 1.
        \item For each $t \in [0,1]$, compute $\mathrm{Critval}(\rho_{\gamma(t)})$. The result is an isotopy of $\mathrm{pr}_1(\mathrm{Critval}(\rho_{\gamma(0)})) \subset \C$ eventually joining two points in $\mathrm{Critval}(\rho_{\gamma(1)})$.
        \item For $t$ close to $1$, pick the linear path joining the two critical values that converge under $k_\gamma^1$, and observe its image when reversing the isotopy to $X_{\gamma(0)}$. This is the matching path $\mu$.
    \end{enumerate}
\end{procedure}

As these computations involve finding zeros of (many) generic polynomials, it is usually not feasible to perform them by hand and one requires a computer to make use of this methodology. An implementation of this procedure is available in \cite{hbds_repo}.

\section{Fronts of Legendrian lifts}
\label{sec:fronts}
In this section, we review the two main ingredients from \cite{cm_lefs} to draw fronts of Legendrian lifts of the vanishing cycles of a Lefschetz fibration. Recall that the basis of matching paths to be chosen in step 4 of \zcref{recipe_3_3} should be such that we know how to draw the fronts of the lifts of their matching cycles; the first ingredient exhibits fronts for certain Lagrangian cycles in sphere plumbings, see \zcref{subsec:plumbings}. The Lefschetz fibrations we choose later on will have as their general fibre such a plumbing, and we will be able to realize these skeleta as unions of matching  cycles.

The second ingredient is related to Step 5, where we express the vanishing cycles in the form of \zcref{matching_paths_expression}; to draw a front of $S_{\mu_i}$ when we know how to draw fronts of the $S_{\nu_i}$, we review the effect of Dehn twists $\tau_{S_{\nu_i}}$ on the Lagrangians $S_{\nu_j}$ (\zcref{subsec:dehn_twists_lifts}).

\subsection{Plumbings and standard skeleta}\label{subsec:plumbings}

When the fibre is a plumbing, one can exhibit a standard set of Lagrangian spheres whose lifts can be described explicitly. Intuitively, given two closed manifolds $Q_0$ and $Q_1$, their plumbing is defined as the disjoint union of their disc cotangent bundles, identifying a cotangent fibre of $\D T^*Q_0$ with a Lagrangian disk in the zero section of $\D T^*Q_1$. For a more rigorous construction of plumbings, we refer to \cite[Definition 2.1]{plumbings}. For the rest of this article, we consider only $Q_0 = Q_1 = S^1$.

\begin{definition}[Tree Plumbings]
    Given a tree graph $T$ with vertex set $T^{(0)}$ and edges $T^{(1)}$, the $T$-plumbing of spheres is the disjoint union of $k = \abs{T^{(0)}}$ copies of $\D T^*S^1$, labeled by the vertices of $T$, modulo the equivalence relation plumbing two copies of $\D T^*S^1$ whenever the corresponding vertices are connected by an edge. The resulting Liouville domain is denoted by $F_T$.
\end{definition}

Notice that plumbing a $\D T^*S^1$ is the same as attaching a $1$-handle. Hence, $\del(F_T \times \D^2) = \mathrm{ob}(F_T; \mathrm{id})$ in a front projection consists of $\abs{T^{(0)}}$ pairs of 3-balls representing the $1$-handles.

\begin{definition}
    Given a $T$-plumbing $F_T$, the \textbf{Lagrangian $T$-skeleton} of $F_T$ is defined to be the collection of the zero sections of each $\D T^*S^1$ in the plumbing.
\end{definition}

The upshot is that by \cite[Proposition 2.13]{cm_lefs}, any collection of $\abs{T^{(0)}}$ Legendrian spheres in $\mathrm{ob}(F_T; \mathrm{id})$ which meet in accordance with the tree $T$ and are in cancelling position with the $1$-handles are Legendrian isotopic to a Legendrian lift of the Lagrangian $T$-skeleton. See \zcref{fig:leg_skeleton} for an example of our choice of lift for the rest of this article.

\begin{figure}[htbp]
    \centering
    \def\svgwidth{\textwidth}
    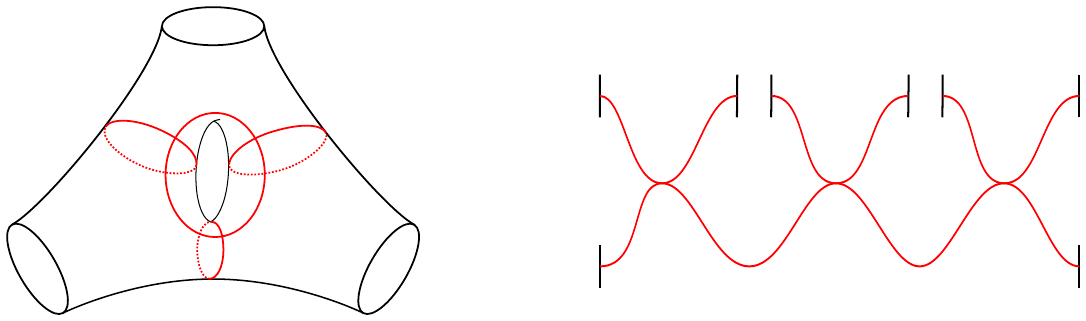
    \caption{Left: $F_T$ for $T=D_4$ with the Lagrangian $D_4$-skeleton in red. Right: The standard Legendrian lift to $\mathrm{ob}(F_T; \mathrm{id})$.}
    \label{fig:leg_skeleton}
\end{figure}

\subsection{Finding a basis of matching paths for plumbings}
\label{subsec:basis_of_matching_paths}

In practice, and indeed, in all cases occurring in our computations, plumbings are often presented as affine varieties in $\C^2$. Given a smooth affine variety $F \subset \C^2$, its diffeomorphism type is determined by its genus and the number of boundary punctures. Let $F = \{f=0\}$ for some $f \in \C[x, y]$. The genus is given by the genus-degree formula ($2g = (d-1)(d-2)$ if $d$ is the degree of $f$), or equivalently by counting the number of interior lattice points of the Newton polytope of $f$; for the number of boundary punctures, consider the projectivization $\overline{F} \subset \P^2$ given by $\{\overline{f} = 0\}$, where $\overline{f} \in \P[x,y,w]$ is the homogenization of $f$. The number of boundary punctures is the number of points in $\overline{F} \cap \{ w= 0\}$, that is, the intersection of $\overline{F}$ with the hyperplane at infinity. Given the diffeomorphism type of $F$, one next determines a tree $T$ such that the sphere plumbing $F_T$ is diffeomorphic to $F$.

\subsubsection{Vanishing cycles of a branched cover}

Suppose we are given a Lefschetz fibration $\rho: F_T \to \C$ (which is in this dimension simply a doubly ramified branched cover), presented as the restriction to $F_T$ of a polynomial map in $\C[x,y]$. We now explain how to exhibit in this setting a basis of matching paths realizing the standard Lagrangian $T$-skeleton.

To construct a single matching path, one requires knowledge of the vanishing cycles of $\rho$. Choose a distinguished basis of vanishing paths for $\rho$, with base point $*$. Assume the cover has degree $d$, so that the generic fibre $\rho^{-1}(*)$ consists of $d$ distinct points. There are $\binom{d}{2}$ $0$-spheres in a generic fibre, and parallel transport along a vanishing path collapses one of these to a single point. In other words, the vanishing cycles of $\rho$ specify which two sheets in the fibre over $*$ come together at a branch point.

Notice that the coordinates in $\C^2$ of the points in the fibre over any point can be computed explicitly using numerical approximation; if we also denote the polynomial map restricting to $\rho$ on $F_T$ by $\rho$, and if $f$ is a defining polynomial for $F_T$, then

\[
    \rho^{-1}(z) = \{ (x,y) \in \C^2 \mid f(x,y) = 0, \, \rho(x,y)-z = 0 \}.
\]

One can visualize the fibres by projecting to a generic plane (in our applications, we project to the first coordinate), so that the generic fibre is represented by $d$ points in $\C$. The effect of parallel transport along a vanishing path $\gamma$ may then be observed by performing this computation for $z$ along the image of $\gamma$, inducing a deformation of $\mathrm{Critval}(\rho)$ such that two critical values merge at $t=1$. These two points form the vanishing cycle $C_\gamma^\rho$ of $\rho$. In this dimension, this computation can be carried out by a computer, and so may be performed conveniently for any vanishing path.

\subsubsection{Matching paths from branching behaviour}
We present some tools to determine a matching path using this information. If $C_\gamma = C_\delta$, then $\gamma \#\overline{\delta}$ is a matching path. One way to obtain more matching paths is given by determining the vanishing cycles for a large collection of vanishing paths and concatenate those with identical vanishing cycles.

A more systematic approach is indicated in \zcref{fig:hurwitz_moves_general}: Start with a distinguished basis of vanishing paths and determine their vanishing cycles. Draw a path connecting two critical values and isotope it into a collection of a vanishing path, a series of loops around critical values, followed by another vanishing path, all based at the reference fibre from the distinguished basis of vanishing paths. To determine whether this path is a matching path, one must verify that parallel transport along the path sends the vanishing cycle at one end to that same vanishing cycle at the other end, which can be checked using knowledge of the effect of Hurwitz moves and that the monodromy along a counterclockwise looping of vanishing path is a Dehn twist along the corresponding vanishing cycle.

\begin{figure}[htbp]
    \centering
    \def\svgwidth{0.8\textwidth}
    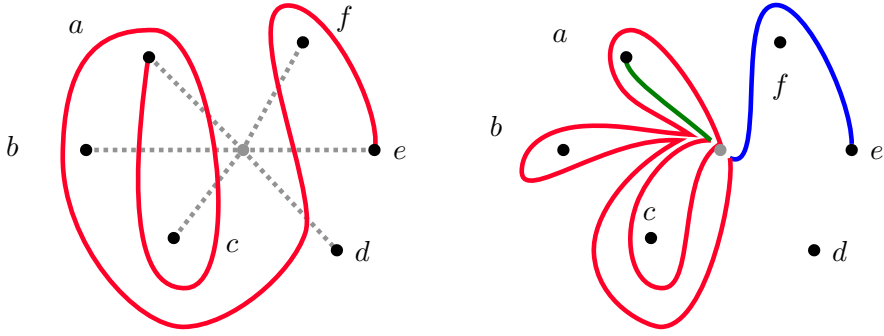
    \caption{Left: A distinguished basis of vanishing paths with their respective vanishing cycles (labeled by letters) and a path connecting two critical values. Right: The isotopy allows one to determine that parallel transport along the red segment sends $a$ to $C := \tau_c\tau_b\tau_a\tau_c^{-1}(a)$ in the reference fibre. If $C$ coincides with the vanishing cycle of the blue segment (which is $\tau_f^{-1}(e)$, as can be determined by a single Hurwitz move on the grey basis of vanishing paths), then the path is a matching path.}
    \label{fig:hurwitz_moves_general}
\end{figure}

\begin{remark}
    A Dehn twist $\tau_{S_0}$ for a $0$-sphere $S_0$ is its own inverse, and if $S_1$ is another $0$-sphere sharing a single point with $S_0$, then $\tau_{S_0}(S_1)$ is the $0$-sphere consisting of the two points in $S_0 \cup S_1 \setminus (S_0 \cap S_1)$.
\end{remark}

The heuristic we use to realize the Lagrangian $T$-skeleton of $F_T$ via matching paths is via the intersection pattern $T$ of the $T$-skeleton. The matching cycle of any matching path is an exact Lagrangian $S^1 \subset F_T$; if one ensures that the intersection pattern of these cycles is $T$, then the collection of matching cycles has to be exactly the Lagrangian $T$-skeleton of $F_T$. From only this information, it is not always possible to precisely determine which path realizes which cycle, but this precision is not needed to apply \zcref{recipe_3_3}.

Observe that for two matching paths intersecting at the endpoints, their matching cycles intersect transversely in a single point, whereas the behaviour of the matching cycles is not a priori clear if the intersection takes place in the interior. Over the $d$ points in the fibre over the intersection point, lying each in a separate sheet of the covering, one has to determine which two sheets come together at the branch points associated to the respective matching paths. One can determine those sheets by isotoping the intersection point to the base point $*$ similarly to the procedure illustrated in \zcref{fig:hurwitz_moves_general}, interpreting the resulting segments of the matching paths as vanishing paths, and determining their vanishing cycles. One checks if they share a single point.

In short, one finds a basis of matching paths for $\rho: F_T \to \C$ by constructing matching paths exhibiting a $T$-intersection pattern, with as many of the intersections as possible being at the endpoints. For any interior intersection, one has to check by hand that the vanishing cycles over the intersection point intersect in a single point (i.e., if the distinguished basis of vanishing paths is isotoped to have its base point at the intersection point, one computes the vanishing cycles of the new basis, and reads off their intersection).

\subsubsection{Dehn twists and Legendrian lifts}\label{subsec:dehn_twists_lifts}
Given two exact Lagrangians $S, L \subset (F, \lambda)$, with $S$ diffeomorphic to a sphere, we can consider the Dehn twists $\tau_{S}^{\pm}(L)$. We review in this section how to describe the Legendrian lift $\Lambda_{\tau_S^{\pm}(L)}$ of $\tau_{S}^{\pm}(L)$, given Legendrian lifts of $S$ and $L$.

\begin{proposition}{\cite[Proposition 2.24]{cm_lefs}}\label{dehn_fronts}
    Consider two exact Lagrangians $S, L$ in the page $F$ of a contact open book $\mathrm{ob}(F, \lambda; \phi)$, such that $S$ is diffeomorphic to a sphere. Assume that the potential functions $\lambda\vert_S$ and $\lambda\vert_L$ are $C^0$-bounded by some small $\eps >0$. Consider the contact manifold $\widehat{\mathrm{ob}(F, \lambda; \phi)}$ obtained by performing $(+1)$-surgery along $\Lambda_S^{\epsilon}$ and $(-1)$-surgery along $\Lambda_S^{5\epsilon}$, where $\Lambda_K^{\epsilon}$ is the Legendrian lift of the exact Lagrangian $K$ given by
    \[
        \Lambda_K^\epsilon = \{[(x, g(x))] \subset \mathrm{ob}(F; \mathrm{\phi}) \mid x \in K \}.
    \]
    Here, $g = \lambda\vert_K: K \to \R$ is a potential. Then
    \begin{itemize}
        \item $\widehat{\mathrm{ob}(F, \lambda; \phi)}$ is canonically contactomorphic to $\mathrm{ob}(F, \lambda; \phi)$;
        \item The Legendrian $\Lambda_L^{3\epsilon}$ is Legendrian isotopic to $\Lambda_{\tau_S(L)}^0 \subset \mathrm{ob}(F, \lambda; \phi)$.
    \end{itemize}
    Analogously, performing $(-1)$-surgery along $\Lambda_S^{\epsilon}$ and $(+1)$-surgery along $\Lambda_S^{5\epsilon}$ yields a contact identification $\overline{\mathrm{ob}(F, \lambda; \phi)} \cong \mathrm{ob}(F, \lambda; \phi)$ under which $\Lambda_L^{3\epsilon} \subset \overline{\mathrm{ob}(F, \lambda; \phi)}$ is Legendrian isotopic to $\Lambda_{\tau_S^{-1}(L)}^0 \subset \mathrm{ob}(F, \lambda; \phi)$.
\end{proposition}

Combining this result with \zcref{handle_slides} now allows us to deduce how to draw $\Lambda_{\tau_S^{\pm}}(L)$ in the front projection of $\mathrm{ob}(F, \lambda; \mathrm{id}) \cong \#^k(S^1 \times S^2) \cong \del_\infty(F \times \C)$.

First of all, notice that if $S$ and $L$ are disjoint, then $\tau_S(L) = L$, and the lift does not change. So let us assume that $S \pitchfork L$; Recall that Lagrangian intersections give rise to Reeb chords between their Legendrian lifts $\Sigma$ and $\Lambda$, respectively, and restrict attention to a neighbourhood of a particular intersection point, and thus to a particular Reeb chord. Choose the Reeb chord to have length $0$, i.e., $\Sigma$ and $\Lambda$ intersect (non-generically) such that their tangent planes span the contact distribution at that point.

\zcref{fig:dehn_fronts} explains what happens in the front projection of a neighbourhood of this intersection point: first draw the three Reeb pushoffs as indicated in \zcref{dehn_fronts}, and label the pushoffs of $\Sigma$ by the framing for their contact surgery; the Reeb vector field in the front projection is $\del_z$, so the Reeb pushoffs are given by vertical shifts. This procedure yields three new arcs.

\begin{figure}[ht]
    \centering
    \def\svgwidth{0.8\textwidth}
    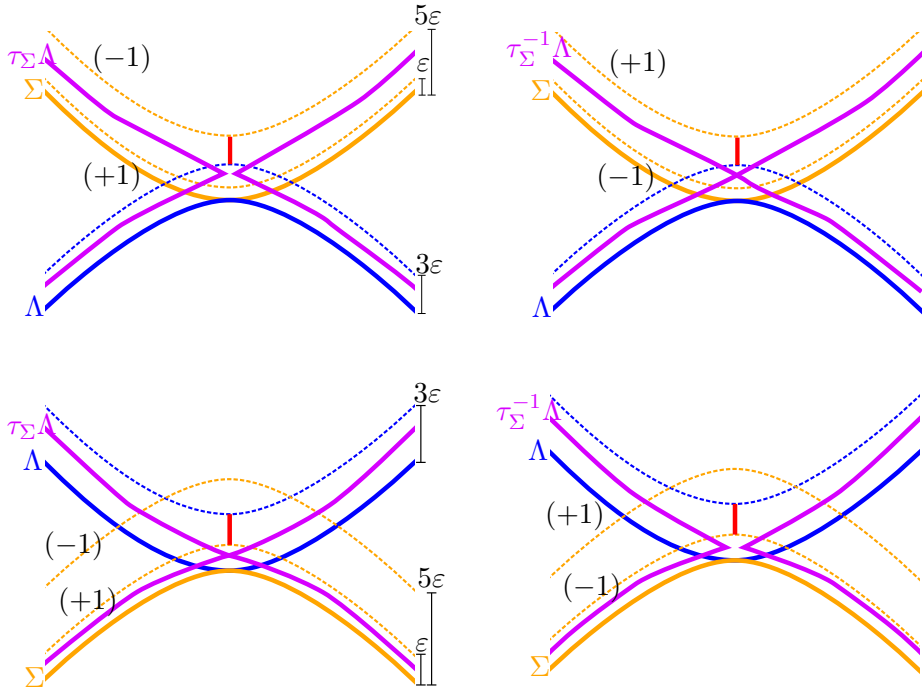
    \caption{The four possible configurations of a Legendrian sphere $\Sigma$ and a Legendrian $\Lambda$; The diagrams take place in a neighbourhood of the intersection point in $\Sigma \cap \Lambda$ within the front projection of $\mathrm{ob}(F, \lambda; \mathrm{id})$.}
    \label{fig:dehn_fronts}
\end{figure}

Notice that the arc corresponding to $R_{3\eps}(\Lambda)$ intersects precisely one of the other two arcs (in the front projection), and lies entirely below or above the other, which one depending on their relative height in the front projection. For instance, if $\Sigma$ lies above $\Lambda$ in a neighbourhood of the intersection point, then $R_{3\eps}(\Lambda)$ lies entirely below $R_{5\eps}(\Sigma)$, along which we are supposed to perform a $(-1)$-surgery. The pushoffs $R_{3\eps}(\Lambda)$ and $R_{5\eps}(\Sigma)$ are still Legendrian lifts of the same transversely intersecting Lagrangians, joined by a Reeb chord $c$.

By \zcref{handle_slides}, we can now perform a handle slide along $c$, and obtain the Legendrian $h_{R_{5\eps}(\Sigma)}(R_{3\eps}(\Sigma))$ in the upper left of \zcref{fig:dehn_fronts}. We obtain a Legendrian with cuspidal singularities since we are handle-sliding over a $(-1)$-surgery Legendrian. By the contactomorphism from the first bullet of \zcref{dehn_fronts} stemming from the cancellation lemma \cite[Proposition 6.4.5]{geiges_contact_intro}, we can now delete the auxilary arcs, and what is left describes the front of $\Lambda_{\tau_S(L)}$.

The remaining cases are entirely analogous; if $\Lambda$ lies instead below $\Sigma$, we instead perform a handle slide over a $(+1)$-surgery Legendrian, and end up with a cone singularity; for the inverse Dehn twists, simply invert the order of the surgeries.

This explains how to draw the Legendrian lift of $\tau_S(L)$ in the front projection of $\mathrm{ob}(F, \lambda; \mathrm{id})$ given the Legendrian lifts of $S$ and $L$.

We are now ready to draw the entire attaching link of $\mathrm{Lf}(F_T; C_1, \ldots, C_k)$, given a factorization of each vanishing cycle $C_i$ into Dehn twists along cycles of the Lagrangian $T$-skeleton of $F_T$, as in \zcref{matching_paths_expression}.

Start by drawing the standard lift from \zcref{fig:leg_skeleton}. Match up the components of the lift with their Lagrangian projections via the $T$-intersection pattern.

If the $l$-th vanishing cycle is expressed by a sequence of Dehn twists on $\nu$, trace a out a small positive Reeb pushoff by $l \cdot \eps$ of $\Lambda_\nu$ in the front, and perform the local modifications from \zcref{fig:dehn_fronts} for each of the Dehn twists in the factorization.

The increasing amount of pushoff reflects the order in which the critical handles are attached.

We are now able to review all the steps of the algorithm \cite[Recipe 3.3]{cm_lefs}, starting from a complex affine variety $X \subset \C^3$. The aim is to take a generic Lefschetz fibration $\pi: X \to \C$ and determine the collection of vanishing cycles in a general fibre $F$. The total space $X$ is then obtained by attaching Weinstein handles to the Legendrian lifts of these cycles to $F \times \D^2$ as in (\zcref{abstract_lefschetz}); The output of the algorithm is a front projection of this Legendrian lift in $\del_\infty(F \times \D^2) = \mathrm{ob}(F; \mathrm{id})$.

\begin{construction}\label{recipe_3_3}
    \begin{enumerate}
        \item Choose two linear polynomials $(\pi, \rho): \C^3 \to \C^2$. Generically, their restriction to $X$ will form a Lefschetz bifibration.
        \item Determine the general fibre $F$ of $\pi$ and express it as a plumbing $F_T$ for a tree $T$.
        \item Perform \zcref{proc:matching_cycles}, which amounts to choosing a basis of vanishing paths for $\pi$ meeting in a base point $*$ and results in a collection of matching cycles for $\rho\vert_{F_*}$.
        \item Determine a basis of matching paths for $\rho\vert_{F_T}$ as in \zcref{subsec:basis_of_matching_paths}.
        \item Express the matching paths from Step 3 in terms of plane twists along the basis paths.
        \item Use \zcref{dehn_fronts} to draw a front of the Legendrian lifts $\Lambda_i$ of the vanishing cycles to $\mathrm{ob}(F; \mathrm{id})$. Since the Reeb vector field in the front projection is $\del_z$, the lifts $\Lambda_i$ are drawn with a vertical pushoff reflecting the ordering of the collection of vanishing cycles. Take note that if a vanishing cycle $C_0$ precedes $C_1$ with respect to the chosen ordering, and if the $C_i$ intersect, then there is a Reeb chord starting at the lift $\Lambda_0$ to $\Lambda_1$. This can lead to clasps in the front projection.
        \item Simplify the front using Reidemeister moves.
    \end{enumerate}

\end{construction}




\section{Handlebodies to the Painlevé Equations}
\label{sec:handle-bodies-Painleve}

In this section, we present the affine varieties corresponding to a generic choice of parameters in the Painlevé equations in terms of their symplectic handle data. Specifically, we give a Weinstein handlebody diagram for each space $M_X$, $X \in \mathfrak{P}$. We collect the affine equations from \cite{saito_vput} here:
\begin{align*}
    \mathbf{M}_{I} &= \{xyz + x + y +1 = 0\}\\
    \mathbf{M}_{II} &= \{xyz - x - \alpha y - z + \alpha  + 1 = 0 \mid \alpha \in \C^* \}\\
    \mathbf{M}_{II(FN)} &= \{xyz + x - y + z + s = 0 \mid s \in \C \}\\
    \mathbf{M}_{III(D6)} &= \{xyz + x^2 + y^2 +(1+\alpha\beta)x + (\alpha + \beta)y + \alpha\beta = 0 \mid \alpha, \beta \in \C^* \}\\
    \mathbf{M}_{III(D7)} &=  \{xyz + x^2 + y^2 + \alpha x + y = 0 \mid \alpha \in \C^* \}\\
    \mathbf{M}_{III(D8)} &= \{xyz + x^2 - y^2 - y = 0\}\\
    \mathbf{M}_{IV} &=  \{xyz + x^2 - (s_2^2 + s_1s_2)x - s_2^2 y -s_2^2 z+ s_2^2 + s_1s_2^3 = 0 \mid s_1 \in \C, \, s_2 \in \C^* \}\\
    \mathbf{M}_{V} &=\{xyz + x^2 + y^2 - (s_1 + s_2s_3)x - (s_2 + s_1s_3) y - s_3z + s_3^2 + s_1s_2s_3 + 1 = 0 \mid s_1, s_2 \in \C, \, s_3 \in \C^* \}\\
    \mathbf{M}_{V(deg)} &= \{ xyz + x^2 + y^2 + s_0x + s_1 y + 1 = 0 \mid s_0, s_1 \in \C \}\\
    \mathbf{M}_{VI} &=  \{xyz + x^2 + y^2 + z^2 -s_1x -s_2y-s_3z + s_4 = 0 \mid s_i = a_i a_4 + a_ja_k \\
    &\phantom{asldjf}\text{for } (i,j,k) \text{ a cyclic permutation of } (1,2,3), \, s_4 = a_1a_2a_3a_4 + a_1^2+a_2^2 + a_3^2 + a_4^2 - 4,\\
    &\phantom{asldjf}a_i \in \C \}.
\end{align*}
We note that the expression for $\mathbf{M}_{III(D8)}$ differs from that in \cite{saito_vput}. Our expression is the correction provided by \cite[Section 4.7]{szabo2021perversity}.

Note that the parameters appearing in the Painlevé equations give rise to a family of cubics.

\begin{lemma}
    \label{lem:parameter_independence}
    Let $\mathcal{P}$ be the parameter space for any of the Painlevé types and denote by $X_\alpha$, $\alpha \in \mathcal{P}$, the corresponding cubic surface. For any $\alpha_0, \alpha_1 \in  \mathcal{P}$ such that $X_{\alpha_0}$ and $X_{\alpha_1}$ are smooth, they are equivalent as Weinstein manifolds. 
\end{lemma}
\begin{proof}
    Set $\mathcal{E} = \{(\alpha, \mathbf{z}) \in \mathcal{P} \times \C^3 \mid f_\alpha(\mathbf{z}) = 0 \}$, where $f_\alpha$ denotes the equation cutting out our surface, with parameter $\alpha$. Denote the projection by $\pi: \mathcal{E} \to \mathcal{P}$ and assume that $\alpha$ avoids the set
        \[
    \Delta = \{ \alpha \in \mathcal{P} \mid \pi^{-1}(\alpha) \text{ is singular} \}
    \].
     We endow the $X_\alpha$ with the Weinstein Morse function $\phi_\alpha(\mathbf{z}) = \norm{z}^2$ coming from $\C^3$. 
    
    Now let $(\alpha_t)_t \subset \mathcal{P}$ be a path joining $\alpha_0$ and $\alpha_1$ avoiding $\Delta$. Consider the bundle of truncations
    \[
    \mathcal{C}_R = \{(t, \mathbf{z}) \in [0,1]\times \C^3 \mid t \in [0,1], \, f_{\alpha_t}(\mathbf{z}) = 0, \, \phi_{\alpha_t}(\mathbf{z}) \leq R \} \xrightarrow{\pi_R} [0,1], 
    \]
    for $R$ to be determined. Assume for now that there exists some positive $R$ which is larger than any critical value of $\phi_\alpha$, for any $\alpha_t$, $t \in [0,1]$. Denote the fibres by $\pi_{R}^{-1}(t) = X_{\alpha_t}^R$.

    It follows that $\pi_R$ is proper, and that both $\pi_R$ and $\pi_R\vert_{\bigcup \{\phi_{\alpha_t} = R\}}$ are smooth submersions (since the fibres are smooth by assumption, and $R$ is a regular value of all the $\phi_\alpha$ along the path $\alpha_t$). Thus by the Ehresmann fibration theorem, $\pi_R$ is a locally trivial fibration, and we obtain some smooth trivialization
    \[
        h: X_{\alpha_0}^R \times [0,1] \xrightarrow{\sim} \mathcal{C}_R.
    \]
    Pulling back the Weinstein structures along $h_t = h(\cdot, t)$ yields a smooth family of Weinstein domains $(X_{\alpha_0}^R, \psi_t := h^*_t \phi_{\alpha_t})$ (we suppress the Liouville form and vector field from notation and refer to a Weinstein structure by its function). Thus by \cite[Lemma 11.6]{ce_stein_weinstein}, their completions are Liouville homotopic. Extending the pulled-back Lyapunov functions $\psi_t$ over the cylindrical ends gives a Weinstein homotopy. Since $R$ lies above all their critical values, these completions are Weinstein deformation equivalent to the corresponding $X_{\alpha_t}$.

    


    We prove in \zcref{lem:bounded_critters} that such an $R$ exists for all the Painlevé cases. What remains to conclude the proposition is to show that $\mathcal{P} - \Delta$ is path-connected. This is the content of \zcref{lem:pathconn_complement}.
\end{proof}

\begin{lemma}\label{lem:bounded_critters}
    For any Painlevé type and any compact set $K \subset \mathcal{P}$ avoiding $\Delta$, the families $(X_\alpha)_{\alpha \in \mathcal{P}}$ admit some $R>0$ such that
    \[
    \{ \mathrm{Critval}(\phi_\alpha) \mid \alpha \in K \} \subset B_R(0).
    \]
\end{lemma}
\begin{proof}
    First, observe that the defining equations for the Painlevé families are of the form
    \begin{equation}
        \label{eq:general_pain}
    f_\alpha = xyz + ax^2 +by^2 + cz^2 + dx + ey + \ell z + g,
    \end{equation}
    with the parameters $a$ to $g$ being complex numbers depending algebraically on $\alpha \in \mathcal{P}$. In every case, $\mathcal{P} = \C^k \times (\C^*)^l$ for some $k,l$. Consider the partial compactification
    \[
    \overline{\mathcal{E}} = \{ (\alpha, [\mathbf{z}: w]) \in \mathcal{P}\times \P^3 \mid \overline{f}_{\alpha}([\mathbf{z}:w]) = 0 \},
    \]
    where $\overline{f}_\alpha$ denotes the homogenization, together with the projections $\overline{\pi}: \overline{\mathcal{E}} \to \mathcal{P}$ and $\mathrm{pr}: \overline{E} \to \P^3$. Let $\overline{D} = \overline{\mathcal{E}} - \mathcal{E}$. For all Painlevé types, the divisor at infinity is
    \[
    \overline{D} = \{ (\alpha, [\mathbf{z}: w]) \mid w = 0, \, xyz = 0 \} =: \mathcal{P}\times D.
    \]
    The fibres of $\overline{\pi}$ are the compactifications $\overline{X_\alpha}$, which meet at $D$.


    We are to show that the critical values of the functions $\phi_\alpha$, which is the restriction of $\phi(\mathbf{z}) := \norm{\mathbf{z}}^2$ to $X_\alpha$, do not leave a common compact set as $\alpha$ varies in a compact range, say, $K \subset \mathcal{P}$. We achieve this by exhibiting an open neighbourhood $\overline{U}_{\alpha_0, p}$ of each $(\alpha_0, p) \in \mathcal{P} \times D$, such that $\dd \phi_\alpha(x) \neq 0$ for $(\alpha, x) \in \overline{U}_{\alpha_0, p} \cap X_\alpha$. 

   Cover $K \times D$ with a finite union $\overline{U}$ of $\overline{U}_{\alpha_k, p_k}$ and consider $\overline{\mathcal{E}}_K := \overline{\pi}^{-1}(K) = \overline{\mathcal{E}} \cap (K \times \P^3)$, which is compact. Now let $\mathcal{C}_K = \overline{\mathcal{E}}_K - \overline{U}$, which is also compact, and consider $\Phi: \mathcal{E} \to \R$ defined by $\Phi(\alpha, x) = \phi(x)$. Then the image $\Phi(\mathcal{C}_K)$ is compact and contains the critical values of all the $(\phi_\alpha)_{\alpha \in K}$, as desired.

    


    We construct $\overline{U}_{\alpha_0, p}$ locally, in a neighbourhood of the smooth and singular points of $D$ (these are $D^{\mathrm{sing}} = \{[1:0:0:0], \, [0:1:0:0],\, [0:0:1:0]\}$).

    \textbf{Near smooth points of $D$}: Up to permutation of variables, a smooth point of $D$ is of the form $p = [1:y_0:0:0]$. In the $\{x \neq 0\}$-chart near $p$, by the explicit forms from \zcref[noname]{eq:general_pain}, we have
    \[
    \overline{\mathcal{E}} = \{ (\alpha, [1:y:z:w]) \in \mathcal{P}\times \P^3 \mid  \underbrace{yz + w A_{\alpha}}_{=:q_\alpha(y,z,w)}=0 \},
    \]
    where $A_\alpha$ is some holomorphic function of $y,z,w$, depending on the Painlevé type (which can be read off directly from \zcref{eq:general_pain}). Since
    \[
    \pdev[q_\alpha]{z}(y_0, 0,0) = y_0 \neq 0,
    \]
    the implicit function theorem yields for any allowable $\alpha_0$ a neighbourhood 
    $\overline{U}_{\alpha_0, p}$ of $(\alpha_0, p) \in \mathcal{P} \times D$ where we can solve for $z = g_\alpha(y,w)$ for some  holomorphic $g_\alpha$, such that $(\alpha, y, w) \mapsto (\alpha, [1: y: g_\alpha(y,w): w])$ are local coordinates near $p$ in which $D$ is given by $\{w=0\}$.

    The function $F := \log(1 + \phi): \P^3 - \{w=0\} \to \R$ has the same critical points as $\phi$. Then we have
    \begin{align*}
        F &= \log\left(1 + \frac{1}{\norm{w}^2} + \norm{\frac{y}{w}}^2 + \norm{\frac{z}{w}}^2\right)\\
        &= -\log\left(\norm{w}^2   \right) + \log\left(1+\norm{w}^2 + \norm{y}^2 + \norm{z}^2   \right).
    \end{align*}
    For later use, observe that the second summand is well-defined and in fact smooth on $\{w=0\}$. Restricting $F$ to $F_\alpha: \overline{X}_\alpha \to \R$ corresponds to replacing $z$ by $g_\alpha(y,w)$; let us set 
    \[
        \psi_\alpha(y, w) = \log\left(1+\norm{w}^2 + \norm{y}^2 + \norm{g_\alpha(y, w)}^2   \right),
    \]
    
    which is hence smooth as a function on $\overline{U}_{\alpha_0, p}$.


    Now consider the $\R$-dilation on coordinates given by
    \[
    s \cdot (y, w) = (y, e^sw).
    \]
    Since this acts directly on the coordinates of $X_\alpha$, the action preserves $X_\alpha$, so that its generator $V$ is tangent to the $X_\alpha$. Now compute
    \begin{align*}
        \dd F_\alpha (V) &= \ddt{s} F(\alpha, y, e^sw)\\
        &= \ddt{s}\left[ -\log \norm{e^s w}^2 + \psi_\alpha(y, e^s w)\right]\\
        &= -2 + \dd \psi_\alpha(V).
    \end{align*}
    Recalling that $\psi_\alpha$ does extend smoothly to $D$, we deduce that $\norm{\dd \psi_\alpha}$ is bounded in a (possibly smaller) open neighbourhood of $(\alpha_0, p)$. The generator $V$ is explicitly computed to be $V = w \del_w + \overline{w}\del_{\overline{w}}$ in our coordinates, so that
    \[
    \abs{\dd \psi_\alpha(V)} \leq \norm{\dd \psi_\alpha} \norm{V} \leq C\abs{w}.
    \]
    Consequently, further shrinking $\overline{U}_{\alpha, p}$, we can arrange that $\abs{\dd \psi_\alpha(V)}<1$, which establishes that $\dd F (V) < -1$ in a neighbourhood of any $(\alpha_0, p) \in \mathcal{P} \times D^{\mathrm{sm}}$. Recalling that the critical points of $F$ are precisely those of $\phi_\alpha$, we deduce that $\phi$ has no critical points on $\overline{U}_{\alpha_0, p}$. For smooth points on the other irreducible components of $D$, the argument is symmetric.

    \textbf{Near singularities of $D$:} Let us work near $p = [1:0:0:0]$. In the $\{x \neq 0\}$-chart, $\overline{\mathcal{E}}$ is defined by 
    \[
    r_\alpha = yz + w(a + by^2 + cz^2) + w^2(d + ey + \ell z) + gw^3 = 0, 
    \]
    the coefficients $a$ to $g$ again depending algebraically on $\alpha$. The Hessian in the variables $y$ and $z$ is
    \[
    \left(\begin{matrix}
        2bw& 1\\ 1 & 2 cw
    \end{matrix}
    \right).
    \]

This is nondegenerate at $p$. By the holomorphic implicit function theorem, the equations
$\partial_y r_\alpha = \partial_z r_\alpha = 0$ admit a unique
local solution
$(y,z)=(y_c(\alpha,w),z_c(\alpha,w))$,
depending holomorphically on $(\alpha,w)$ and satisfying
$y_c(\alpha,0)=z_c(\alpha,0)=0$.
After translating this critical section to the origin in the
$(y,z)$-variables, the holomorphic Morse lemma (e.g., \cite[Lemma 2.26]{zoladek2006monodromy}) with parameters
$(\alpha,w)$ supplies a change of coordinates, preserving
$\alpha$ and $w$, in which
\[
    r_\alpha = uv+h_\alpha(w).
\]



    Let $k$ denote the vanishing order of $h_\alpha$ at $w=0$. If $k$ was constant in $\alpha$, we could, similarly to the smooth case, write 
    \[
    r_\alpha = uv + w^k A_\alpha(w)
    \]
    for some holomorphic $A_\alpha$ with $A_\alpha(w=0) \neq 0$, and perform similar estimates to rule out critical points in a (potentially smaller) neighbourhood of $p$. Indeed, \cite[Table 3]{szabo2021perversity} lists the singularity types of all points $D^{\mathrm{sing}} \subset \overline{X_\alpha}$. These are stated to hold for generic parameters, but inspecting the calculations in §§4.1-4.9. reveals that the only nonvanishing assumptions used are those already imposed by the allowed parameter ranges. The singularity types are thus constant in $\alpha$, and in fact, all the points in $D^{\mathrm{sing}}$ are either smooth or of type $A_k$, for $k\in \{1, \ldots, 4\}$, inside $\overline{X}_\alpha$. The vanishing orders of $h_\alpha$ at $w=0$ are thus in the range $\{1, \ldots, 5\}$, depending on the Painlevé type and the vertex in $D^{\mathrm{sing}}$. 

    Now replacing $u$ by $u/A_\alpha$ and dividing $r_\alpha= 0$ by $A_\alpha$, we get the holomorphic normal forms
    \[
        r_\alpha = uv + w^k = 0
    \] 
    for each Painlevé type on some neighbourhood of $D^{\mathrm{sing}}$, with $\{w=0\}$ still defining $D$.

    \begin{remark}
        One might wonder why we did not directly use the normal forms arising from the quoted $A_k$-type singularities. The reason is that this argument leaves the $w$-coordinate untouched, so that we retain explicit coordinate control near $D$. 
    \end{remark}

    As before, consider $F = \log(1 + \phi)$. The restriction of the second summand to the $\overline{X}_\alpha$ in a neighbourhood of $p$ can be expressed in the $(u,v)$-coordinates as
    \[
    \psi_\alpha(u, v, w) = \log\left(1 + \norm{w}^2 + \norm{y(\alpha, u,v,w)}^2 + \norm{z(\alpha, u,v, w)}^2 \right).
    \]
    As before, $\psi_\alpha$ is smooth at $(u,v,w) = 0$ for any $\alpha$. Next, consider the $\R$-action
    \[
    s \cdot (\alpha, u, v, w) = (\alpha, e^{\frac{ks}{2}}u, e^{\frac{ks}{2}}v, e^{s}w).
    \]
    Observe that $r_\alpha(s\cdot(u, v, w)) = e^{ks}r_\alpha(u,v,w)$, which implies that its generator $V$ is tangent to the $X_\alpha$. We evaluate
    \begin{align*}
        \dd F_\alpha(V) &= \ddt{s} F(\alpha, e^{\frac{ks}{2}}u, e^{\frac{ks}{2}}v, e^{s}w)\\
        &= \ddt{s} [- \log \norm{e^s w}^2] + \dd\psi_\alpha(V)\\
        &= -2 + \dd \psi_\alpha(V).
    \end{align*}
    In fact, $V = \frac{k}{2}(u \del_u + \overline{u}\del_{\overline{u}}) + \frac{k}{2}(v \del_v + \overline{v}\del_{\overline{v}}) +  w \del_w + \overline{w} \del_{\overline{w}}$. Hence,
    \begin{align*}
        \abs{\dd \psi_\alpha(V)} &\leq \norm{\dd \psi_\alpha}\norm{V}\\
        &\leq C  \sqrt{\frac{k^2}{4}\abs{u}^2 + \frac{k^2}{4}\abs{v}^2 + \abs{w}^2}. 
    \end{align*}
   Thus we achieve $\abs{\dd \psi_\alpha(V)} \leq 1$ by shrinking $\overline{U}_{\alpha_0, p}$. As before, we conclude that $\phi_\alpha$ has no critical points in a neighbourhood of $(\alpha_0, p)$, for any $\alpha_0$.

    

    

\end{proof}

\begin{lemma}\label{lem:pathconn_complement}
    For each Painlevé type, the set
    \[
    \Delta = \{ \alpha \in \mathcal{P} \mid \pi^{-1}(\alpha) \text{ is singular} \}
    \]
    has path-connected complement.
\end{lemma}
\begin{proof}
    Observe that $\Delta = \pi(\Sigma)$, where
    \[
    \Sigma = \{(\alpha, \mathbf{z}) \in \mathcal{E} \mid \del_x f_{\alpha}(\mathbf{z}) = 0, \, \del_y f_\alpha(\mathbf{z})=0, \, \del_z f_\alpha(\mathbf{z}) = 0  \}.
    \]
    We prove that $\Delta$ is a proper closed algebraic subset of $\mathcal{P}$, whence its real codimension is at least $2$. 
    
    First, note that for the closure $\overline{\Sigma}$ in $\overline{\mathcal{E}}$, we have $\overline{\Sigma} \cap \mathcal{P} \times D = \emptyset$. This can be checked using the normal forms of $\q_\alpha$ and $r_\alpha$ valid on the neighbourhoods found in the proof of \zcref{lem:bounded_critters}. Thus no affine singular point escapes to $D$ as $\alpha$ varies in $\mathcal{P}$. This implies $\overline{\pi}(\overline{\Sigma}) = \pi(\Sigma)$.

    Now $\overline{\pi}: \overline{\mathcal{E}} \to \mathcal{P}$ is projective, so that the image $\overline{\pi}(\overline{\Sigma})$ is closed algebraic. To see that it is proper, observe that smooth fibres exist for all Painlevé types. Since $\mathcal{P} = \C^k \times (\C^*)^l$ is connected and smooth, the complement of $\Delta$ is path-connected. 
\end{proof}

In the following subsections, we list the expression for the affine family $\mathbf{M}_X$ for each $X \in \mathfrak{P}$ (except for $X=I$, which was treated in \cite[Section 4.1]{cm_lefs}), along with our choice of parameter values, Lefschetz bifibration $(\pi, \rho)$, and regular values of $\pi$ and $\rho$ serving as the base point from which vanishing paths emanate, and compute the resulting handlebody. We use the following

\begin{notation}
    \begin{itemize}
        \item For diagrams in $\C$ of matching paths, red symbols correspond to critical values of $\rho$. It turns out that in all cases, $\rho$ is a branched cover of degree $3$; when interpreted as a Lefschetz fibration, its vanishing cycles are $0$-dimensional spheres in the chosen fibre. Critical values with the same symbol have the same vanishing cycle.
        \item When performing simplifications of the computed handlebody, we let \textbf{C} denote handle cancellation, \textbf{Ri} a Reidemeister move, \textbf{G} a Gompf move, and \textbf{S} a handle slide.
    \end{itemize}
\end{notation}

\addtocontents{toc}{\SkipTocEntry}

\subsection{The PII moduli space}
\label{subsec:P2-calculations}





The PII moduli space is characterized by
\[
    \mathbf{M}_{II} = \{xyz - x - \alpha y - z + \alpha  + 1 = 0, \qquad \alpha \in \C^* \}
\]

Choosing $\alpha=-2$ yields the affine variety $M_{II} = \{xyz -x + 2y -z - 1 = 0\}$. We endow it with the Lefschetz bifibration $(\pi, \rho) = (0.3x + 0.1y+z, \, 3 x + 3/4 z)$. The base point for the distinguished bases of vanishing paths is $0$ for both $\pi$ and $\rho$. \zcref{fig:p2-matching} shows the outcome of \zcref{matching_cycle_construction}.

\begin{figure}[htbp]
    \centering
    \def\svgwidth{\textwidth}
    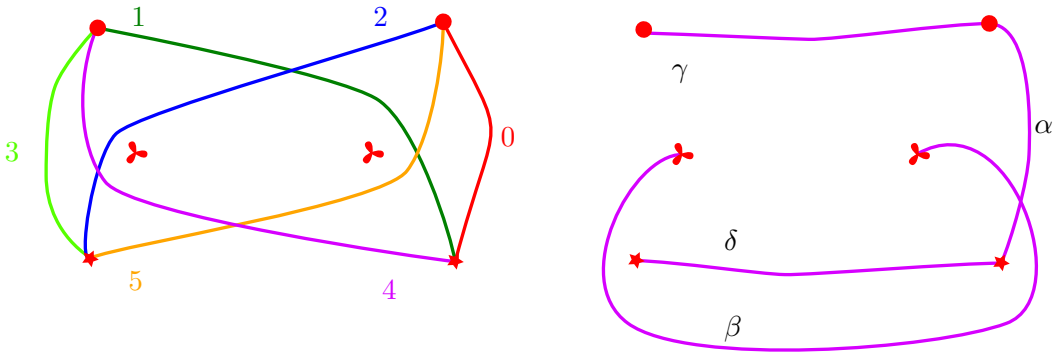

    \caption{Left: the collection of matching paths of $\rho$ induced by $\pi$. Right: A basis of matching paths for $\rho$.}
    \label{fig:p2-matching}
\end{figure}

One can factor the paths on the left in terms of the paths on the right of \zcref{fig:p2-matching} as below (up to V-moves), yielding expressions for the vanishing cycles $C_i$:
\begin{align*}
    C_{0} & = \alpha,                                           \\
    C_{1} & = \tau_\gamma^{-1}\alpha,                           \\
    C_{2} & = \tau_\beta\alpha,                                 \\
    C_{3} & = \tau_\gamma^{-1}\tau_\beta\alpha,                 \\
    C_{4} & = \tau_\delta^{-1}\tau_\gamma^{-1}\tau_\beta\alpha, \\
    C_{5} & = \tau_\delta\alpha.                                \\
\end{align*}

Next, we perform the Hurwitz move $\mathcal{H}_4^{-1}$. This has the effect of replacing $C_3$ with $\tau_{C_3}^{-1}(C_4)$, which one draws directly to see that it yields $\delta$. The resulting expressions for the vanishing cycles are
\begin{align*}
    C_{0} & = \alpha,                           \\
    C_{1} & = \tau_\gamma^{-1}\alpha,           \\
    C_{2} & = \tau_\beta\alpha,                 \\
    C_{3} & = \delta,                           \\
    C_{4} & = \tau_\gamma^{-1}\tau_\beta\alpha, \\
    C_{5} & = \tau_\delta\alpha.                \\
\end{align*}
The front from this data and subsequent simplifications leading to the diagram in \zcref{fig:main_results} is shown in \zcref{fig:p2_simplify}.

\begin{figure}[htbp]
    \centering
    \def\svgwidth{\textwidth}
    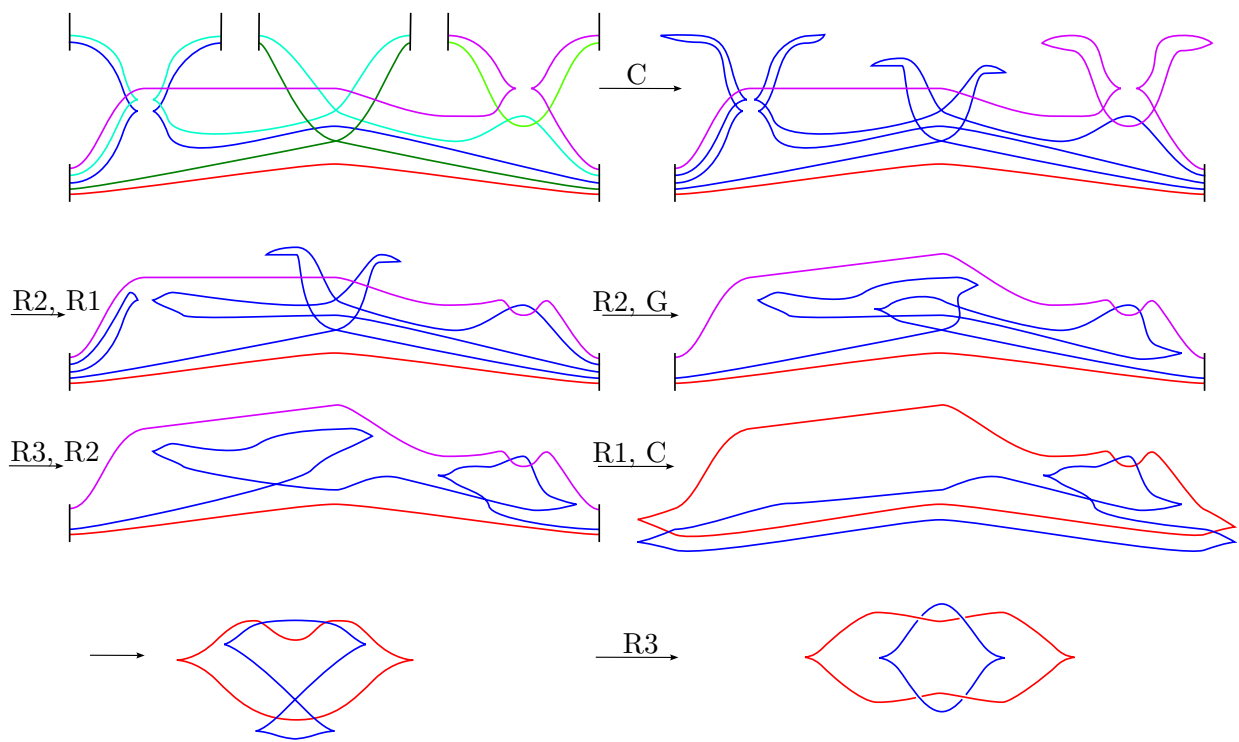
    \caption{Legendrian front and simplifications for PII.}
    \label{fig:p2_simplify}
\end{figure}

\addtocontents{toc}{\SkipTocEntry}

\subsection{The PII(FN) moduli space}
\label{sec:p2fn}

The PII(FN) moduli space is characterized by
\[
   \mathbf{M}_{II(FN)} = \{ xyz + x - y + z + s = 0, \quad s \in \C^*\}.
\]
For the choice of $s = 1$, this yields the affine equation $M_{II(FN)} = \{xyz + x-y-z+1 = 0 \}$. We endow it with the Lefschetz bifibration $(\pi, \rho) = (-3x - y -2z, 3x + 1/2 z)$. The base point for the distinguished bases of vanishing paths is $0$ for both $\pi$ and $\rho$. \zcref{fig:p2fn-matching} shows the outcome of \zcref{matching_cycle_construction}.

\begin{figure}[htbp]
    \centering
    \def\svgwidth{\textwidth}
    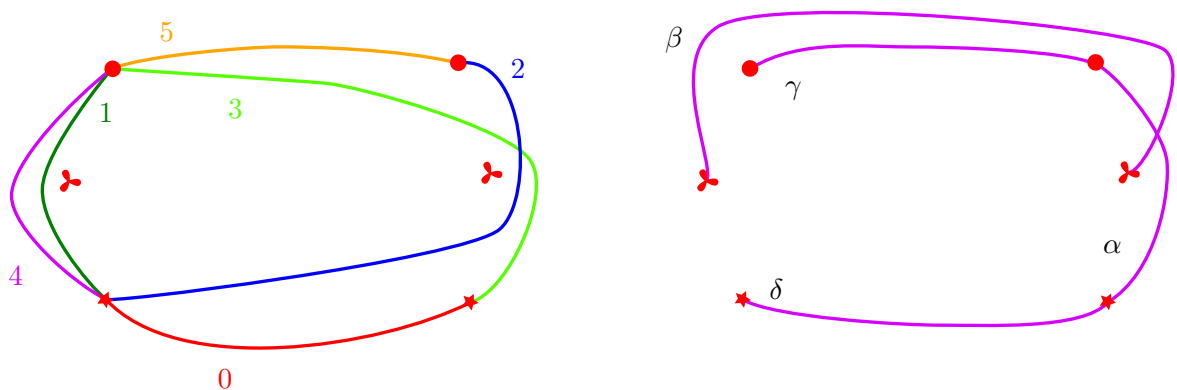

    \caption{Left: the collection of matching paths of $\rho$ induced by $\pi$. Right: A basis of matching paths for $\rho$.}
    \label{fig:p2fn-matching}
\end{figure}

One can factor the paths on the left in terms of the paths on the right of \zcref{fig:p2fn-matching} as below (up to V-moves), yielding expressions for the vanishing cycles $C_i$:
\begin{align*}
    C_{0} & = \delta,                       \\
    C_{1} & = \tau_\delta\tau_\alpha \beta, \\
    C_{2} & = \tau_\delta\alpha,            \\
    C_{3} & = \tau_\gamma^{-1}\alpha,       \\
    C_{4} & = \tau_\delta\tau_\alpha \beta, \\
    C_{5} & = \gamma.
\end{align*}
We perform the Hurwitz moves $\mathcal{H}_1^{-1}$ and $\mathcal{H}_2^{-1}$ (the subscript is relative to the basis after performing the preceding Hurwitz move). The resulting expressions for the vanishing cycles are
\begin{align*}
    C_{0} & = \tau_\alpha \beta,            \\
    C_{1} & = \alpha,                       \\
    C_{2} & = \delta,                       \\
    C_{3} & = \tau_\gamma^{-1}\alpha,       \\
    C_{4} & = \tau_\delta\tau_\alpha \beta, \\
    C_{5} & = \gamma.
\end{align*}

The front from this data and subsequent simplifications leading to the diagram in \zcref{fig:main_results} is shown in \zcref{fig:p2fn_simplify}.

\begin{figure}[htbp]
    \centering
    \def\svgwidth{\textwidth}
    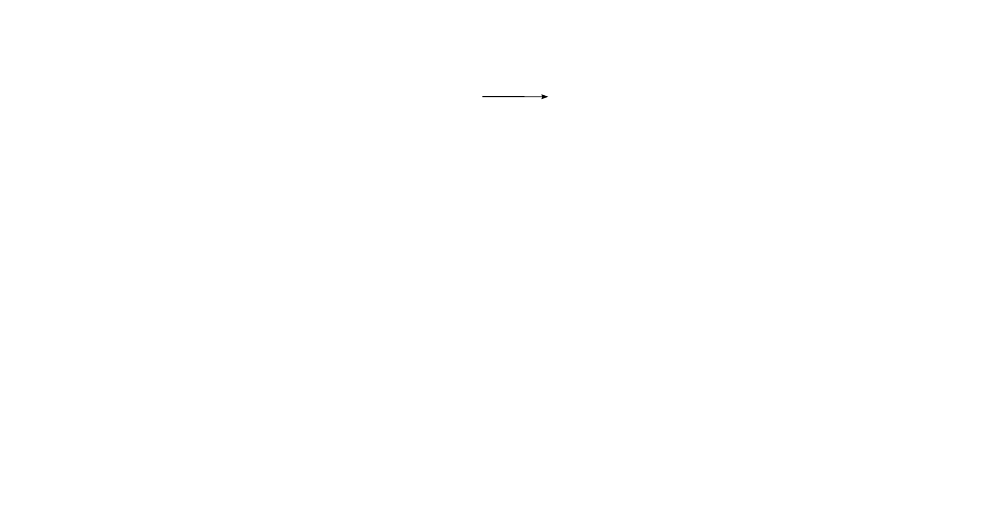
    \caption{Legendrian front and simplifications for PII(FN).}
    \label{fig:p2fn_simplify}
\end{figure}

\addtocontents{toc}{\SkipTocEntry}

\subsection{The PIII(D6) and PV(deg) moduli spaces}
\label{sec:t220}

The PIII(D6) and PV(deg) moduli spaces are characterized by
\begin{align*}
    \mathbf{M}_{III(D6)} = \{ xyz + x^2 + y^2 + (1+\alpha \beta)x + (\alpha + \beta)y + \alpha \beta & = 0, \quad \alpha, \beta \in \C^*\}, \\
    \mathbf{M}_{V(deg)} = \{ xyz + x^2 + y^2 + s_0x + s_1 y + 1                                     & = 0, \quad s_0, s_1 \in \C\}.
\end{align*}

For the choice of $\alpha = -1$ and $\beta = 1$, and $s_0=s_1 =0$, respectively, this yields the space $T_{2,2,0}$, as defined in \cite{ailsa_cusp}. We endow the space with the Lefschetz bifibration $(\pi, \rho) = (-0.7x +0.3iy+z,\, x -  0.05iy)$. The base point for the distinguished bases of vanishing paths is $0$ for both $\pi$ and $\rho$. \zcref{fig:p3d6-matching} shows the outcome of \zcref{matching_cycle_construction}.

Note that Capovilla-Searle describes the handlebody for this space in \cite{orsola_thesis}. We expect our handlebodies to agree after suitable simplifications.

\begin{figure}[htbp]
    \centering
    \def\svgwidth{\textwidth}
    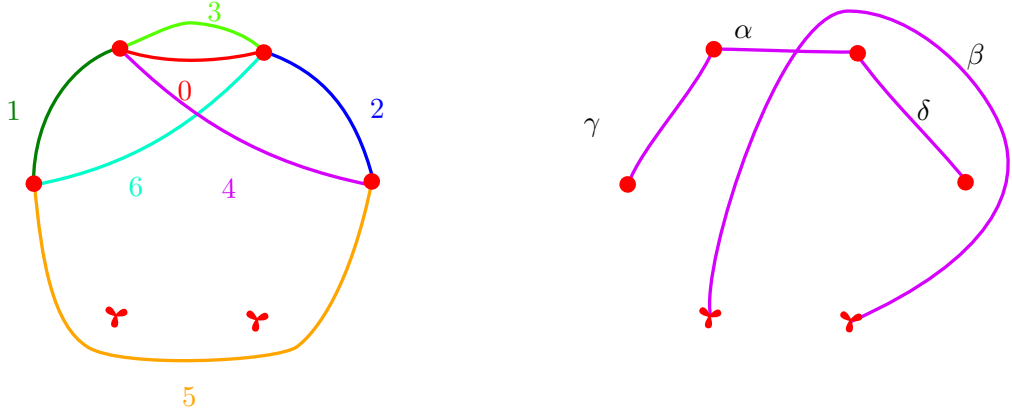

    \caption{Left: the collection of matching paths of $\rho$ induced by $\pi$. Right: A basis of matching paths for $\rho$.}
    \label{fig:p3d6-matching}
\end{figure}

One can factor the paths on the left in terms of the paths on the right of \zcref{fig:p3d6-matching} as below (up to V-moves), yielding the following expressions for the vanishing cycles.

\begin{align*}
    C_{0} & = \alpha,                           \\
    C_{1} & = \gamma,                           \\
    C_{2} & = \delta,                           \\
    C_{3} & = \alpha,                           \\
    C_{4} & = \tau_\delta\alpha,                \\
    C_{5} & = \tau_\gamma^{-1}\tau_\beta \alpha \\
    C_6   & = \tau_{\gamma}^{-1}\alpha.
\end{align*}

\begin{figure}[htbp]
    \centering
    \def\svgwidth{\textwidth}
    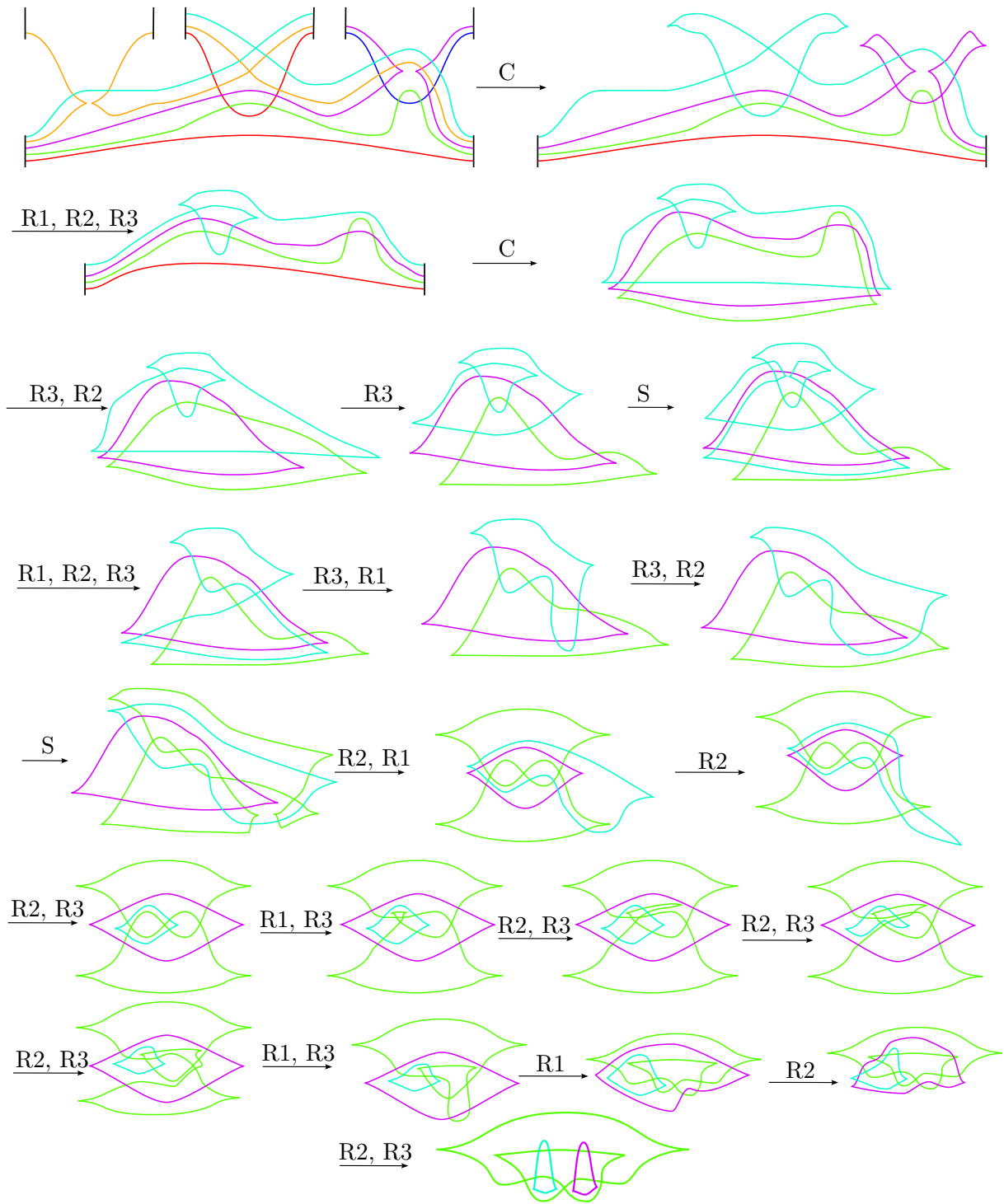
    \caption{Handlebody for PIII(D6) and PV(deg).}
    \label{fig:p3d6}
\end{figure}

\addtocontents{toc}{\SkipTocEntry}

\subsection{The PIII(D7) moduli space}
\label{sec:p3d7}

The PIII(D7) moduli space is characterized by
\[
    \mathbf{M}_{III(D7)} = \{ xyz + x^2 + y^2 + \alpha x + y = 0, \quad \alpha \in \C^*\}.
\]
For the choice of $\alpha = 1$, this yields the affine equation $M_{III(D7)} = \{xyz + x^2 + y^2 + x + y = 0 \}$. We endow it with the Lefschetz bifibration $(\pi, \rho) = (-3x + 2y + \frac{1}{2}z, 2x + 1/4 z)$. The base point for the distinguished bases of vanishing paths is $0$ for both $\pi$ and $\rho$. \zcref{fig:p3d7-matching} shows the outcome of \zcref{matching_cycle_construction}.

\begin{figure}[htbp]
    \centering
    \def\svgwidth{\textwidth}
    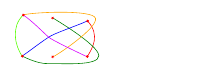

    \caption{Left: the collection of matching paths of $\rho$ induced by $\pi$. Right: A basis of matching paths for $\rho$.}
    \label{fig:p3d7-matching}
\end{figure}

One can factor the paths on the left in terms of the paths on the right of \zcref{fig:p3d7-matching} as below (up to V-moves), yielding expressions for the vanishing cycles $C_i$:
\[
    \begin{aligned}
        C_{0} & = \delta,                                     \\
        C_{1} & = \tau_{\gamma}^{-1}\tau_{\beta}\alpha,       \\
        C_{2} & = \tau_{\delta}^{-1}\tau_{\gamma}^{-1}\alpha, \\
        C_{3} & = \gamma,                                     \\
        C_{4} & = \alpha,                                     \\
        C_{5} & = \tau_{\delta}^{-1}\tau_{\beta}^{-1}\alpha.
    \end{aligned}
    \xrightarrow{\mathcal{H}_3 + \mathcal{H}_2 + C + \mathcal{H}_1}\quad
    \begin{aligned}
        C_{0} & = \delta,                   \\
        C_{1} & = \tau_{\beta}^{-1}\alpha,  \\
        C_{2} & = \gamma,                   \\
        C_{3} & = \tau_\beta \alpha,        \\
        C_{4} & = \tau_{\delta}^{-1}\alpha, \\
        C_{5} & = \alpha.
    \end{aligned}
\]

We perform the Hurwitz moves $\mathcal{H}_3$ and $\mathcal{H}_2$ (the subscript is relative to the basis after performing the preceding Hurwitz move); Perform the cyclic permutation sending $C_5$ to the top; and lastly perform the Hurwitz move $\mathcal{H}_1$.

The front from this data and subsequent simplifications leading to the diagram in \zcref{fig:main_results} is shown in \zcref{fig:p3d7_simplify}.

\begin{figure}[htbp]
    \centering
    \def\svgwidth{1.1\textwidth}
    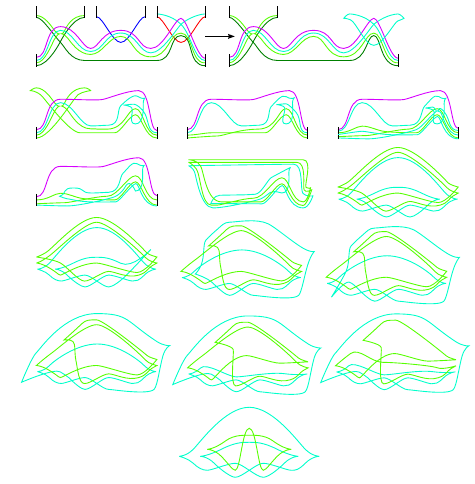
    \caption{Legendrian front and simplifications for PIII(D7). S denotes a handle slide, and the arrow without a label is an ambient isotopy.}
    \label{fig:p3d7_simplify}
\end{figure}

\addtocontents{toc}{\SkipTocEntry}

\subsection{The PIII(D8) moduli space}
\label{sec:p3d8}

Note that \cite{saito_vput} contains a description of the PIII(D8) moduli space which was incorrect, and the right expression was given by Szabó in \cite[Section 4.7]{szabo2021perversity}. The PIII(D8) moduli space is accordingly characterized by
\begin{equation}\label{eq:p3d8}
	\mathbf{M}_{III(D8)} = \{ xyz + x^{2} + y^{2} -y = 0\}.
\end{equation}
We endow the resulting variety with the Lefschetz bifibration $(\pi, \rho) = (-ix + 3y -z, -1.5 y +  z)$. The base point for the distinguished bases of vanishing paths is $0$ for $\pi$ and $1$ for $\rho$. \zcref{fig:p3d8-matching} shows the outcome of \zcref{matching_cycle_construction}.

\begin{figure}[htbp]
	\centering
	\def\svgwidth{\textwidth}
	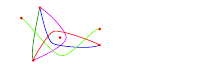

	\caption{Left: the collection of matching paths of $\rho$ induced by $\pi$. Right: A basis of matching paths for $\rho$.}
	\label{fig:p3d8-matching}
\end{figure}

One can factor the paths on the left in terms of the paths on the right of \zcref{fig:p3d8-matching} as below (up to V-moves), yielding expressions for the vanishing cycles $C_i$:

\[
	\begin{aligned}
		C_{0} & = \tau_\alpha^{-1} \tau_\gamma^2 \tau_\delta^{-1} \alpha, \\
		C_{1} & = \alpha,                                                 \\
		C_{2} & = \tau_\delta^{-1} \alpha,                                \\
		C_{3} & = \beta,                                                  \\
		C_{4} & = \tau_\gamma^2 \alpha,                                   \\
	\end{aligned}
	\xrightarrow{\mathcal{H}_1} \quad
	\begin{aligned}
		C_{0} & = \alpha,                                \\
		C_{1} & = \tau_\gamma^{2}\tau_\delta^{-1}\alpha, \\
		C_{2} & = \tau_\delta^{-1}\alpha,                \\
		C_{3} & = \beta,                                 \\
		C_{4} & = \tau_\gamma^2 \alpha,                  \\
	\end{aligned}
\]

We perform the single Hurwitz move $\mathcal{H}_1$. The front from this data and subsequent simplifications leading to the diagram in \zcref{fig:main_results} is shown in \zcref{fig:p3d8_simplify}.

\begin{figure}[htbp]
	\centering
	\def\svgwidth{\textwidth}
	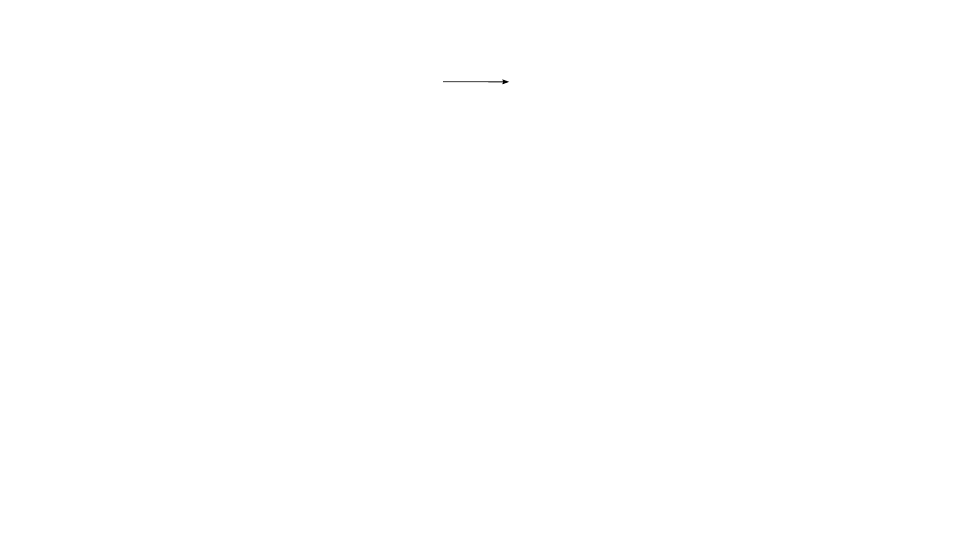
	\caption{Legendrian front and simplifications for PIII(D8).}
	\label{fig:p3d8_simplify}
\end{figure}

\subsubsection{Computing $H_1$}
The Legendrian front obtained in \zcref{fig:p3d8_simplify} contains an uncancellable $1$-handle; Writing $M := M_{III(D8)}$, a cellular homology computation readily yields that $H_1(M;\Z) \cong \Z/2\Z$. This is unexpected when compared to the other Painlevé spaces, and may raise the question whether there has been a mistake along the rather delicate computation process. However, we recover this homology group using results from \cite{szabo2021perversity} and \cite{szabo_nemethi_geom_pw}. All (co--)homology groups in the following are taken with $\Z$-coefficients.

Consider the compactification $\overline{M}$ obtained by homogenizing the defining  \zcref{eq:p3d8}. Note that $\overline{M}$ contains a singularity in the complement of $M$; Let $X \to \overline{M}$ denote its minimal resolution. We will use the sequence of the pair $(X,M)$.

\begin{lemma}\label{lem:lefschetz_duality}
	We have $H_k(X,M) = H^{4-k}(D)$, where $D:= X - M$ is the compactification divisor.
\end{lemma}
\begin{proof}
	By excision, $H_k(X, M) \cong H_k(X - (X - N(D)), M - (X - N(D))) = H_k(N(D), N(D) - D)$, where $N(D)$ is a small neighbourhood of $D$ in $X$. $N(D)- D$ retracts to $\del N(D)$, and we apply Lefschetz duality to obtain $H_k(N(D), \del N(D)) \cong H^{4-k}(N(D)) \cong H^{4-k}(D)$.
\end{proof}

According to \cite[Lemma 1]{szabo_nemethi_geom_pw}, $D$ consists of seven $\P^1$'s, each intersecting two others transversely in one point each, forming a cycle of length $7$. Three consecutive components have self-intersection $-1$, and the remaining four have self-intersection $-2$ (\zcref{fig:p3d8_divisor}).

\begin{figure}[htbp]
	\centering
	\def\svgwidth{0.5\textwidth}
\begingroup%
  \makeatletter%
  \providecommand\color[2][]{%
    \errmessage{(Inkscape) Color is used for the text in Inkscape, but the package 'color.sty' is not loaded}%
    \renewcommand\color[2][]{}%
  }%
  \providecommand\transparent[1]{%
    \errmessage{(Inkscape) Transparency is used (non-zero) for the text in Inkscape, but the package 'transparent.sty' is not loaded}%
    \renewcommand\transparent[1]{}%
  }%
  \providecommand\rotatebox[2]{#2}%
  \newcommand*\fsize{\dimexpr\f@size pt\relax}%
  \newcommand*\lineheight[1]{\fontsize{\fsize}{#1\fsize}\selectfont}%
  \ifx\svgwidth\undefined%
    \setlength{\unitlength}{95.94991849bp}%
    \ifx\svgscale\undefined%
      \relax%
    \else%
      \setlength{\unitlength}{\unitlength * \real{\svgscale}}%
    \fi%
  \else%
    \setlength{\unitlength}{\svgwidth}%
  \fi%
  \global\let\svgwidth\undefined%
  \global\let\svgscale\undefined%
  \makeatother%
  \begin{picture}(1,0.32878372)%
    \lineheight{1}%
    \setlength\tabcolsep{0pt}%
    \put(0,0){\includegraphics[width=\unitlength,page=1]{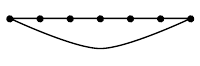}}%
    \put(0.08026722,0.26325316){\color[rgb]{0,0,0}\makebox(0,0)[lt]{\lineheight{1.25}\smash{\begin{tabular}[t]{l}-1\\\end{tabular}}}}%
    \put(0.23806016,0.26325316){\color[rgb]{0,0,0}\makebox(0,0)[lt]{\lineheight{1.25}\smash{\begin{tabular}[t]{l}-1\\\end{tabular}}}}%
    \put(0.39585313,0.26325316){\color[rgb]{0,0,0}\makebox(0,0)[lt]{\lineheight{1.25}\smash{\begin{tabular}[t]{l}-1\\\end{tabular}}}}%
    \put(0.55364612,0.26325316){\color[rgb]{0,0,0}\makebox(0,0)[lt]{\lineheight{1.25}\smash{\begin{tabular}[t]{l}-2\end{tabular}}}}%
    \put(0.71143909,0.26325316){\color[rgb]{0,0,0}\makebox(0,0)[lt]{\lineheight{1.25}\smash{\begin{tabular}[t]{l}-2\end{tabular}}}}%
    \put(0.86923203,0.26325316){\color[rgb]{0,0,0}\makebox(0,0)[lt]{\lineheight{1.25}\smash{\begin{tabular}[t]{l}-2\end{tabular}}}}%
    \put(0.4904771,0.03114536){\color[rgb]{0,0,0}\makebox(0,0)[lt]{\lineheight{1.25}\smash{\begin{tabular}[t]{l}-2\end{tabular}}}}%
  \end{picture}%
\endgroup%

	\caption{Plumbing graph of the compactification divisor $D = X - M$.}
	\label{fig:p3d8_divisor}
\end{figure}

Using the fact that $H_1(X) = H_3(X) = 0$ (\cite[Lemma 5]{szabo2021perversity}), consider the piece of the sequence of the pair $(X,M)$ given by

\[
	\begin{tikzcd}
		0=H_3(X) \arrow[r] & H^1(D) \arrow[r, "\imath"] & H_2(M) \arrow[r, "\jmath"] & H_2(X) \arrow[r, "\Phi"] & H^2(D) \arrow[r] & H_1(M) \arrow[r] & H_1(X) = 0.\\
		&\Z \arrow[u, phantom,sloped,"\cong"] & \Z \arrow[u, phantom,sloped,"\cong"] & \Z^7\arrow[u, phantom,sloped,"\cong"] & \Z^7 \arrow[u, phantom,sloped,"\cong"] & &
	\end{tikzcd}
\]

Note that we have substituted the terms $H_k(X, M) \cong H^{4-k}(D)$ according to \zcref{lem:lefschetz_duality}. The identifications for $H^k(D)$ are immediate from the description given above, and $H_2(X) = \Z^7$ follows from the facts that (1) the minimal resolution of a singular cubic is diffeomorphic to a smooth cubic, and (2) any smooth projective cubic can be obtained from $\P^2$ by  performing a blow-up at six points (representatives for the generators of $H_2(X)$ are given by the six exceptional divisors and the pullback of the class of a hyperplane); That $H_2(M) \cong \Z$ is a consequence of \cite[Proposition 2]{szabo2021perversity}.

Denote by $D_i$, $i = 1, \ldots, 7$, the irreducible components of $D$. The map $\Phi$ sends a class $[\alpha] \in H_2(X)$ to the functional $\alpha.D_i [D_i]^* \in \mathrm{Hom}(H_2(D), \Z) \cong H^2(D)$, and we wish to determine $\mathrm{coker}(\Phi)$.

The intersection pairing is a nondegenerate bilinear form on $H_2(X)$, making it into an unimodular lattice which we denote by $\Lambda$. Let $L$ be the sublattice generated by the classes $[D_i] \in H_2(X)$, and let $j: L \hookrightarrow L^\vee$ be the natural inclusion map $j_l(l') = l.l'$, for $l,l' \in L$.

The map $\Phi$ can be identified with $\Lambda \to L^\vee$, and one sees that $j = \Phi\vert_L$. Denote by $A_L := L^\vee / j(L)$ the discriminant group of $L$, and set $H:= \Phi(\Lambda) / j(L)$. Then $\mathrm{coker}(\Phi)$ is isomorphic to $A_L/H$, since
\[
	\mathrm{coker}(\Phi) = L^\vee/\Phi(\Lambda) \cong (L^\vee / j(L))/(\Phi(\Lambda) / j(L)) = A_L/H.
\]

By definition,

\[
	[A_L: H] = \frac{[L^\vee:j(L)]}{[\Phi(\Lambda):j(L)]};
\]

All indices appearing in this formula will turn out to be finite. Note that, also by definition, $[L^\vee: j(L)] = |A_L|$.

\begin{lemma}
	$[\Phi(\Lambda): j(L)] = [\Lambda: L].$
\end{lemma}
\begin{proof}
	Note first of all that $\Phi$ is injective: this follows from the fact that $\jmath$ in the long exact sequence is the zero map ($\imath$ is injective, and if $\mathrm{im}(\imath) \subset \Z$ were a subgroup of index $k>1$, then $\jmath(kv) = k \jmath(v) = 0$ for all $v \in H_2(M)$, implying that $\jmath \equiv 0$ since $H_2(X)$ has no torsion. But then $H_2(M) = \ker  \jmath$, so $k=1$). Hence $\jmath$ is the zero map, and $\Phi$ is injective.

	In the diagram
	\[
		\begin{tikzcd}
			\Lambda \arrow[d] \arrow[r, "\Phi"] & \Phi(\Lambda) \arrow[d]\\
			\Lambda/L \arrow[r] & \Phi(\Lambda)/j(L),
		\end{tikzcd}
	\]
	injectivity of $\Phi$ implies the top arrow is an isomorphism; The bottom horizontal map is well-defined because $\Phi\vert_L = j$, evidently surjective, and its kernel is $\Phi^{-1}(j(L))/L = \{0\}$. This implies the lemma.
\end{proof}

It remains to compute $[A_L: H]$.

\begin{lemma}
	We have $[\Lambda: L] = 2$ and $|A_L| = 4$.
\end{lemma}
\begin{proof}
	To compute $[\Lambda: L]$, form the Gram matrix $Q$ of $L$, which is given by the intersection form of $D$; it can be read off from \zcref{fig:p3d8_divisor} to be
	\[
		Q = \begin{pmatrix}
			-1 & 1  & 0  & 0  & 0  & 0  & 1  \\
			1  & -1 & 1  & 0  & 0  & 0  & 0  \\
			0  & 1  & -1 & 1  & 0  & 0  & 0  \\
			0  & 0  & 1  & -2 & 1  & 0  & 0  \\
			0  & 0  & 0  & 1  & -2 & 1  & 0  \\
			0  & 0  & 0  & 0  & 1  & -2 & 1  \\
			1  & 0  & 0  & 0  & 0  & 1  & -2
		\end{pmatrix}.
	\]

	Its determinant is $4$; This implies that the $[D_i]$ form a collection of seven $\Z$-linearly independent vectors in $H_2(X) \cong \Z^7$, so that $[\Lambda:L]$ is finite. In this case,
	\[
		[\Lambda: L]^2 = \frac{\mathrm{disc}(L)}{\mathrm{disc}(\Lambda)},
	\]
	(see e.g. \cite{lattice}). $\Lambda$ being unimodular, we see that the $[\Lambda:L] = 2$ (the discriminant $\mathrm{disc}(L)$ is the absolute value of the determinant of the Gram matrix).

	The discriminant group $A_L$ can be computed from the Smith normal form of $Q$. Performing this computation yields
	\[
		A_L \cong \Z/2\Z \oplus \Z/2\Z.
	\]
\end{proof}

\begin{corollary}
	$H_1(M_{III(D8)};\Z) \cong \Z/2\Z$.
\end{corollary}
\begin{proof}
	The preceding lemma implies that $[A_L: H] = 2$, and thus $\mathrm{coker}(\Phi) \cong \Z/2\Z$.
\end{proof}

\addtocontents{toc}{\SkipTocEntry}

\subsection{The PIV moduli space}

\label{subsec:PIV-moduli-space}



The PIV moduli space is characterized by
\begin{equation}
   \mathbf{M}_{IV} = \{ xyz + x^2 - (s_2^2+s_1s_2)x -s_2^2y -s_2^2z + s_2^2 + s_1 s_2^3 = 0, \quad s_1 \in \C, \, s_2 \in \C^*\}.
\end{equation}
Choosing $s_2 = -s_1 = \frac{1}{2}$ yields the affine variety $M_{IV} = \{xyz + x^2 -\frac{1}{4}y - \frac{1}{4}z + \frac{3}{16} = 0\}$. We endow it with the Lefschetz bifibration $(\pi, \rho) = (3x - 0.5y+2/3z, -4.5 x +  z)$. The base point for the distinguished bases of vanishing paths is $0$ for both $\pi$ and $\rho$. \zcref{fig:p4-matching} shows the outcome of \zcref{matching_cycle_construction}.

\begin{figure}[htbp]
    \centering
    \def\svgwidth{\textwidth}
    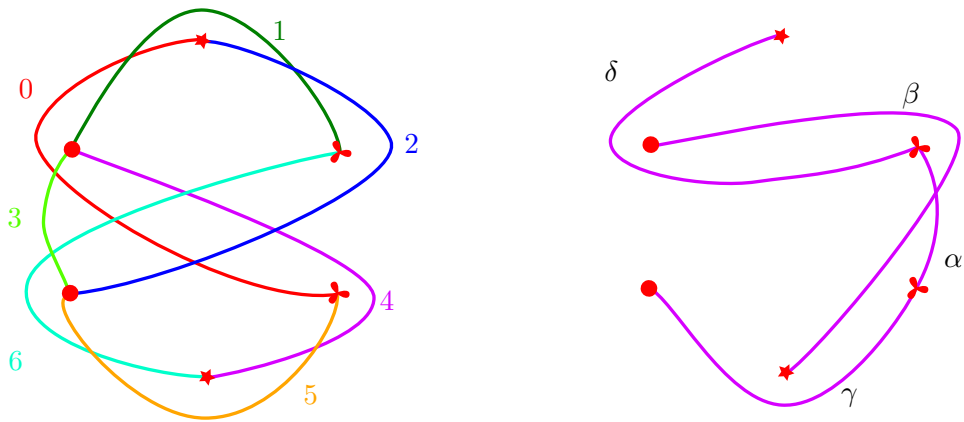

    \caption{Left: the collection of matching paths of $\rho$ induced by $\pi$. Right: A basis of matching paths for $\rho$.}
    \label{fig:p4-matching}
\end{figure}

One can factor the paths on the left in terms of the paths on the right of \zcref{fig:p4-matching} as below (up to V-moves), yielding expressions for the vanishing cycles $C_i$:
\[
    \begin{aligned}
        C_{0} & = \tau_\delta^{-1} \alpha,                        \\
        C_{1} & = \delta,                                         \\
        C_{2} & = \tau_\beta^{-1} \tau_\gamma \tau_\delta \alpha, \\
        C_{3} & = \tau_\beta^{-1}\tau_\gamma \alpha,              \\
        C_{4} & = \tau_\beta^{-1}\alpha                           \\
        C_{5} & = \gamma,                                         \\
        C_{6} & = \tau_\gamma \alpha
    \end{aligned}
    \xrightarrow{\mathcal{H}_1+\mathcal{H}_6^{-1}} \quad
    \begin{aligned}
        C_{0} & = \delta,                                         \\
        C_{1} & = \alpha,                                         \\
        C_{2} & = \tau_\beta^{-1} \tau_\gamma \tau_\delta \alpha, \\
        C_{3} & = \tau_\beta^{-1}\tau_\gamma \alpha,              \\
        C_{4} & = \tau_\beta^{-1}\alpha                           \\
        C_{5} & = \alpha                                          \\
        C_{6} & = \gamma,                                         \\
    \end{aligned}
    \xrightarrow{\mathcal{H}_3} \quad
    \begin{aligned}
        C_{0} & = \delta,                            \\
        C_{1} & = \alpha,                            \\
        C_{2} & = \tau_\beta^{-1}\tau_\gamma \alpha, \\
        C_{3} & = \delta,                            \\
        C_{4} & = \tau_\beta^{-1}\alpha              \\
        C_{5} & = \alpha                             \\
        C_{6} & = \gamma,                            \\
    \end{aligned}
\]
We first perform the Hurwitz moves $\mathcal{H}_1$ and $\mathcal{H}_6^{-1}$. We moreover apply $\mathcal{H}_3$, which has the effect of replacing $C_2$ with $C_3$ and $C_3$ with $\tau_{C_3}(C_2)$; One checks directly that $\tau_{C_3}(C_2) = \delta$.

The front from this data and subsequent simplifications leading to the diagram in \zcref{fig:main_results} is shown in \zcref{fig:p4_simplify}.

\begin{figure}[htbp]
    \centering
    \def\svgwidth{\textwidth}
    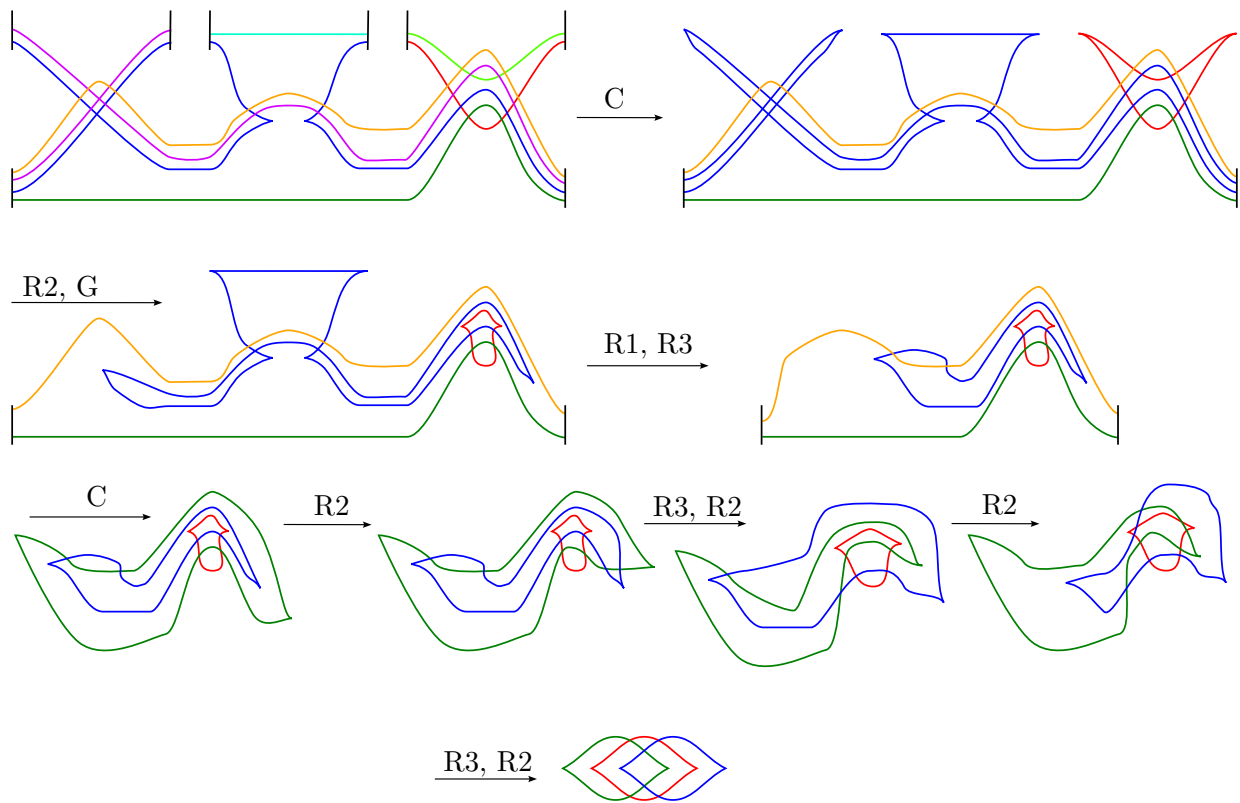
    \caption{Legendrian front and simplifications for PIV.}
    \label{fig:p4_simplify}
\end{figure}

\addtocontents{toc}{\SkipTocEntry}

\subsection{The PV moduli space}
\label{subsec:PV-moduli-space}

The PV moduli space is characterized by
\begin{equation}
	\label{eq:PV-moduli-space}
	\mathbf{M}_{V} = \{ xyz + x^2 +y^2 -(s_1 + s_2s_3)x - (s_2 + s_1 s_3)y - s_3 z + s_3^2 + s_1s_2s_3 = 0, \quad s_1,s_2 \in \C, \, s_3 \in \C^*\}.
\end{equation}

Choosing $s_1=i, s_2=-i, s_3=1$ yields the affine variety $M_{V} = \{ xyz + x^2 + y^2 - z + 3 = 0$\}. We endow it with the Lefschetz bifibration $(\pi, \rho) = (3x -1y-2z, 1 x +  1.5 z)$. The base point for the distinguished bases of vanishing paths is $0$ for both $\pi$ and $\rho$. \zcref{fig:p5-matching} shows the outcome of \zcref{matching_cycle_construction}.

\begin{figure}[htbp]
	\centering
	\def\svgwidth{0.7\textwidth}
	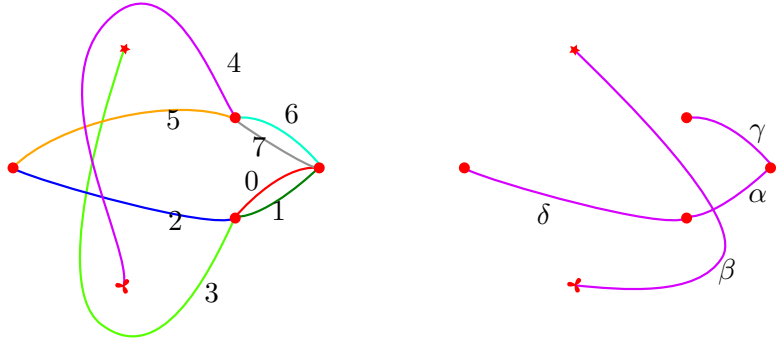
	\caption{Left: the collection of matching paths of $\rho$ induced by $\pi$. Right: A basis of matching paths for $\rho$.}
	\label{fig:p5-matching}
\end{figure}

One can factor the paths on the left in terms of the paths on the right of \zcref{fig:p5-matching} as below (up to V-moves).
\[
	\begin{aligned}
		C_{0} & = \alpha,                                \\
		C_{1} & = \alpha,                                \\
		C_{2} & = \delta,                                \\
		C_{3} & = \beta,                                 \\
		C_{4} & = \tau_\gamma^{-1}\tau^{-1} _\alpha\beta \\
		C_{5} & = \tau_\delta \tau_\gamma^{-1} \alpha    \\
		C_{6} & = \gamma,                                \\
		C_7   & =  \gamma.
	\end{aligned}
	\xrightarrow{\mathcal{H}_6 + \mathcal{H}_5}\quad
	\begin{aligned}
		C_0   & = \alpha,                                    \\
		C_{1} & = \alpha,                                    \\
		C_{2} & = \delta,                                    \\
		C_{3} & = \beta,                                     \\
		C_{4} & = \gamma                                     \\
		C_{5} & = \tau_\delta^{-1}\beta = \tau_\beta \alpha, \\
		C_{6} & = \tau_\delta,                               \\
		C_{7} & = \gamma.
	\end{aligned}
\]

The front from this data and subsequent simplifications leading to the diagram in \zcref{fig:main_results} is shown in \zcref{fig:p5_simplify}.

\begin{figure}[htbp]
	\centering
	\def\svgwidth{\textwidth}
\begingroup%
  \makeatletter%
  \providecommand\color[2][]{%
    \errmessage{(Inkscape) Color is used for the text in Inkscape, but the package 'color.sty' is not loaded}%
    \renewcommand\color[2][]{}%
  }%
  \providecommand\transparent[1]{%
    \errmessage{(Inkscape) Transparency is used (non-zero) for the text in Inkscape, but the package 'transparent.sty' is not loaded}%
    \renewcommand\transparent[1]{}%
  }%
  \providecommand\rotatebox[2]{#2}%
  \newcommand*\fsize{\dimexpr\f@size pt\relax}%
  \newcommand*\lineheight[1]{\fontsize{\fsize}{#1\fsize}\selectfont}%
  \ifx\svgwidth\undefined%
    \setlength{\unitlength}{204.81448196bp}%
    \ifx\svgscale\undefined%
      \relax%
    \else%
      \setlength{\unitlength}{\unitlength * \real{\svgscale}}%
    \fi%
  \else%
    \setlength{\unitlength}{\svgwidth}%
  \fi%
  \global\let\svgwidth\undefined%
  \global\let\svgscale\undefined%
  \makeatother%
  \begin{picture}(1,0.62851883)%
    \lineheight{1}%
    \setlength\tabcolsep{0pt}%
    \put(0,0){\includegraphics[width=\unitlength,page=1]{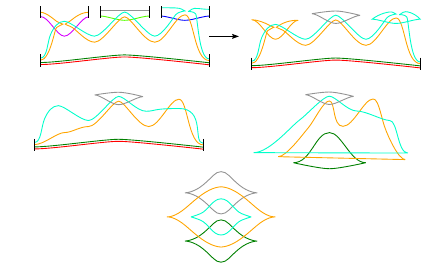}}%
    \put(0.50726098,0.55108247){\color[rgb]{0,0,0}\makebox(0,0)[lt]{\lineheight{1.25}\smash{\begin{tabular}[t]{l}$\text{C}$\end{tabular}}}}%
    \put(0,0){\includegraphics[width=\unitlength,page=2]{p5_simpler.pdf}}%
    \put(0.01415951,0.3460507){\color[rgb]{0,0,0}\makebox(0,0)[lt]{\lineheight{1.25}\smash{\begin{tabular}[t]{l}$\text{R1, R2}$\end{tabular}}}}%
    \put(0,0){\includegraphics[width=\unitlength,page=3]{p5_simpler.pdf}}%
    \put(0.51300922,0.33196879){\color[rgb]{0,0,0}\makebox(0,0)[lt]{\lineheight{1.25}\smash{\begin{tabular}[t]{l}$\text{C}$\end{tabular}}}}%
    \put(0,0){\includegraphics[width=\unitlength,page=4]{p5_simpler.pdf}}%
    \put(0.27752668,0.1267661){\color[rgb]{0,0,0}\makebox(0,0)[lt]{\lineheight{1.25}\smash{\begin{tabular}[t]{l}$\text{R2}$\end{tabular}}}}%
  \end{picture}%
\endgroup%

	\caption{Legendrian front and simplifications for PV.}
	\label{fig:p5_simplify}
\end{figure}

\addtocontents{toc}{\SkipTocEntry}

\subsection{The PVI moduli space}
\label{sec:t222}

The PVI moduli space is characterized by
\[
    \mathbf{M}_{VI} = \{ xyz + x^2 + y^2 + z^2 - s_1x - s_2y - s_3z + s_4 = 0\},
\]
where $s_i = a_ia_4 + a_ja_k$ for $(i,j,k)$ a cyclic permutation of $(1,2,3)$, and $s_4 = a_1a_2a_3a_4 + a_1^2 + a_2^2 + a_3^2+a_4^2-4$ with $a_i \in \C$.

We choose $a_1=1$ and $a_i=0$ for $i>1$, yielding the affine variety $M_{VI}= \{xyz + x^2 + y^2 + z^2 - 3 = 0\}$. We endow it with the Lefschetz bifibration $(\pi, \rho) = (0.1x - 0.5iy - 0.2z, 0.9x + 0.5y)$. The base point for the distinguished bases of vanishing paths are $0$ for $\pi$ and $-15i$ for $\rho$, respectively. The induced matching paths together with a basis are shown in \zcref{fig:t222_basis}. Observe that this affine equation coincides with the space $T_{2,2,2}$ from \cite{ailsa_tori}, a handelbody diagram for which was computed in \cite{orsola_thesis}. We expect it to agree with ours after simplifications.

\begin{figure}[htbp]
    \centering
    \def\svgwidth{\textwidth}
    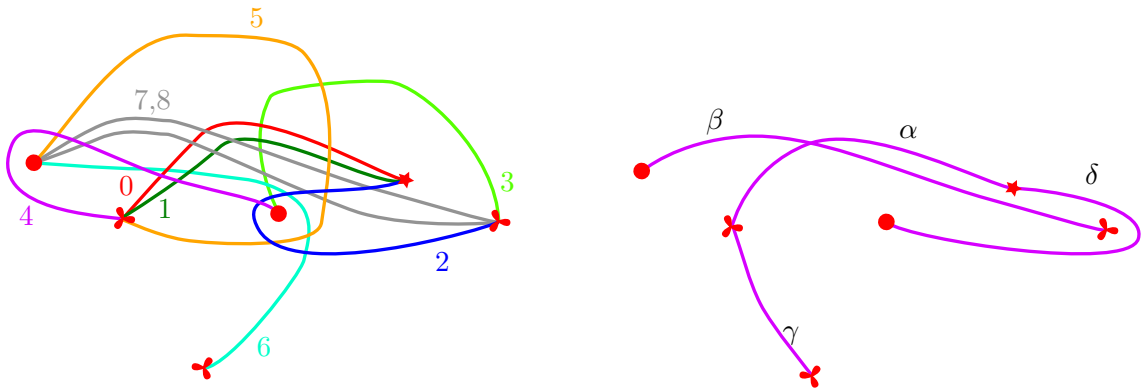
    \caption{Left: the collection of matching paths of $\rho$ induced by $\pi$. Right: A basis of matching paths for $\rho$.}
    \label{fig:t222_basis}
\end{figure}

One can factor the paths on the left in terms of the paths on the right of \zcref{fig:t222_basis} as below (up to V-moves).

\[
    \begin{aligned}
        C_{0} & = \alpha,                                         \\
        C_{1} & = \alpha,                                         \\
        C_{2} & = \delta,                                         \\
        C_{3} & = \delta,                                         \\
        C_{4} & = \tau_\delta \tau_\beta^{-1} \alpha              \\
        C_{5} & = \tau_\beta^{-1} \tau_\delta \alpha              \\
        C_{6} & = \tau_\beta \tau_\delta \tau_\gamma^{-1} \alpha, \\
        C_7   & =  \beta,                                         \\
        C_9   & = \beta.
    \end{aligned}
\]

The resulting handlebody and subsequent simplifications is shown in \zcref{fig:p6_legendrian}

\begin{figure}[htbp]
    \centering
    \def\svgwidth{\textwidth}
    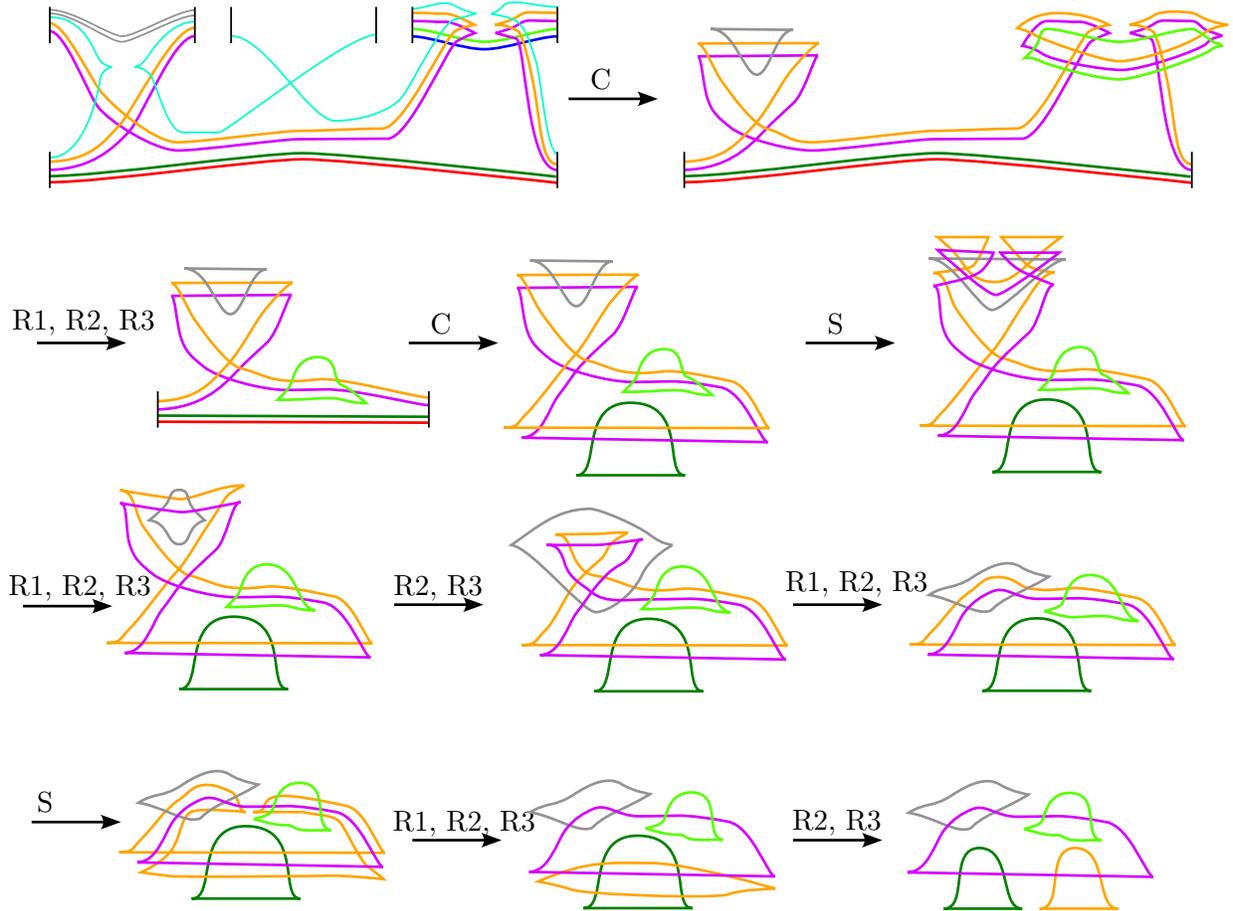
    \caption{Legendrian front and simplifications for PVI. The form presented in \zcref{fig:main_results} can be obtained by observing that the cusps of the little unknots can be moved in- or outside the purple unknot by a Reidemeister 1 move on the smooth end, isotoping the resulting crossing past the purple knot, and undoing the twist thus generated with another Reidemeister 1 move.}
    \label{fig:p6_legendrian}
\end{figure}

We also remark that $PVI$ is the unique Painlev\'e moduli space whose compactified Betti moduli space $\overline{\mathcal{M}}_{B}^{PVI}$ is smooth at infinity--see for example \cite[Section 4.1]{szabo2021perversity}.

\section{Handle attachment to the Stokes Legendrian}
\label{sec:stokes_lifts}

Recall that each of the spaces $\mathbf{M}_X$ is associated to a (possibly) irregular singularity type, which is encoded in a Legendrian projecting to a front on $\P^1 - \mathbf{p}_X$ encircling the punctures. The cardinality of $\mathbf{p}_X$ ranges from $1$ to $4$ depending on $X$, as seen in \zcref{fig:stokes_legendrians}. To compute this Legendrian, one may proceed as outlined in \cite[Example 2.5]{casals_nho} to obtain the Betti surfaces from \zcref{fig:stokes_legendrians}; the only input required for the computation are the \emph{generalized local exponents} of \cite[Section 4]{saito_vput}. The authors have written an implementation in \cite{hbds_repo}. Another, web-hosted implementation along with an excellent explanation of how to obtain the Stokes Legendrians can be found in \cite{boalch_program}.

We are interested in comparing the spaces $M_X$ (generic members of the family $\mathbf{M}_X$) to the spaces obtained by attaching handles along the Stokes Legendrians. An algorithm to determine a front of Legendrian lifts of cooriented curves in \emph{closed} surfaces is presented in \cite{acu_complements}; we extend this algorithm to the case with boundary in the next subsection.

\subsection{Legendrian lifts of curves in surfaces with boundary}
Observe that, \emph{a priori}, spaces arising from Weinstein handle attachment to Legendrians in the unit cotangent bundle of surfaces with boundary are merely Weinstein \emph{sectors}. One still has the structures of an exhausting Morse function and a Liouville vector field $(Z, \phi)$, but now there is a subset of the sublevel sets of $\phi$ to which $Z$ can be arranged to be tangent - known as the \textbf{sectorial boundary}. One can produce a unique Weinstein manifold from a sector via convex completion in the sense of \cite[Section 2.7]{gps1}.

This operation attaches to the sectorial boundary a convex end $F \times \C_{\mathrm{Re\geq 0}}$ for some Weinstein manifold $F$, and extends the structures $(Z, \phi)$ to a Weinstein manifold structure without changing them on the original sector. On the ideal boundary, which is necessarily a contact manifold with boundary, it has the effect of capping it off to produce a closed contact manifold. Let us write $T^*_cQ$ for the convex completion of the cotangent bundle of a manifold $Q$ with boundary. We now specialize to surfaces, and let $\Sigma_{g,b}$ denote a surface of genus $g$ with $b$ boundary components.

Our first lemma describes the ideal boundary $\del_\infty T^*_c \Sigma_{g,b}$.

\begin{lemma}\label{lem:ideal_sector_boundary}
	There is a contact isotopy between $\del_\infty T^*_c \Sigma_{g,b}$ and $\del_\infty (T^{\ast}\Sigma_{g,0})_1$, the ideal boundary of the subcritical part of $T^{\ast}\Sigma_{g,0}$.
\end{lemma}
\begin{proof}
		We start by considering the morsified Weinstein structure  $(T^*\Sigma_{g,0}, \lambda_f, \phi_f)$ in the sense of \cite[Section 7]{acu_complements} coming from a self-indexing Morse function $f: \Sigma_{g,0} \to \R$ with a unique index-$0$ critical point, $2g+b-1$ index-$1$ critical points, and $b$ index-$2$ critical points. This Weinstein structure is homotopic to the standard one, and has the following properties:

	\begin{itemize}
		\item The critical points of $\phi_f = \frac{1}{2}\norm{p}^2 + f: T^*\Sigma_{g,0} \to \R$ are given by $\{(q_0, 0) \in T^*\Sigma_{g,0} \mid q_0 \in \mathrm{Crit}(f) \}$. The norm in the expression for $\phi_f$ is with respect to any choice of metric.
		\item The skeleton of $(T^*\Sigma_{g,0}, \lambda_f, \phi_f)$ consists of the zero section.
		\item The Liouville vector field $Z_f$ restricted to the zero section coincides with the gradient flow of $f$.
	\end{itemize}

	We model our surface with boundary by $\Sigma_{g,b} = f^{-1}((-\infty, d])$ for some regular value $d$ lying in between the index-$1$ and -$2$ critical values of $f$. Endow the sector $T^*\Sigma_{g,b}$ with the analogous structures $(\lambda_h, Z_h, \phi_h)$ induced by restricion of $(\lambda_f, Z_f, \phi_f)$ to $T^*\Sigma_{g,b} \subset T^*\Sigma_{g,0}$, and consider the domains
	\begin{align*}
		W_{\mathrm{sect}} &:= \phi_h^{-1}((-\infty, d']) \subset T^*\Sigma_{g,b},\\
		W_c &:= \phi_c^{-1}((-\infty, d']), \\
		W_f &:= \phi_f^{-1}((-\infty, d'']), 
	\end{align*}
	where $d''>d'>d$ are regular values of $\phi_f$, still smaller than the index-$2$ critical values. There is a commutative diagram
	\[
		\begin{tikzcd}
		W_c \arrow[rr, "\sim"', "\Phi"] & & W_f\\
		& W_{\mathrm{sect}} \arrow[ul, hook, "\imath_c"] \arrow[ur, hook, "\imath_f"']
	\end{tikzcd}
	\]
	such that $\imath_f^*\lambda_f = \imath_c^* \lambda_c = \lambda_h$, and $\imath_f^*\phi_f = \imath_c^* \phi_c = \phi_h$: the maps $\imath_c$ and $\imath_f$ are inclusion maps, and $\Phi$ is, a priori, just a Weinstein deformation equivalence, which exists because both $W_c$ and $W_f$ are subcritical. By the fact that $\phi_f$ and $\phi_c$ agree on $W_{\mathrm{sect}}$, and because all their critical points lie inside $W_{\mathrm{sect}}$, we can arrange for $\Phi$ to be fixed on $W_{\mathrm{sect}}$, giving commutativity. Then $W_f - \Phi(W_c)$ is a Weinstein cobordism without critical points, which induces a contact isotopy between $\del W_c$ and $\del W_f$.	
\end{proof}

Sectors being understood, our algorithm slightly modifies the one of \cite[Theorem 8.1]{acu_complements}:

\begin{proposition}\label{prop:lifting_with_boundary}
	Let $\Sigma_{g,b}$ denote a surface of genus $g$ with $b$ boundary components, and let $\mathcal{C}$ be a collection of immersed cooriented curves in $\Sigma_{g,b}$. Write $W_{\mathcal{C}}$ for the Weinstein manifold obtained by attaching $2$-handles along the Legendrian lifts of the curves in $\mathcal{C}$ to $\del_{\infty} T^{*}\Sigma_{g,b}$ and convex completing the sectorial boundary. One may obtain a Weinstein handlebody diagram for $W_{\mathcal{C}}$ as follows.

	\begin{enumerate}
		\item Draw $\mathcal{C}$ in a gluing polygon for $\Sigma_{g,b}$, trading coorientation for orientation. Observe that this gluing polygon has a subset of edges which are not identified with any of the other edges.
        
		\item Isotope the curves so that they traverse the polygon in the \emph{clockwise} direction. This may introduce bigons between curves, which can (and should) be resolved whenever the orientations of both strands are opposite.
        
		\item Identify the Legendrian lift of the boundary curves $\del \Sigma_{g,b}$; this amounts to drawing a diagram with the properties of \zcref{lem:boundary_diagram}.
        
		\item Observe that in the gluing polygon after the isotopy, for any curve $\gamma$, one of the following can always be arranged in a neighbourhood of $\gamma$ (c.f. \zcref{fig:local_curve_conf}):
		      \begin{enumerate}[label=(\alph*)]
			      \item $\gamma$ is parallel to a boundary component $\ell$;
			      \item $\gamma$ departs from a boundary curve $\ell$ where $\ell$ passes through a $1$-handle but $\gamma$ does not.
		      \end{enumerate}
		      Denote a boundary component of $\Sigma_{g,b}$ by $\ell_i$, and its lift by $L_i$.
		      \begin{enumerate}[label=(\alph*)]
			      \item Where $\gamma$ is parallel to $\ell_i$, its lift $\Gamma$ is a small positive Reeb pushoff of $L_i$; Diagrammatically, draw a small vertical translate of the handle representing $L_i$.
			      \item Where $\gamma$ does not pass through a $1$-handle, but $\ell_i$ does, the front of $\Gamma$ has a cusp instead of passing through the four-dimensional $1$-handle corresponding to the one $L_i$ passes through. Draw the cusp \emph{above} the top strand of $L_i$.
		      \end{enumerate}
		      For systems of multiple curves, their relative order is preserved (cf. \zcref{fig:local_curve_lifts}).
		\item Delete all $L_i$ from the diagram.
	\end{enumerate}

\end{proposition}

\begin{figure}[htbp]
	\centering
	\def\svgwidth{0.7\textwidth}
	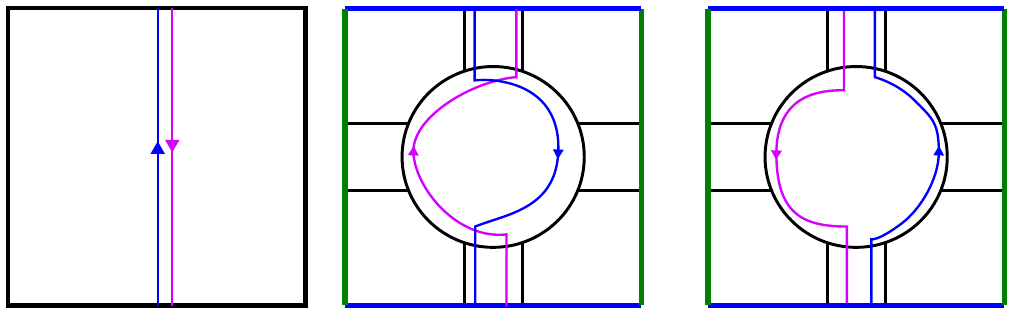
	\caption{Local picture of curve configurations after performing the isotopies in Step 2.}
	\label{fig:local_curve_conf}
\end{figure}

\begin{figure}[htbp]
	\centering
	\def\svgwidth{0.4\textwidth}
	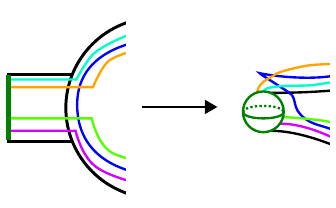
	\caption{The order of curves to be lifted in terms of distance from the boundary is reflected in the order of the fronts of their Legendrian lifts.}
	\label{fig:local_curve_lifts}
\end{figure}

\begin{proof}
    Pick a Morse function $f: \Sigma_{g,0} \to \R$ which has $b$ index-$2$ critical points and observe that $\Sigma_{g,b} = f^{-1}((-\infty, d])$ for some regular value $d$ which lies between the index-$1$ and index-$2$ critical values of $f$. This Morse function induces the gluing polygon used in Step 1, and we use it as before to obtain a morsified Weinstein structure $\phi_f$ on $T^*\Sigma_{g,0}$. By \zcref{lem:ideal_sector_boundary}, it suffices to lift the curves to $\phi_f^{-1}((-\infty, d])$, which \cite[Theorem 8.1]{acu_complements} instructs us how to do.

    More precisely, observe that their algorithm establishes the prescriptions of Steps 1., 2., and 4. above, in the case of a perfect Morse function. One can adapt the procedure to a general Morse function as long as one knows how to draw the attaching links of the Weinstein $2$-handles induced by $\phi_f$; \zcref{lem:boundary_diagram} below identifies these as the Legendrian lifts of the boundary components of $\Sigma_{g,b}$ and provides their fronts. 

	Having performed Steps 1. to 4., one has a Legendrian front of the link consisting of the lifts of $\mathcal{C}$ together with the $b$ attaching links of the index-$2$ critical points of $\phi_f$, that is, one has a diagram representing the Weinstein manifold obtained by attaching handles along $\mathcal{C}$ to $T^*\Sigma_{g,0}$. Denote this Weinstein manifold by $(X, \phi_X)$. It remains to show that we may simply delete the $b$ attaching circles to obtain a handlebody diagram for $W_{\mathcal{C}}$.
	
	By \cite[Lemma 12.20]{ce_stein_weinstein}, we may perform a Weinstein homotopy on $(X, \phi_X)$ varying only $\phi_X$, leaving it invariant near the level set $\phi_X^{-1}(d) = \phi_f^{-1}(d)$, and arbitrarily reordering its critical values. Invariance near the chosen level set in particular implies that the same handlebody diagram describes this Weinstein structure. We choose to reorder them so that the critical values corresponding to the lifts of the attaching circles of $f$ are largest, and continue to denote the resulting function by $\phi_X$. Let $d_0$ denote the smallest of these critical values, and let $\eps>0$ be small enough so that $d' = d_0 - \eps$ is a regular value and such that $[d', d_0)$ contains no critical values.
	
	Consider the Liouville completion of $\phi_X^{-1}((-\infty, d'])$: This is the Weinstein manifold obtained by attaching $2$-handles along lifts in $C$ to $T^*\Sigma_{g,b}$, but not along the lifts of the boundary curves, that is, it is $W_\mathcal{C}$. By the fact that the Weinstein homotopy leaves the ideal boundary invariant and that we truncated the level set below the critical values coming from the attaching circles of $f$, we see that a handlebody diagram for $W_\mathcal{C}$ is obtained by just deleting the corresponding $2$-handles from the diagram, as claimed in Step 5.    
\end{proof}

\begin{remark}
	Observe that choosing not to delete the attaching circles of the Morse function $\phi_f$ results in a handlebody diagram for the Weinstein manifold obtained by attaching handles along the lifts of $\mathcal{C}$ to $T^*\Sigma_{g,0}$ which is equivalent to the one obtained in \cite[Theorem 8.1]{acu_complements}, though perhaps not obviously so. 
\end{remark}

We complete the proof by explaining how to draw Legendrian lifts of the boundary curves. First, we need a preliminary lemma stating that boundary lifts are independent up to Legendrian isotopy of the chosen coorientation.

\begin{lemma}\label{lem:coorientation_independence}
	Let $\Sigma$ be a surface with nonempty boundary and let $\gamma \subset \Sigma$ be a curve parametrizing a boundary component. Denote by $\Lambda_\pm$ the Legendrian lifts of $\gamma$ to $T^\infty\Sigma$ corresponding to either coorientation of $\gamma$.

	Then $\Lambda_+$ is Legendrian isotopic to $\Lambda_-$ in $\del_\infty T^*_c\Sigma$.
\end{lemma}

\begin{proof}
	Without loss of generality, we may assume that $\Sigma$ has a single circular boundary component and consider the product decomposition 
	\[
	\mathcal{N}(\del T^*\Sigma) \cong T^*S^1 \times T^*[0, \eps)
	\]
	induced by taking the cotangent bundle of a collar neighbourhood of $\del \Sigma$. Denoting coordinates by $(\theta, p, r, \xi) \in T^*S^1 \times T^*[0, \eps)$, the standard Liouville form on $\mathcal{N}(\del T^*\Sigma)$ is $\lambda_{\mathrm{can}} = p \dd \theta + \xi \dd r$. Forming the convex completion corresponds, by definition, to gluing the piece $T^*S^1 \times \C_{\mathrm{Re}\leq 0}$, carrying the standard form $\lambda_\C = p\dd \theta + \frac{1}{2}(\xi \dd r - r \dd \xi)$. Observe that the forms agree up to the form $\frac{1}{2}\dd(r\xi)$.
	
	Before we glue, we deform the Liouville structure on $T^*S^1 \times T^*[0, \eps)$ in such a way that we get the Weinstein Morse function
	\[
	\phi: \mathcal{N}(\del T^*\Sigma) \to \R, \qquad (\theta, p, r, \xi) \mapsto \frac{1}{2}(p^2 + \xi^2) -r,
	\]
	which is just the morsification of the standard Morse-Bott Weinstein function, using the Morse function $f = -r$ on $\del\Sigma \times [0, \eps) \subset \Sigma$. Denote this structure by $(\mathcal{N}(\del T^*\Sigma), \lambda_f, Z_f)$. In fact, we have $\lambda_f = \lambda_{\mathrm{can}} + \dd \xi$ according to \cite[Section 7.1]{acu_complements}; observe that 
	\[
	\lambda_{\mathrm{can}} - \lambda_\C =  \dd(\frac{r\xi}{2} + \xi).
	\]
	Pick a smooth function $\chi: \R \to \R$ and a small $\delta \in (0, \eps)$ such that
	\[
	\chi(r) \equiv
		\begin{cases}
			0, & \text{if } r < -\delta,\\
			1, & \text{if } r > \delta,
		\end{cases}
	\]
	and such that $\delta$ is strictly increasing on $[-\delta, \delta]$. Then in a neighbourhood of the sectorial boundary, the convex completion is
	\[
	T^*_c\Sigma \cong T^*S^1 \times (\C_{\mathrm{Re \leq \delta}} \cup T^*[0, \eps)),
	\]
	glued using the obvious identification, with primitive 
	\[
		\lambda := \lambda_\C +  \dd(\chi(r) \cdot \underbrace{(\frac{r\xi}{2} + \xi)}_{=:g} ).
	\]
	For $r < -\delta$, one has $\lambda = \lambda_\C$, and $\lambda = \lambda_f$ where $r > \delta$. To establish that $\phi$ is a valid Weinstein Morse function, we must show that $\phi$ is gradientlike for the resulting Liouville vector field. It is clear that $\phi$ is gradientlike for $Z_f$, which is the Liouville vector field where $r > \delta$. For $r< - \delta$, one has the Liouville vector field $Z_\C = p \del_p + \frac{1}{2}(\xi \del \xi + r \del_r)$, so that
	\[
	\dd \phi(Z_\C) = p^2 + \frac{1}{2}\xi^2 - \frac{r}{2} > 0, \quad \text{since } r < -\delta.
	\]
	The only locus where $\dd\phi(Z_{\C})$ could vanish is at $\xi = p = r = 0$. It remains to check on the overlap. There, the Liouville vector field is
	\[
	Z = Z_{\C} + X_{\chi g},
	\]
	where $X_{\chi g}$ denotes the hamiltonian vector field. One explicitly computes it to be
	\[
		X_{\chi g} = -\chi(1+\frac{r}{2})\del_r + \xi(\frac{r}{2}\dot{\chi} + \dot{\chi} + \frac{1}{2}\chi)\del_\xi,
	\]
	so that

	\[
		\dd \phi(X_{\chi_g}) = \chi(1+\frac{r}{2}) + \xi^2(\dot{\chi}(1+\frac{r}{2}) + \frac{1}{2}\chi).
	\]
	For $r \in [-\delta, \delta]$ and $\delta$ small enough, the above is positive (since both $\chi$ and $\dot{\chi}$ are positive in this region by construction). It follows that
	\[
	\dd \phi(Z) = \dd \phi(Z_{\C}) + \dd \phi(X_{\chi g}) > 0
	\]
	also for $r \in [-\delta, \delta]$, and thus $\phi$ is gradientlike for $Z$.

	We are now in the position to determine the Legendrian lifts of a boundary curve to the ideal boundary $\del_\infty T^*_c \Sigma \cong \phi^{-1}(c)$, which may be parametrized as $\gamma(t) = (t, 0) \in S^1 \times [0, \eps) \subset \Sigma$. We have $\dot{\gamma} = \del_\theta,$ so by definition, the biconormal lift is
	\begin{align*}	
		\Lambda_\pm  &= \{(\theta, p, r, \xi) \in \phi^{-1}(c) \mid (\theta, r) \in \gamma, \, (p \dd \theta + \xi \dd r)(\del_\theta) = 0 \}\\
						&= \{(\theta, 0, 0, \pm \sqrt{2c} ) \in T^*S^1 \times T^*[0, \eps)\}.
	\end{align*}
	Define the family
	\[
		\Lambda_s(t) := (t, 0, c(s^2-1), s \sqrt{2c}), \qquad s \in [-1, 1].
	\]
	One checks that $\Lambda_s(t) \in \phi^{-1}(c)$ for all $s$ and $t$, and that $\Lambda_1 = \Lambda_+$ and $\Lambda_{-1} = \Lambda_-$. Observe that $\Lambda_s$ acquires a negative $r$-coordinate for $s \in (-1,1)$, so the isotopy necessarily enters the region added by forming the convex completion.

	It remains to check that each $\Lambda_s$ defines a Legendrian submanifold. Observe that again, $\dot{\Lambda}_s = \del_\theta$ for each $s$. The contact form on $\phi^{-1}(c)$ is the restriction of $\lambda$, and one checks that $\lambda(\del_\theta) = p$; thus evidently, each $\Lambda_s$ is a Legendrian submanifold since the $p$-coordinate vanishes. 
\end{proof}

We finally explain how to obtain the lifts of the boundary curves.

\begin{lemma}\label{lem:boundary_diagram}
	Let $\Sigma_{g,b}$ denote the surface of genus $g$ with $b$ boundary components. Write $\mathcal{B}$ for the collection of the $b$ boundary curves of $\Sigma_{g,b}$, endowed with \textbf{any} coorientation, and let $\Lambda$ be the Legendrian lift of $\mathcal{B}$ to $\del_\infty T^*_c \Sigma_{g,b}$.

	One can obtain a front projection of $\Lambda$ in Gompf normal form as follows:
	\begin{enumerate}
		\item Draw $\mathrm{rk}(H_1(\Sigma_{g,b})) = 2g+b-1$ pairs of $1$-handle attaching spheres.
		\item Draw a Legendrian $tb$ $-1$ unknot representing the boundary of the zero-handle of $\Sigma_{g,b}$, and mark on it the relative positions of the attaching spheres of the (two-dimensional) $1$-handles of $\Sigma_{g,b}$ (see \zcref{fig:cotangent_torus} and \zcref{fig:cotangent_pants}).
		\item Perform Legendrian isotopies on the unknot to isotope the (zero-dimensional) attaching spheres drawn in the last step to cusps in the centres of the attaching regions of the corresponding (three-dimensional) $1$-handles. Take care that, for each pair of balls describing the attaching region of a $1$-handle, the pair of cusps isotoped to their centres is a pair of up- and downwards cusps; otherwise, the resulting figure is not in Gompf normal form (compare \zcref{fig:non_gompf_cusp}).
		\item Exchange the cusps for a pair of strands passing over the $1$-handles.
	\end{enumerate}
\end{lemma}

\begin{figure}[htbp]
	\centering
	\def\svgwidth{\textwidth}
	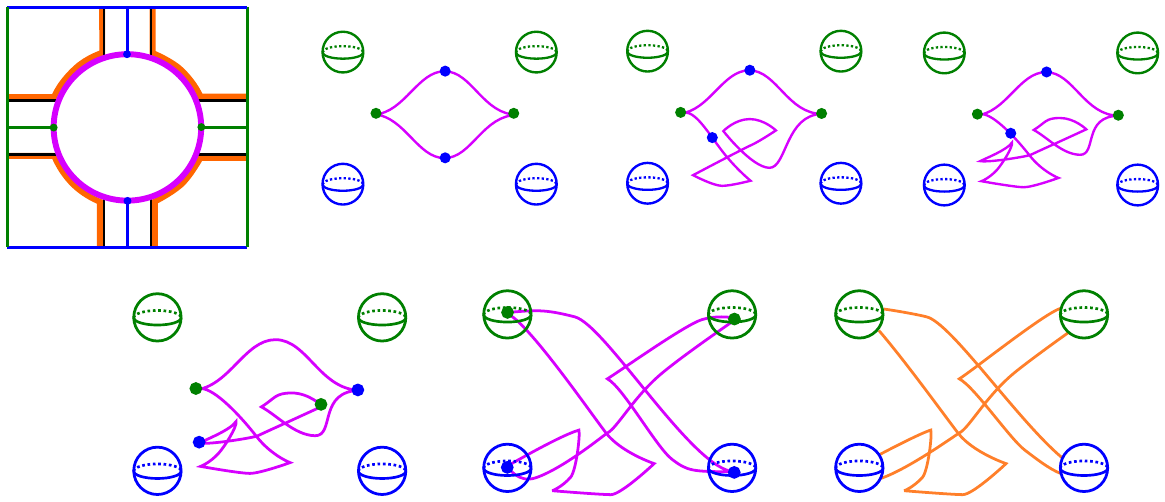
	
	\caption{Applying \zcref{lem:boundary_diagram} to recover the Gompf handlebody diagram for $T^*T^2$.}
	\label{fig:cotangent_torus}
\end{figure}

\begin{figure}[htbp]
	\centering
	\def\svgwidth{\textwidth}
	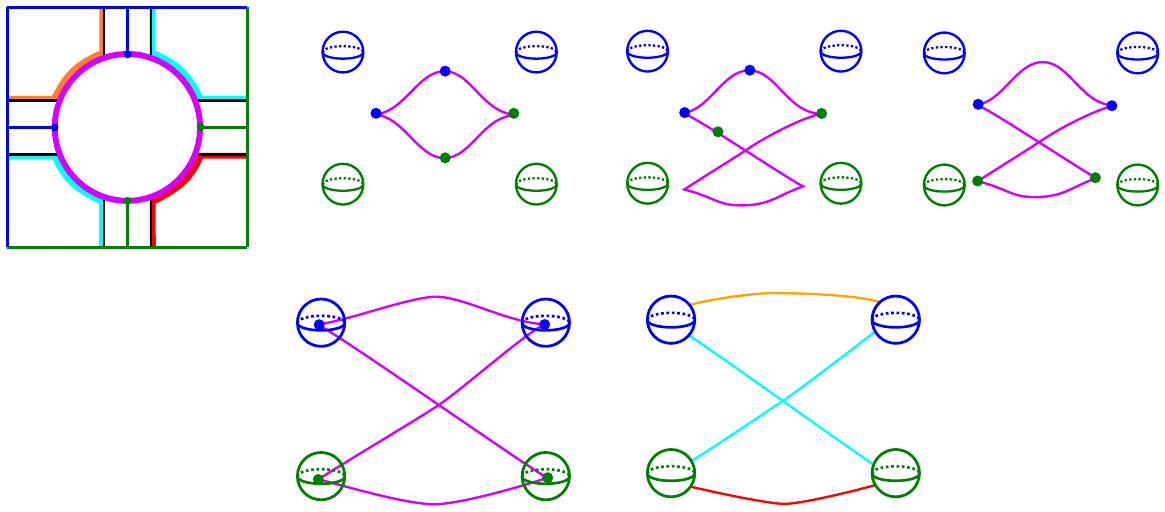
	\caption{Applying \zcref{lem:boundary_diagram} with a Morse function on $S^2$ with three index-$2$ critical points to obtain the Legendrian boundary lifts of the pair of pants $P$ to $T^\infty P$.}
	\label{fig:cotangent_pants}
\end{figure}

\begin{figure}[htbp]
	\centering
	\def\svgwidth{\textwidth}
	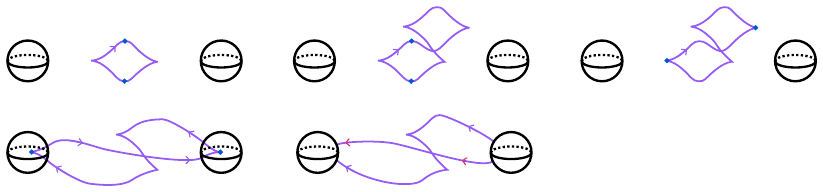
	\caption{When a pair of upwards cusps (with respect to any orientation) is matched in Step 3, the Gompf normal form convention of identifying strands entering the attaching balls at identical heights is inconsistent with the chosen orientation of the unknot.}
	\label{fig:non_gompf_cusp}
\end{figure}

\begin{proof}
	We consider again the morsified Weinstein structure  $(T^*\Sigma_{g,0}, Z_f, \phi_f)$  coming from a self-indexing Morse function $f: \Sigma_{g,0} \to \R$ with a unique index-$0$ critical point, $2g+b-1$ index-$1$ critical points, and $b$ index-$2$ critical points. 

	Choose regular values $c_0$ and $c_1$ lying inbetween the critical values of $\phi_f$ of index $0$ and $1$, and $1$ and $2$, respectively. We model $\Sigma_{g,b}$ as $f^{-1}((-\infty, c_1])$. Set $W_i = \phi_f^{-1}((-\infty, c_i])$, and $Y_i = \del W_i = \phi_f^{-1}(c_i)$, and endow the $Y_i$ with the contact structure induced by the Liouville form. Our first claim is that $\Lambda$ coincides with 
	\[
	L_1 := \mathrm{Skel}(W,Z_f, \phi_f) \cap Y_1,
	\]
	the attaching link of the $2$-cells of $\phi_f$, up to Legendrian isotopy. To see this, first observe that by \zcref{lem:ideal_sector_boundary}, $L_1$ lives in the same space as $\Lambda$. We next prove that the projection of $L_1$ to $\Sigma_{g,b}$ is $\del \Sigma_{g,b}$ (strictly speaking, $L_1$ is already a subset of the zero section since it is a subset of the skeleton).
	
	Recall that the skeleton is the union of the stable manifolds $W^s_{Z_f}(x_0)$ for $x_0 \in \mathrm{Crit}(\phi)$. Since $Z_f\vert_{\Sigma_{g,0}} = \nabla f$, it follows that the projection of $L_1$ to $\Sigma_{g,0}$ consists of the set 
	\[
		\bigcup_{x_0 \in \mathrm{Crit}_2(f)}\{ x \in \Sigma_{g,0} \mid \varphi_{\nabla f}^t(x) \xrightarrow{t \to \infty} x_0, \, f(x) = c_1  \}.
	\]

	This union is over index-$2$ critical points, and therefore, this is a union of circles coinciding with the set $f^{-1}(c_1) = \del \Sigma_{g,b}$. By \zcref{lem:coorientation_independence}, the lift $\Lambda$ does not depend on the coorientation of the base curve up to Legendrian isotopy, establishing the first claim.
	
	It thus suffices to describe $L_1$. Following the proof of \cite[Theorem 7.3]{acu_complements}, we describe $L_1$ in terms of 
	\[
		L_0 := \mathrm{Skel}(T^*\Sigma_{g,0}, Z_f, \phi_f) \cap Y_0.
	\]

	The restriction of $\phi_f$ to $\mathrm{Skel}(T^*\Sigma_{g,0}, Z_f, \phi_f) \cap W_0 =: D_0$ is just $f$, which has a unique index-$0$ critical point. It follows that $D_0$ is an unknotted Lagrangian disk, and its boundary $L_0$ is thus an unknot of maximal Thurston-Bennequin number. $D_0$ inherits a stratification by the stable manifolds of $\phi_f$ (or $f = \phi_f\vert_{\mathrm{Skel}(T^*\Sigma_{g,b}, Z_f, \phi_f)})$, which consists of two $1$-dimensional strata limiting to the index-$0$ critical point for each index-$1$ critical point, and their complement forms the $2$-dimensional stratum given by the stable manifold of the index-$2$ critical point. These strata intersect $L_0$ in two points for each $1$-critical point, and their complement is the intersection with the stable manifolds of the index-$2$ critical points.

	It follows that the unknot $L_0$ is such that each intersection point with an index-$1$ stable manifold lies in the center of a $3$-ball representing part of the attaching region of the corresponding $1$-handle. Inside this attaching region, $L_0$ forms an unknotted Legendrian arc, which can be drawn in the front projection as a standard cusp inside the attaching regions.
	
	Flowing $L_0$ along $Z_f$ over an index-$1$ critical point $x_0$ has the effect of performing a contact $0$-surgery along the $0$-sphere $L_0 \cap W^s_{Z_f}(x_0)$, which results in the pair of Legendrian arcs of $L_0$ in the attaching region being replaced by the boundary of the core of the resulting $1$-handle. Observe that this may increase the number of connected components of the Legendrian.

	Flowing over all index-$1$ critical points to the level set $Y_1$ yields $L_1$; we infer that $L_0$ differs from $L_1$ precisely by this set of $0$-surgeries. See \zcref{fig:cotangent_torus} and \zcref{fig:cotangent_pants} for the corresponding effect in the front.

	Therefore, it suffices to understand $L_0$; however, we have already established that $L_0$ is a max-$tb$ unknot such that the centers for the $0$-surgeries on the ambient $Y_0$ match up with the intersection points of $L_0$ with their corresponding stable manifolds. 

	It thus suffices to perform an isotopy of the standard max-$tb$ unknot realizing this matching to give a front projection of $L_0$, and performing the $0$-surgery to obtain a front of $L_1$ has the effect stated in Step 4. 
\end{proof}

\begin{remark}
	There are many ways to isotope a Legendrian unknot in accordance with the conditions from Steps 1-3, giving rise to ostensibly different presentations of the boundary lifts. In this way, \zcref{lem:boundary_diagram} can be used to construct equivalent (but non-trivially so) handlebody diagrams for cotangent bundles of closed surfaces. We thank Emmy Murphy for helpful correspondence on this point.
\end{remark}

\subsection{Attaching along the Stokes Legendrian}

In this subsection, we prove \zcref{thm:stokes_is_generic_monodromy} by applying \zcref{prop:lifting_with_boundary} to all the diagrams in \zcref{fig:stokes_legendrians}, and performing some handlebody calculus moves to match the pictures with those from \zcref{fig:main_results}. Observe that, for PI and PII, the once-punctured sphere is just $\R^2$, and thus the fronts of their Stokes Legendrians can be drawn as a front in standard $\R^3$ via the contactomorphism $J^1S^1 \cong T^\infty \R^2$, which amounts to un-folding and closing the braid (see \zcref{fig:p1_p2_matching}):

\begin{figure}[htbp]
	\centering
	\def\svgwidth{\textwidth}
\begingroup%
  \makeatletter%
  \providecommand\color[2][]{%
    \errmessage{(Inkscape) Color is used for the text in Inkscape, but the package 'color.sty' is not loaded}%
    \renewcommand\color[2][]{}%
  }%
  \providecommand\transparent[1]{%
    \errmessage{(Inkscape) Transparency is used (non-zero) for the text in Inkscape, but the package 'transparent.sty' is not loaded}%
    \renewcommand\transparent[1]{}%
  }%
  \providecommand\rotatebox[2]{#2}%
  \newcommand*\fsize{\dimexpr\f@size pt\relax}%
  \newcommand*\lineheight[1]{\fontsize{\fsize}{#1\fsize}\selectfont}%
  \ifx\svgwidth\undefined%
    \setlength{\unitlength}{535.93795656bp}%
    \ifx\svgscale\undefined%
      \relax%
    \else%
      \setlength{\unitlength}{\unitlength * \real{\svgscale}}%
    \fi%
  \else%
    \setlength{\unitlength}{\svgwidth}%
  \fi%
  \global\let\svgwidth\undefined%
  \global\let\svgscale\undefined%
  \makeatother%
  \begin{picture}(1,0.30688591)%
    \lineheight{1}%
    \setlength\tabcolsep{0pt}%
    \put(0,0){\includegraphics[width=\unitlength,page=1]{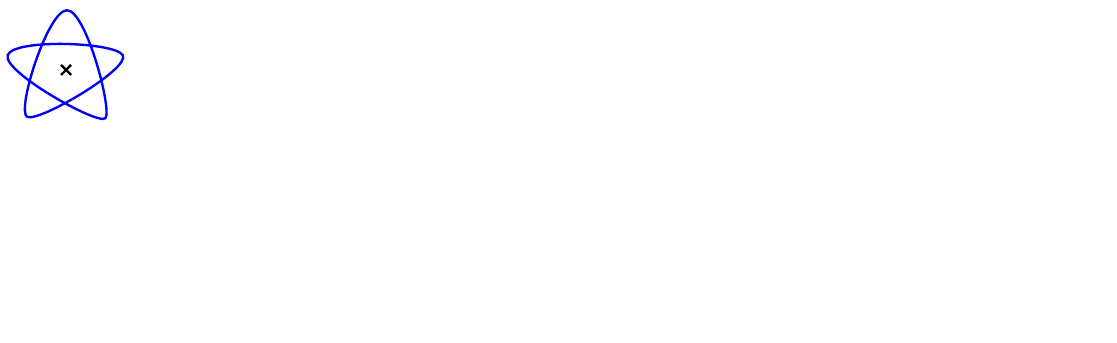}}%
    \put(0.05018424,0.17846416){\color[rgb]{0,0,0}\makebox(0,0)[lt]{\lineheight{1.25}\smash{\begin{tabular}[t]{l}PI\end{tabular}}}}%
    \put(0.03561271,0.0160008){\color[rgb]{0,0,0}\makebox(0,0)[lt]{\lineheight{1.25}\smash{\begin{tabular}[t]{l}$T^\infty\mathbb{R}^2$\end{tabular}}}}%
    \put(0,0){\includegraphics[width=\unitlength,page=2]{stokes_easy.pdf}}%
    \put(0.04654437,0.05680719){\color[rgb]{0,0,0}\makebox(0,0)[lt]{\lineheight{1.25}\smash{\begin{tabular}[t]{l}PII\\\end{tabular}}}}%
    \put(0,0){\includegraphics[width=\unitlength,page=3]{stokes_easy.pdf}}%
    \put(0.29999375,0.0160008){\color[rgb]{0,0,0}\makebox(0,0)[lt]{\lineheight{1.25}\smash{\begin{tabular}[t]{l}$J^1S^1$\end{tabular}}}}%
    \put(0,0){\includegraphics[width=\unitlength,page=4]{stokes_easy.pdf}}%
    \put(0.61575627,0.0160008){\color[rgb]{0,0,0}\makebox(0,0)[lt]{\lineheight{1.25}\smash{\begin{tabular}[t]{l}$\mathbb{R}^3$\end{tabular}}}}%
    \put(0,0){\includegraphics[width=\unitlength,page=5]{stokes_easy.pdf}}%
  \end{picture}%
\endgroup%

	\caption{Representing the Stokes Legendrians of PI and PII as fronts in contact $\R^3$. Observe that the expression for $PII$ does not match the one given in \zcref{fig:main_results}, but both are evidently the $(2,4)$-torus link.}
	\label{fig:p1_p2_matching}
\end{figure}

The case of PVI can also be treated by a simplified argument. Observe that the Stokes Legendrian of PVI consists of a doubled copy of the boundary curves. Attaching along the Legendrian link formed by the components projecting to the \emph{inner} circle in \zcref{fig:stokes_legendrians} corresponds to re-attaching the $2$-handles below the critical values of which we cut off the Morse function in the proof of \zcref{prop:lifting_with_boundary}, and thus the diagram is  equivalent to attaching along the Legendrian lifts of four disjoint circles on $S^2$, coriented outwards. Moreover, since these circles are all contractible and do not intersect, they are seen to be cotangent fibres perturbed under the Reeb flow; flowing in the opposite direction contracts each circle to a point and re-expands them with the opposite coorientation. This is cosmetic, but makes the resulting orientations fit squarely into our conventions. \zcref{fig:stokes_p6} carries out the simplifications.

\begin{figure}[h]
	\centering
	\def\svgwidth{\textwidth}
	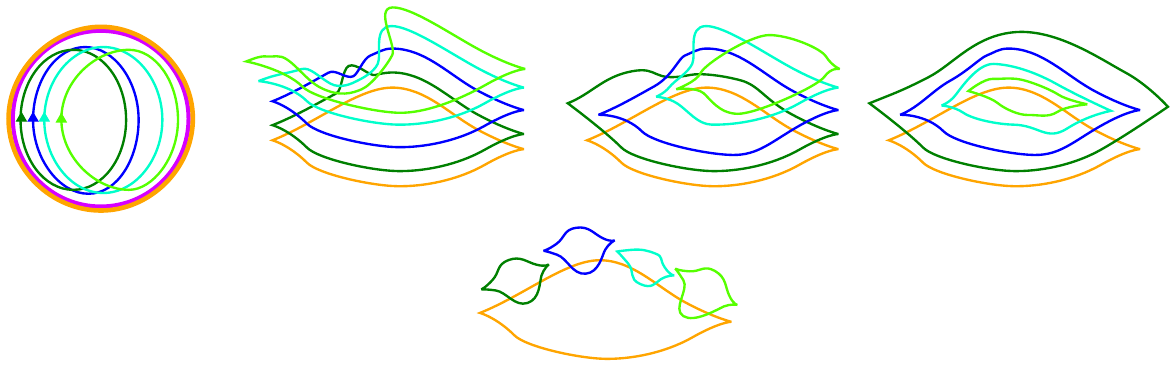
	
	\caption{Applying \cite[Theorem 8.1]{acu_complements} to four unlinked circles with inward coorientation in $S^2$. Observe that the bigons in the first diagram do not result in a linking of the lift since the coorientations are opposite. The orange and purple circles denote $L_1$ and $L_0$ in the proof of \zcref{lem:boundary_diagram}.}
	\label{fig:stokes_p6}
\end{figure}

\begin{figure}[htbp]
	\centering
	\def\svgwidth{\textwidth}
    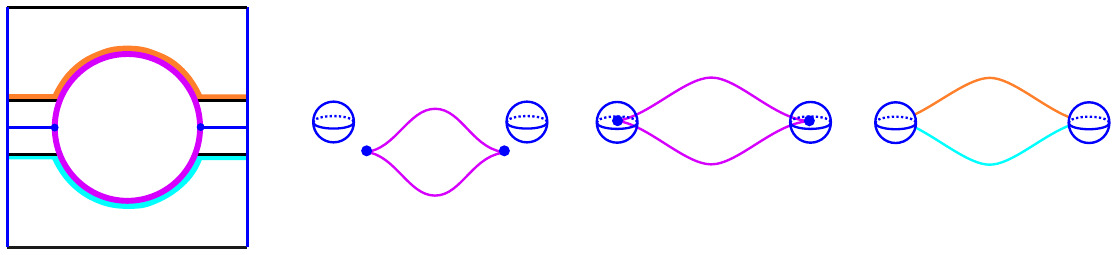	
	\caption{Applying \zcref{lem:boundary_diagram} to obtain fronts of the lifts of the boundary circles of the cylinder.}
	\label{fig:cyl_boundary}
\end{figure}

We proceed to go through all the Stokes diagrams that have two punctures.

\begin{minipage}{\textwidth}
\addtocontents{toc}{\SkipTocEntry}
\subsection*{PII(FN)}
	\centering
	\def\svgwidth{\textwidth}
\begingroup%
  \makeatletter%
  \providecommand\color[2][]{%
    \errmessage{(Inkscape) Color is used for the text in Inkscape, but the package 'color.sty' is not loaded}%
    \renewcommand\color[2][]{}%
  }%
  \providecommand\transparent[1]{%
    \errmessage{(Inkscape) Transparency is used (non-zero) for the text in Inkscape, but the package 'transparent.sty' is not loaded}%
    \renewcommand\transparent[1]{}%
  }%
  \providecommand\rotatebox[2]{#2}%
  \newcommand*\fsize{\dimexpr\f@size pt\relax}%
  \newcommand*\lineheight[1]{\fontsize{\fsize}{#1\fsize}\selectfont}%
  \ifx\svgwidth\undefined%
    \setlength{\unitlength}{501.89954382bp}%
    \ifx\svgscale\undefined%
      \relax%
    \else%
      \setlength{\unitlength}{\unitlength * \real{\svgscale}}%
    \fi%
  \else%
    \setlength{\unitlength}{\svgwidth}%
  \fi%
  \global\let\svgwidth\undefined%
  \global\let\svgscale\undefined%
  \makeatother%
  \begin{picture}(1,0.59888511)%
    \lineheight{1}%
    \setlength\tabcolsep{0pt}%
    \put(0,0){\includegraphics[width=\unitlength,page=1]{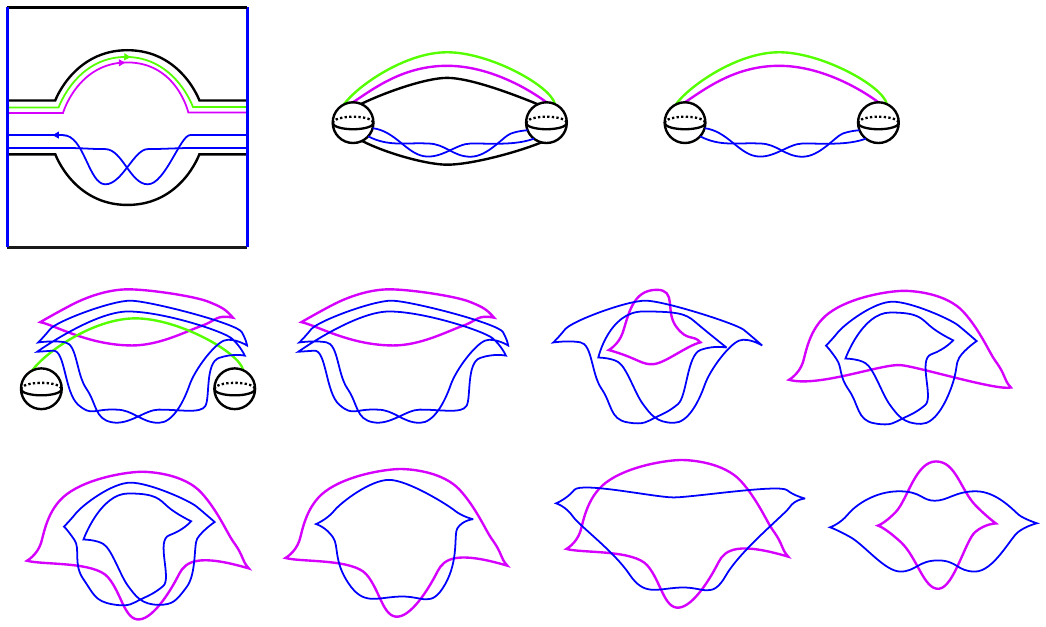}}%
  \end{picture}%
\endgroup%

	\captionof{figure}{Application of \zcref{prop:lifting_with_boundary} and subsequent simplifications of PII(FN).}
	\label{fig:stokes_p2fn}
\end{minipage}

\begin{minipage}{\textwidth}
\addtocontents{toc}{\SkipTocEntry}
\subsection*{PIII(D6)}
	\centering
	\def\svgwidth{\textwidth}
\begingroup%
  \makeatletter%
  \providecommand\color[2][]{%
    \errmessage{(Inkscape) Color is used for the text in Inkscape, but the package 'color.sty' is not loaded}%
    \renewcommand\color[2][]{}%
  }%
  \providecommand\transparent[1]{%
    \errmessage{(Inkscape) Transparency is used (non-zero) for the text in Inkscape, but the package 'transparent.sty' is not loaded}%
    \renewcommand\transparent[1]{}%
  }%
  \providecommand\rotatebox[2]{#2}%
  \newcommand*\fsize{\dimexpr\f@size pt\relax}%
  \newcommand*\lineheight[1]{\fontsize{\fsize}{#1\fsize}\selectfont}%
  \ifx\svgwidth\undefined%
    \setlength{\unitlength}{497.04946923bp}%
    \ifx\svgscale\undefined%
      \relax%
    \else%
      \setlength{\unitlength}{\unitlength * \real{\svgscale}}%
    \fi%
  \else%
    \setlength{\unitlength}{\svgwidth}%
  \fi%
  \global\let\svgwidth\undefined%
  \global\let\svgscale\undefined%
  \makeatother%
  \begin{picture}(1,0.88240522)%
    \lineheight{1}%
    \setlength\tabcolsep{0pt}%
    \put(0,0){\includegraphics[width=\unitlength,page=1]{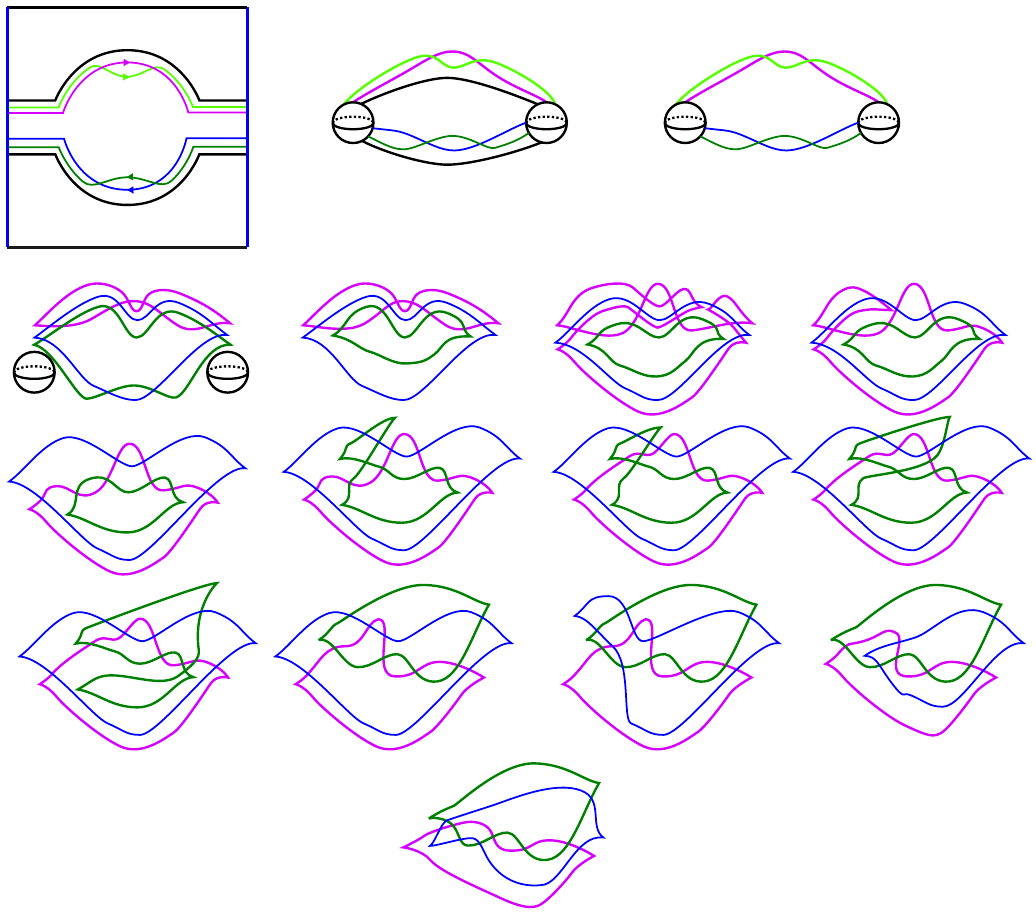}}%
  \end{picture}%
\endgroup%

	\captionof{figure}{Application of \zcref{prop:lifting_with_boundary} and subsequent simplifications of PIII(D6). The last diagram is observed to be a $180$-degree rotation of the central link on the fourth line of \zcref{fig:p3d6}, thereby proving Weinstein deformation equivalence for PIII(D6).}
	\label{fig:stokes_p3d6}
\end{minipage}

\begin{minipage}{\textwidth}
\addtocontents{toc}{\SkipTocEntry}
\subsection*{PIII(D7)}
	\centering
	\def\svgwidth{\textwidth}
\begingroup%
  \makeatletter%
  \providecommand\color[2][]{%
    \errmessage{(Inkscape) Color is used for the text in Inkscape, but the package 'color.sty' is not loaded}%
    \renewcommand\color[2][]{}%
  }%
  \providecommand\transparent[1]{%
    \errmessage{(Inkscape) Transparency is used (non-zero) for the text in Inkscape, but the package 'transparent.sty' is not loaded}%
    \renewcommand\transparent[1]{}%
  }%
  \providecommand\rotatebox[2]{#2}%
  \newcommand*\fsize{\dimexpr\f@size pt\relax}%
  \newcommand*\lineheight[1]{\fontsize{\fsize}{#1\fsize}\selectfont}%
  \ifx\svgwidth\undefined%
    \setlength{\unitlength}{457.82079501bp}%
    \ifx\svgscale\undefined%
      \relax%
    \else%
      \setlength{\unitlength}{\unitlength * \real{\svgscale}}%
    \fi%
  \else%
    \setlength{\unitlength}{\svgwidth}%
  \fi%
  \global\let\svgwidth\undefined%
  \global\let\svgscale\undefined%
  \makeatother%
  \begin{picture}(1,0.78733074)%
    \lineheight{1}%
    \setlength\tabcolsep{0pt}%
    \put(0,0){\includegraphics[width=\unitlength,page=1]{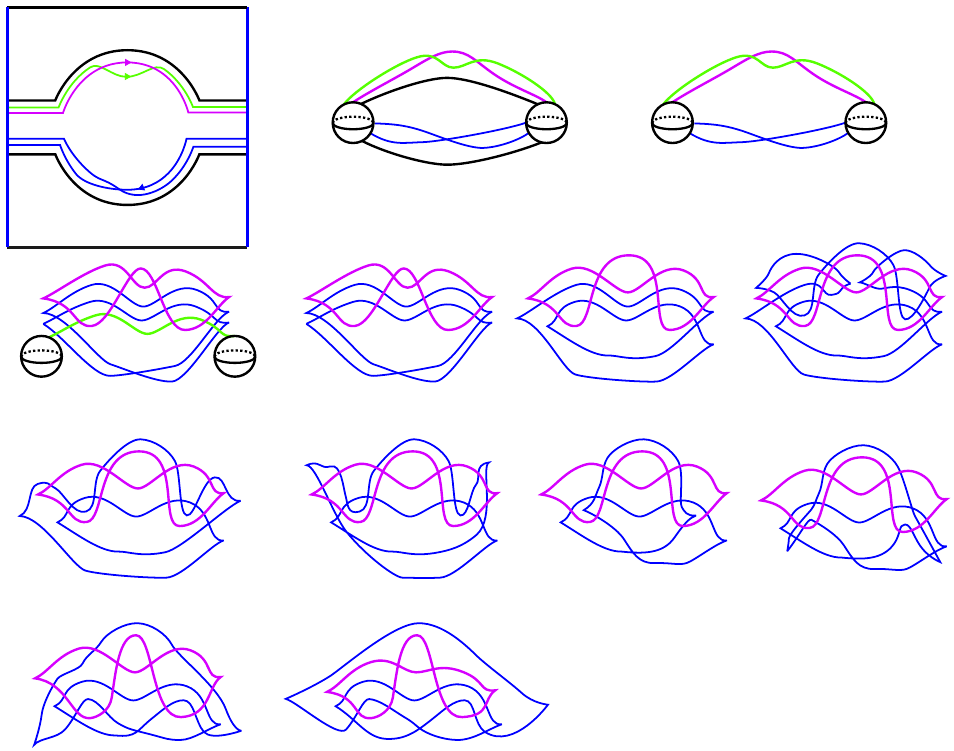}}%
  \end{picture}%
\endgroup%

	\captionof{figure}{Application of \zcref{prop:lifting_with_boundary} and subsequent simplifications of PIII(D7).}
	\label{fig:stokes_p3d7}
\end{minipage}

\begin{minipage}{\textwidth}
	\addtocontents{toc}{\SkipTocEntry}
		\subsection*{PIII(D8)}
	\centering
	\def\svgwidth{\textwidth}
\begingroup%
  \makeatletter%
  \providecommand\color[2][]{%
    \errmessage{(Inkscape) Color is used for the text in Inkscape, but the package 'color.sty' is not loaded}%
    \renewcommand\color[2][]{}%
  }%
  \providecommand\transparent[1]{%
    \errmessage{(Inkscape) Transparency is used (non-zero) for the text in Inkscape, but the package 'transparent.sty' is not loaded}%
    \renewcommand\transparent[1]{}%
  }%
  \providecommand\rotatebox[2]{#2}%
  \newcommand*\fsize{\dimexpr\f@size pt\relax}%
  \newcommand*\lineheight[1]{\fontsize{\fsize}{#1\fsize}\selectfont}%
  \ifx\svgwidth\undefined%
    \setlength{\unitlength}{417.08124116bp}%
    \ifx\svgscale\undefined%
      \relax%
    \else%
      \setlength{\unitlength}{\unitlength * \real{\svgscale}}%
    \fi%
  \else%
    \setlength{\unitlength}{\svgwidth}%
  \fi%
  \global\let\svgwidth\undefined%
  \global\let\svgscale\undefined%
  \makeatother%
  \begin{picture}(1,0.29247574)%
    \lineheight{1}%
    \setlength\tabcolsep{0pt}%
    \put(0,0){\includegraphics[width=\unitlength,page=1]{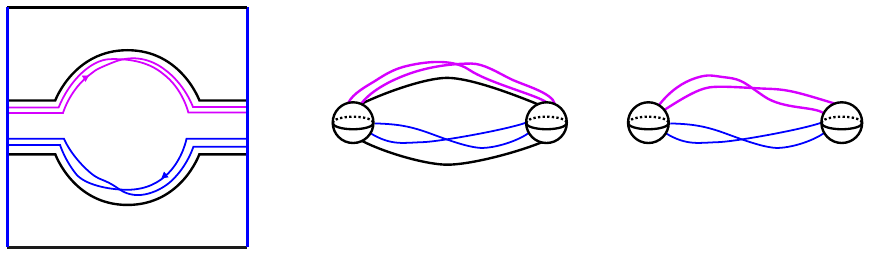}}%
  \end{picture}%
\endgroup%

	\captionof{figure}{Application of \zcref{prop:lifting_with_boundary} and subsequent simplifications of PIII(D8).}
	\label{fig:stokes_p3d8}
\end{minipage}

\begin{minipage}{\textwidth}
	\addtocontents{toc}{\SkipTocEntry}
	\subsection*{PIV}
	\centering
	\def\svgwidth{\textwidth}
\begingroup%
  \makeatletter%
  \providecommand\color[2][]{%
    \errmessage{(Inkscape) Color is used for the text in Inkscape, but the package 'color.sty' is not loaded}%
    \renewcommand\color[2][]{}%
  }%
  \providecommand\transparent[1]{%
    \errmessage{(Inkscape) Transparency is used (non-zero) for the text in Inkscape, but the package 'transparent.sty' is not loaded}%
    \renewcommand\transparent[1]{}%
  }%
  \providecommand\rotatebox[2]{#2}%
  \newcommand*\fsize{\dimexpr\f@size pt\relax}%
  \newcommand*\lineheight[1]{\fontsize{\fsize}{#1\fsize}\selectfont}%
  \ifx\svgwidth\undefined%
    \setlength{\unitlength}{457.30487012bp}%
    \ifx\svgscale\undefined%
      \relax%
    \else%
      \setlength{\unitlength}{\unitlength * \real{\svgscale}}%
    \fi%
  \else%
    \setlength{\unitlength}{\svgwidth}%
  \fi%
  \global\let\svgwidth\undefined%
  \global\let\svgscale\undefined%
  \makeatother%
  \begin{picture}(1,0.76096378)%
    \lineheight{1}%
    \setlength\tabcolsep{0pt}%
    \put(0,0){\includegraphics[width=\unitlength,page=1]{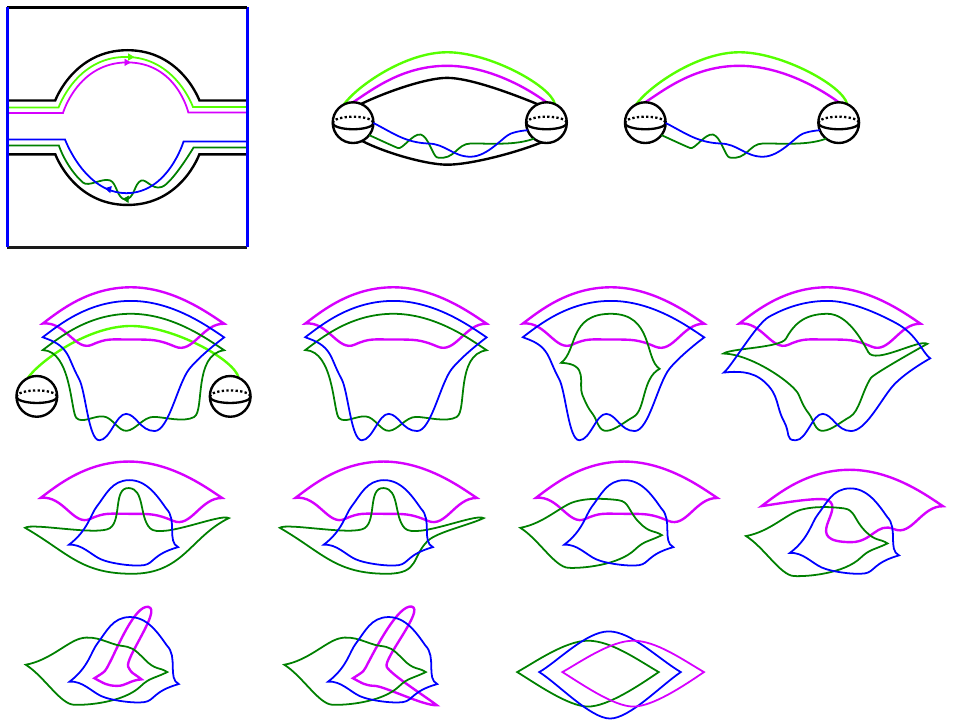}}%
  \end{picture}%
\endgroup%

	\captionof{figure}{Application of \zcref{prop:lifting_with_boundary} and subsequent simplifications of PIV.}
	\label{fig:stokes_p4}
\end{minipage}

We now proceed to the cases with three punctures. 

\begin{minipage}{\textwidth}
	\addtocontents{toc}{\SkipTocEntry}
	\subsection*{PV}
	\centering
	\def\svgwidth{\textwidth}
\begingroup%
  \makeatletter%
  \providecommand\color[2][]{%
    \errmessage{(Inkscape) Color is used for the text in Inkscape, but the package 'color.sty' is not loaded}%
    \renewcommand\color[2][]{}%
  }%
  \providecommand\transparent[1]{%
    \errmessage{(Inkscape) Transparency is used (non-zero) for the text in Inkscape, but the package 'transparent.sty' is not loaded}%
    \renewcommand\transparent[1]{}%
  }%
  \providecommand\rotatebox[2]{#2}%
  \newcommand*\fsize{\dimexpr\f@size pt\relax}%
  \newcommand*\lineheight[1]{\fontsize{\fsize}{#1\fsize}\selectfont}%
  \ifx\svgwidth\undefined%
    \setlength{\unitlength}{459.11623016bp}%
    \ifx\svgscale\undefined%
      \relax%
    \else%
      \setlength{\unitlength}{\unitlength * \real{\svgscale}}%
    \fi%
  \else%
    \setlength{\unitlength}{\svgwidth}%
  \fi%
  \global\let\svgwidth\undefined%
  \global\let\svgscale\undefined%
  \makeatother%
  \begin{picture}(1,1.00280302)%
    \lineheight{1}%
    \setlength\tabcolsep{0pt}%
    \put(0,0){\includegraphics[width=\unitlength,page=1]{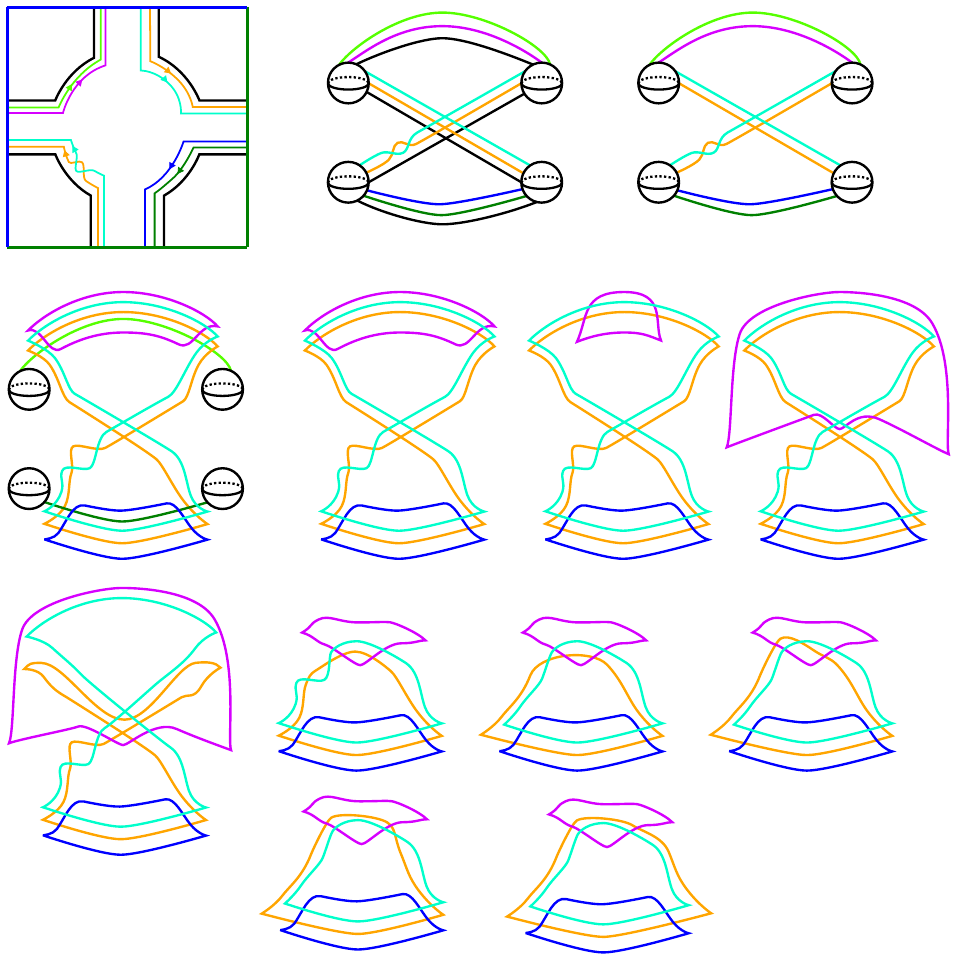}}%
  \end{picture}%
\endgroup%

	\captionof{figure}{Application of \zcref{prop:lifting_with_boundary} and subsequent simplifications of PV.}
	\label{fig:stokes_p5}
\end{minipage}

\begin{minipage}{\textwidth}
	\subsection*{PV(deg)}
	\addtocontents{toc}{\SkipTocEntry}
	\centering
	\def\svgwidth{\textwidth}
\begingroup%
  \makeatletter%
  \providecommand\color[2][]{%
    \errmessage{(Inkscape) Color is used for the text in Inkscape, but the package 'color.sty' is not loaded}%
    \renewcommand\color[2][]{}%
  }%
  \providecommand\transparent[1]{%
    \errmessage{(Inkscape) Transparency is used (non-zero) for the text in Inkscape, but the package 'transparent.sty' is not loaded}%
    \renewcommand\transparent[1]{}%
  }%
  \providecommand\rotatebox[2]{#2}%
  \newcommand*\fsize{\dimexpr\f@size pt\relax}%
  \newcommand*\lineheight[1]{\fontsize{\fsize}{#1\fsize}\selectfont}%
  \ifx\svgwidth\undefined%
    \setlength{\unitlength}{453.01604984bp}%
    \ifx\svgscale\undefined%
      \relax%
    \else%
      \setlength{\unitlength}{\unitlength * \real{\svgscale}}%
    \fi%
  \else%
    \setlength{\unitlength}{\svgwidth}%
  \fi%
  \global\let\svgwidth\undefined%
  \global\let\svgscale\undefined%
  \makeatother%
  \begin{picture}(1,1.10619846)%
    \lineheight{1}%
    \setlength\tabcolsep{0pt}%
    \put(0,0){\includegraphics[width=\unitlength,page=1]{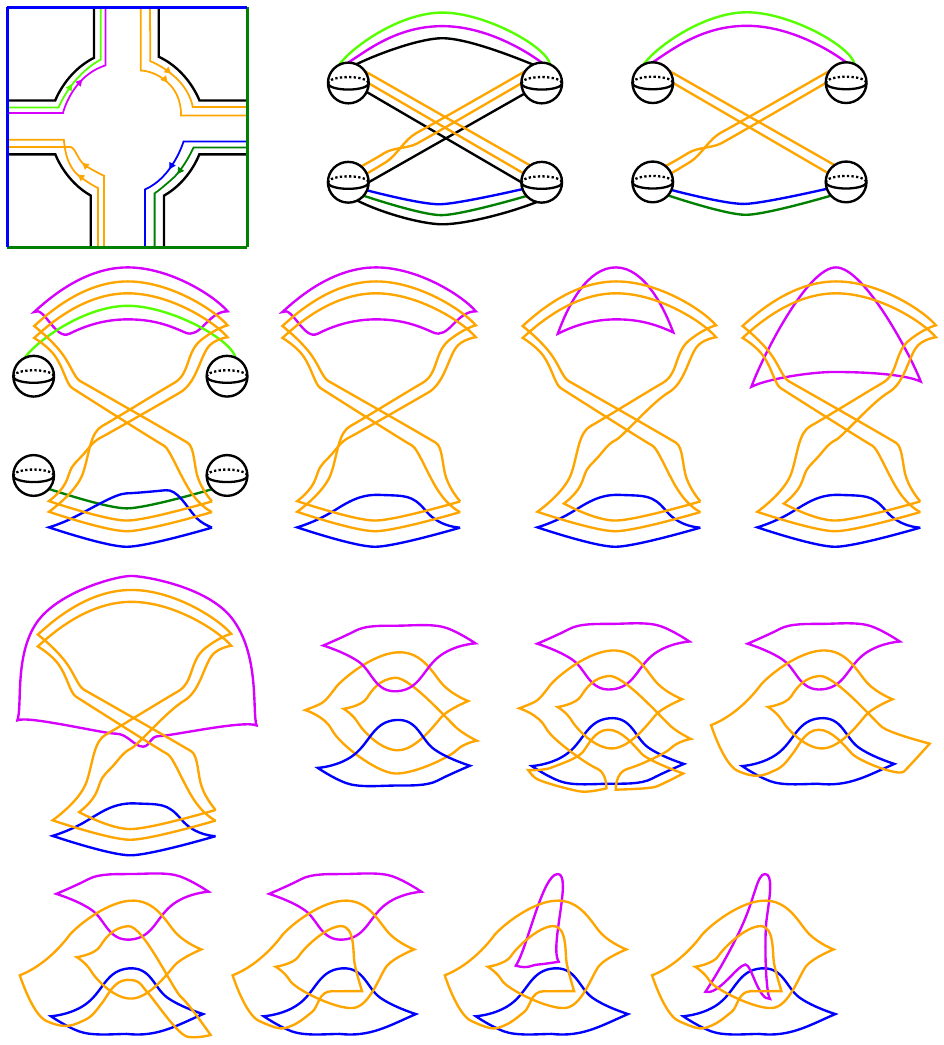}}%
  \end{picture}%
\endgroup%

	\captionof{figure}{Application of \zcref{prop:lifting_with_boundary} and subsequent simplifications of PV(deg).}
	\label{fig:stokes_p5deg_1}
\end{minipage}

\begin{minipage}{\textwidth}
	\centering
	\def\svgwidth{\textwidth}
\begingroup%
  \makeatletter%
  \providecommand\color[2][]{%
    \errmessage{(Inkscape) Color is used for the text in Inkscape, but the package 'color.sty' is not loaded}%
    \renewcommand\color[2][]{}%
  }%
  \providecommand\transparent[1]{%
    \errmessage{(Inkscape) Transparency is used (non-zero) for the text in Inkscape, but the package 'transparent.sty' is not loaded}%
    \renewcommand\transparent[1]{}%
  }%
  \providecommand\rotatebox[2]{#2}%
  \newcommand*\fsize{\dimexpr\f@size pt\relax}%
  \newcommand*\lineheight[1]{\fontsize{\fsize}{#1\fsize}\selectfont}%
  \ifx\svgwidth\undefined%
    \setlength{\unitlength}{430.49195778bp}%
    \ifx\svgscale\undefined%
      \relax%
    \else%
      \setlength{\unitlength}{\unitlength * \real{\svgscale}}%
    \fi%
  \else%
    \setlength{\unitlength}{\svgwidth}%
  \fi%
  \global\let\svgwidth\undefined%
  \global\let\svgscale\undefined%
  \makeatother%
  \begin{picture}(1,0.87247485)%
    \lineheight{1}%
    \setlength\tabcolsep{0pt}%
    \put(0,0){\includegraphics[width=\unitlength,page=1]{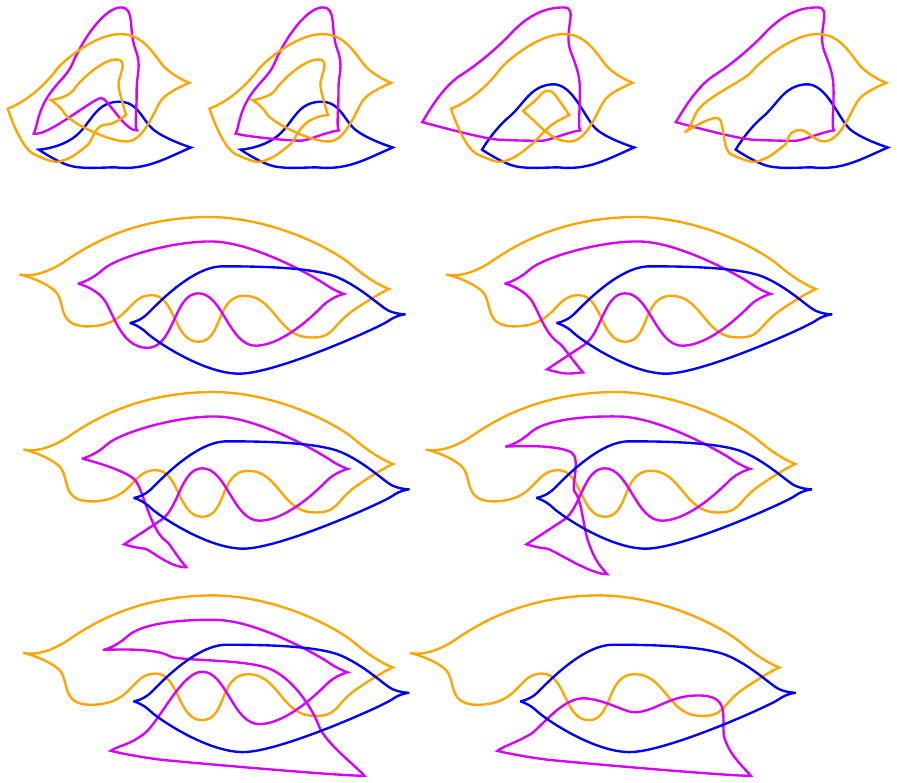}}%
  \end{picture}%
\endgroup%

	\captionof{figure}{Observe that the final diagram was encountered in the second to last step of \zcref{fig:stokes_p3d6}, proving \zcref{thm:stokes_is_generic_monodromy} for PV(deg).}
	\label{fig:stokes_p5deg_2}
\end{minipage}

\newpage
\printbibliography
\end{document}